\documentclass[11pt,reqno]{amsart}

\usepackage[T1]{fontenc}
\usepackage{newtxtext}
\usepackage{newtxmath}
\usepackage{microtype}

\usepackage[
left=1in,
right=1in,
top=1in,
bottom=1in
]{geometry}
\usepackage{amsmath,amssymb,amsthm}
\usepackage{thmtools}
\usepackage{enumitem}
\usepackage{amscd}
\usepackage{tikz-cd}
\usepackage{booktabs}
\usepackage{placeins}

\usepackage{hyperref}

\hypersetup{
	hidelinks,
	pdftitle={Coincidence Correspondences and Nonlinear Root Geometry},
	pdfauthor={Alok Shukla}
}

\usepackage{cleveref}

\numberwithin{equation}{section}

\theoremstyle{plain}
\newtheorem{theorem}{Theorem}[section]

\newtheorem{proposition}[theorem]{Proposition}
\newtheorem{lemma}[theorem]{Lemma}
\newtheorem{corollary}[theorem]{Corollary}
\theoremstyle{definition}
\newtheorem{definition}[theorem]{Definition}
\newtheorem{example}[theorem]{Example}
\theoremstyle{remark}
\newtheorem{remark}[theorem]{Remark}



\crefname{section}{Sect.}{Sects.}
\Crefname{section}{Sect.}{Sects.}

\crefname{definition}{Defn.}{Defns.}
\Crefname{definition}{Defn.}{Defns.}

\crefname{theorem}{Thm.}{Thms.}
\Crefname{theorem}{Thm.}{Thms.}

\crefname{proposition}{Prop.}{Props.}
\Crefname{proposition}{Prop.}{Props.}

\crefname{lemma}{Lemma}{Lemmas}
\Crefname{lemma}{Lemma}{Lemmas}

\crefname{corollary}{Cor.}{Cors.}
\Crefname{corollary}{Cor.}{Cors.}

\crefname{remark}{Rem.}{Rems.}
\Crefname{remark}{Rem.}{Rems.}

\crefname{example}{Ex.}{Exs.}
\Crefname{example}{Ex.}{Exs.}

\crefname{equation}{Eqn.}{Eqns.}
\Crefname{equation}{Eqn.}{Eqns.}

\crefname{table}{Table}{Tables}
\Crefname{table}{Table}{Tables}

\title[Nonlinear Root Geometry]{Coincidence Correspondences and Nonlinear Root Geometry}

\author[A. Shukla]{%
  {\bfseries Alok Shukla}\\
  {\normalfont\small\itshape School of Arts and Sciences, Ahmedabad University, India}\\
  {\normalfont\small\itshape E-mail: \texttt{alok.shukla@ahduni.edu.in}}%
}

\makeatletter
\def\@setauthors{%
  \begingroup
  \def\thanks{\protect\thanks@warning}%
  \trivlist
  \centering\normalsize \@topsep30\p@\relax
  \advance\@topsep by -\baselineskip
  \item\relax
  \author@andify\authors
  \def\\{\protect\linebreak}%
  \authors
  \ifx\@empty\contribs
  \else
    ,\penalty-3 \space \@setcontribs
    \@closetoccontribs
  \fi
  \endtrivlist
  \endgroup
}
\makeatother

\date{}
\subjclass[2020]{Primary 14A10, 20F55; Secondary 14E20, 14F35, 53C05}
\keywords{Finite morphism, fiber product, coincidence correspondence, reflection group, Coxeter geometry, Cartan matrix, Galois cover, transport symmetry, root arrangement, nonlinear root geometry}

\begin{document}

\begin{abstract}
We show that, under suitable hypotheses, finite morphisms between smooth complex algebraic varieties give rise to local root-theoretic structures. Starting from the self-fiber product $X\times_Y X$ of a finite morphism $Q:X\to Y$, selected local branches of non-diagonal irreducible components give rise to local symmetry operators and normalized Cartan-type invariants. Thus root-theoretic data can be extracted from algebraic correspondences without assuming a pre-existing reflection group.

A central part of the theory is a rank-two geometry associated with pairs of such branches. We establish a rank-two reduction theorem, derive explicit trace and determinant formulas for the corresponding operators, and obtain an elliptic–parabolic–hyperbolic classification, distinguishing scalar and non-scalar cases at repeated eigenvalues. These results yield transport analogues of Cartan matrices and Coxeter transformations, together with associated Dynkin and Coxeter graphs. Under appropriate finite-order and primitive-spectrum hypotheses, the classical Coxeter labels are recovered.

We further prove rigidity results showing that, at each fixed point, the local symmetries preserving a single finite morphism form a finite group and are simultaneously analytically linearizable. These results impose strong restrictions on the root-theoretic structures arising from one morphism. The framework recovers classical Weyl and finite complex-reflection geometries as special cases.
\end{abstract}

\maketitle

\setcounter{tocdepth}{1}
	\tableofcontents
	
	\thispagestyle{empty} 
	
	\newpage

	\section{Introduction}
\label{sec:Introduction}

Root systems and reflection groups are among the fundamental organizing structures in modern mathematics. They encode geometric and algebraic symmetries arising throughout Lie theory, algebraic geometry, and representation theory, and they also play important roles in mathematical physics.

The classical theory of real reflection groups and root systems forms one of the foundations of Lie theory  \cite{Coxeter1934, MR376938, MR323842}. In the complex setting, the fundamental work of Shephard and Todd, together with Chevalley, established that finite groups generated by complex reflections are precisely those whose invariant rings are polynomial algebras \cite{ShephardTodd1954, Chevalley1955, humphreys1990reflection}. More recently, several attempts have been made to generalize classical reflection structures. The theory of Weyl groupoids, initiated by Heckenberger \cite{Heckenberger2006} and further developed by Cuntz and Heckenberger \cite{cuntz2012finite}, replaces a single Cartan matrix by a family of Cartan data varying over a groupoid. 

Many established extensions of classical reflection theory are formulated in terms of linear group actions, hyperplane arrangements, or combinatorial Coxeter data \cite{OrlikTerao1992}. A complementary question is whether root-theoretic structures can be extracted directly from nonlinear finite morphisms. This is the question pursued here. The perspective developed in this work is motivated in part by this observation: instead of starting from a reflection group and its action, we begin with a finite morphism itself and investigate the geometric structures naturally encoded by its fibers and their mutual interactions.

While generalized root systems have also appeared in algebraic geometry, notably in the work of Looijenga \cite{Looijenga1981}, the underlying objects are fundamentally different from those considered here. Looijenga's roots are classes in a Picard lattice associated with rational surfaces, whereas the coincidence components introduced in the present work are irreducible correspondence components of a fiber product. In Looijenga's theory the Cartan data arise from lattice-theoretic intersection forms, while in the present framework the transport Cartan matrices and transport reflection operators are derived from the geometry of finite morphisms and their associated coincidence correspondences. Although both settings exhibit wall structures and reflection-type phenomena, the root data developed here are reconstructed directly from the geometry of the correspondence, without presupposing a pre-existing Picard lattice.

A further point of comparison is the McKay correspondence, which gives a bijection between the nontrivial irreducible representations of a finite subgroup of $SL_2(\mathbb{C})$ and the exceptional divisors appearing in the minimal resolution of the corresponding Kleinian singularity \cite{mckay1980graphs,mckay1981cartan}. The resulting McKay graph is an ADE Dynkin diagram and therefore determines the associated ADE root system.
By contrast, the transport Cartan matrices developed in the present work are reconstructed directly from the geometry of a finite morphism and the interactions of its associated root walls. In this sense, the present framework bypasses singularity resolution and is not restricted to quotient spaces arising from finite group actions, thereby extending root-theoretic constructions to a broader class of finite correspondences.

The present paper studies a different source of root-geometric data. We start with a finite morphism $Q:X\to Y$ between smooth algebraic varieties and its self-fiber product
\[
\mathcal C_Q=X\times_YX
=\{(x_1,x_2)\in X\times X:Q(x_1)=Q(x_2)\}.
\]
The correspondence records pairs of points in the same fiber. Over a connected open set $U\subseteq Y$ on which $Q$ is finite étale, the geometric fiber of $\mathcal C_Q|_U\to U$ is $\Omega\times\Omega$, where $\Omega=Q^{-1}(u)$ for a chosen $u\in U$. The fundamental group acts diagonally on this product, and the connected components of the finite étale cover correspond to its orbits, as in the Galois-category description of finite étale covers \cite{MR354651}. In the Galois case these components are graphs of deck transformations; for a general finite étale cover they need not be global graphs. Our additional input is the local analytic branch data at selected points of non-diagonal components, especially at their intersections with the diagonal.

For selected regular branches, the diagonal intersection determines a root wall, while a chosen branch through a point of that wall determines a local transport symmetry germ and its differential. Under the smooth-wall, rank-one, splitting, and (for pairs) transversality hypotheses stated below, these data produce transport operators, conormal root data, and normalized pairwise Cartan invariants. The construction is branch-indexed and local: it does not in general produce a global reflection group or a Weyl groupoid. Its guiding chain is
\[
Q\;\Longrightarrow\;\mathcal C_Q\;\Longrightarrow\;
\text{coincidence components}\;\Longrightarrow\;
\text{local transport geometry}\;\Longrightarrow\;
\text{root-theoretic data}.
\]

The rank-two theory studies the interaction of pairs of selected branches. The reduced pairwise operator has characteristic polynomial determined by $(\zeta_i,\zeta_j,\Theta_{ij})$, so that, once the transverse eigenvalues are fixed, the normalized invariant $\Theta_{ij}$ controls the remaining pairwise dependence. In the involutive case and for real $\Theta_{ij}$, this gives elliptic and hyperbolic spectral regimes together with a repeated-root boundary. At that boundary the invariant determines the repeated eigenvalue but does not, by itself, distinguish a scalar operator from a nontrivial Jordan operator; the latter is parabolic in the two-dimensional projective setting. If the reduced product has finite order, it is semisimple and the nontrivial Jordan case is excluded. More strongly, for every transverse rank-one pair preserving one common finite morphism, averaging a positive-definite Hermitian form over the finite local symmetry group gives $0\leq\Theta_{ij}<1$, without an involutivity or reality assumption. Thus the $+1$ boundary is excluded for all such intrinsic pairs; it can occur for mixed-chart pairs whose branches need not preserve a common morphism.

For a single finite morphism, the local symmetry result is stronger than finite order of the selected generators: every holomorphic self-germ fixing a point and preserving the morphism is automatically a finite-order biholomorphism. The resulting finite local group is simultaneously analytically linearizable near that point, in every dimension. This is a local analytic statement; it does not assert global algebraic linearization of the morphism or its source. To describe variable pairwise geometry beyond the common-morphism setting, we introduce weak transport atlases: chart-labelled collections of finite morphisms and selected coincidence components over an open cover of a common ambient variety, without transition maps, overlap identifications, or cocycle conditions. Different charts may represent two meeting walls by different morphisms, allowing a mixed intrinsic Cartan invariant to vary. Each selected germ is finite order relative to its own chart morphism, but products of germs from different charts need not be finite order. A separate connection-induced extension, requiring a flat holomorphic connection and wall-point maps, produces choice-dependent two-point Cartan fields; this supplementary construction is recorded in Appendix~\ref{sec:Connection-Induced-Cartan-Extensions}.

The quotient examples recover familiar reflection geometry. For a finite Weyl group quotient, the coincidence correspondence recovers the classical reflecting arrangement and the associated Cartan--Coxeter data. For a finite complex reflection group quotient, it recovers the Shephard--Todd reflection geometry and the acting group. In a general finite morphism, however, the local transport construction should not be confused with the existence of a global group action or a classical root system.

Table~\ref{tab:terminology-guide} summarizes the principal objects introduced in the paper, and Table~\ref{tab:classical-dictionary} lists their nearest analogues in classical Coxeter--Weyl theory. These comparisons indicate analogies of role and structure, not literal identifications.

\begin{table}[ht]
		\centering
		\caption{Guide to the principal objects introduced in this paper.}
		\label{tab:terminology-guide}
		\renewcommand{\arraystretch}{1.15}
		\begin{tabular}{p{4.5cm} p{6.5cm} p{2.5cm}}
			\toprule
			\textbf{Object} &
			\textbf{Informal Description} &
			\textbf{Reference} \\
			\midrule
			
			Coincidence correspondence $\mathcal C_Q=X\times_YX$
			&
			Pairs of points lying in the same fiber of $Q$
			&
			\Cref{subsec:Notation}
			\\
			
			Coincidence component
			&
			A non-diagonal irreducible component of $\mathcal C_Q$
			&
			\Cref{subsec:Notation}
			\\
			
			Admissible coincidence component
			&
			A generically smooth non-diagonal irreducible component of $\mathcal C_Q$
			&
			\Cref{def:admissible-coincidence-component}
			\\
			
			Root wall
			&
			Fixed-point locus associated with a coincidence component
			&
			\Cref{def:root-wall}
			\\
			
			Primitive transport symmetry
			&
			Local symmetry germ associated with a selected regular branch
			of a coincidence component
			&
			\Cref{def:primitive-transport-symmetry}
			\\
			
			Regular transport root point
			&
			Point where the local transport geometry is well-defined
			&
			\Cref{def:Regular-Transport-Root-Point}
			\\
			
			Transport root
			&
			Root datum attached to a transport symmetry
			&
			\Cref{def:Transport-Root}
			\\
			
			Transport Cartan coefficients
			&
			Cartan coefficients attached to transport roots
			&
			\Cref{def:Transport-Cartan-Coefficients}
			\\
			
			Transport Cartan matrix
			&
			Pairing matrix determined by transport roots
			&
			\Cref{def:Rank--Two-Transport-Cartan-Matrix}
			\\
			
			Transport Dynkin graph
			&
			Graph encoding transport Cartan data
			&
			\Cref{def:Transport-Dynkin-Graph}
			\\
			
			Reduced pairwise transport operator
			&
			Operator induced by $M_{ij}=R_iR_j$ on the canonical
			two-dimensional quotient
			&
			\Cref{thm:Rank-Two-Reduction-Theorem}
			\\
			
			\bottomrule
		\end{tabular}
	\end{table}

\begin{table}[ht]
		\centering
		\caption{Comparison with classical Coxeter--Weyl theory. The objects in the right column should be understood as analogues, not direct generalizations, of the classical notions.}
		\label{tab:classical-dictionary}
		\renewcommand{\arraystretch}{1.15}
		\begin{tabular}{p{5cm} p{5cm}}
			\toprule
			\textbf{Classical Coxeter Theory}
			&
			\textbf{Coincidence-Root Theory}
			\\
			\midrule
			
			Reflection hyperplane
			&
			Root wall
			\\
			
			Reflection
			&
			Primitive transport symmetry
			\\
			
			Root
			&
			Transport root
			\\
			
			Cartan matrix
			&
			Transport Cartan matrix
			\\
			
			Rank-two Coxeter element (product of the two simple reflections)
			&
			Reduced pairwise transport operator
			\\
			
			Coxeter graph
			&
			Transport Coxeter graph
			\\
			
			Dynkin diagram
			&
			Transport Dynkin graph
			\\
			
			Reflection arrangement
			&
			Coincidence geometry
			\\
			
			Weyl / reflection group
			&
			Transport symmetries
			\\
			
			\bottomrule
		\end{tabular}
	\end{table}
\FloatBarrier

The principal contributions are as follows:
\begin{enumerate}[label=(\roman*)]
    \item \textbf{Coincidence-root systems.}
    We construct root walls and local transport symmetry germs from selected regular branches of non-diagonal components of $X\times_YX$, and derive the associated transport operators and root cotangent lines (\Cref{sec:Coincidence-correspondences,sec:Transport-Symmetries,sec:Transport-Root-Theory}).
    \item \textbf{Rank-two reduction and Cartan geometry.}
    We establish the rank-two reduction theorem, derive trace and determinant formulas, and obtain the spectral classification and transport Cartan, Coxeter, and Dynkin data (\Cref{sec:Transport-Root-Theory,sec:Rank-Two-Transport-Geometry}).
    \item \textbf{Local and global rigidity for a single morphism.}
    We prove that local self-germs preserving one finite morphism form a finite group and are simultaneously analytically linearizable, and establish finite-monodromy restrictions on intrinsic transport invariants (\Cref{sec:Transport-Symmetries,sec:finite-monodromy-rigidity}); further projective-rigidity and tangency results, under additional hypotheses, are recorded in Appendix~\ref{sec:Geometric-Rigidity-and-Degenerations}.
    \item \textbf{Recovery of classical reflection geometry.}
    We show that quotient maps associated with finite Weyl groups and finite complex reflection groups recover their reflection geometry and the acting groups (\Cref{sec:Finite-Galois-Covers}).
    \item \textbf{Weak transport atlases.}
    We introduce unglued, chart-labelled collections of finite morphisms in which mixed Cartan invariants can vary across selected walls represented by different morphisms. The rank--two characteristic-polynomial formula extends to such mixed-chart pairs, without implying a common finite-morphism symmetry group (\Cref{sec:transport-atlases-and-Local-Finite-Covers}).
    \item \textbf{Simultaneous-linearization obstruction.}
    We prove that a nonconstant mixed invariant obstructs simultaneous linearization to the common-vector-space linear reflection model defined in Section~\ref{sec:Beyond-Reflection-Groups} (\Cref{sec:Beyond-Reflection-Groups}).
\end{enumerate}

The construction is distinct both from the classical theory of ramification and from the established theories of reflection groups, Weyl groupoids, hyperplane arrangements, and lattice root systems. The coincidence correspondence contains information about the interaction of distinct points in the same fiber; selected local branches turn this information into differential transport data. All constructions are made at specified regular transport root points. We do not assume that such points form a dense open subset without additional family hypotheses, nor do we impose transversality between a coincidence component and the diagonal at a reflection-type wall: a hypersurface fixed locus arises from a nontransverse intersection. The transversality used in the rank-two reduction is instead the transversality of two selected root walls. Extensions to singular coincidence components and higher-order wall singularities are left for future work.

\section{Coincidence Correspondences}
	\label{sec:Coincidence-correspondences}
	
	\subsection{Notation and the Fiber Product Construction}
	\label{subsec:Notation}
	
	Unless explicitly stated otherwise, the algebraic varieties in this paper are defined over an algebraically closed field $k$ of characteristic zero. Whenever a construction uses an associated complex analytic space, holomorphic or biholomorphic germs, or variation along a complex-analytic locus, we specialize to $k=\mathbb C$. Theorems explicitly stated over an arbitrary algebraically closed field of characteristic zero are purely algebraic and do not rely on such analytification. In particular, a variety over an arbitrary field $k$ is not implicitly regarded as a complex analytic space. The analytic constructions below are formulated in the complex analytic category; no general analogue in the category of smooth manifolds is asserted.

	Given a finite morphism
	$Q:X\to Y$,
	the fundamental geometric object underlying our construction is the
	fiber product
	\begin{equation}
		\mathcal C_Q:=X\times_YX,
	\end{equation}
	which we call the \emph{coincidence correspondence} associated with $Q$.
	Recall that the fiber product is defined by the cartesian diagram
	\[
	\begin{tikzcd}
		X\times_Y X
		\arrow[r,"p_2"]
		\arrow[d,"p_1"']
		&
		X
		\arrow[d,"Q"]
		\\
		X
		\arrow[r,"Q"']
		&
		Y
	\end{tikzcd}.
	\]
	Equivalently, $X\times_Y X = \{(x_1,x_2)\in X\times X: Q(x_1)=Q(x_2)\}$. The natural projections
	$p_1,p_2:\mathcal C_Q\to X$
	are induced by the coordinate projections on \(X\times X\) and satisfy
	$Q\circ p_1 = Q\circ p_2$.
	Together they realize \(\mathcal C_Q\) as the fiber product of \(X\) with itself over \(Y\) via the universal property of the fiber product.
	
	The coincidence correspondence captures the geometry of pairs of points lying in the same fiber of $Q$. If $Q^{-1}(y) = \{x_1,\dots,x_d\}$, then the fiber of $\mathcal C_Q\to Y$ over $y$ consists of all ordered pairs $(x_i,x_j)$, for $1\le i,j\le d$. Thus $\mathcal C_Q$ may be viewed as the space of relations among sheets of the finite covering defined by $Q$.
	
	A distinguished component of $\mathcal C_Q$ is the diagonal $\Delta = \{(x,x):x\in X\} \subseteq \mathcal C_Q$. Indeed, $Q(x)=Q(x)$ for every $x\in X$, so the diagonal is always contained in the coincidence correspondence. The diagonal records the trivial coincidence relation. The remaining components encode nontrivial relations among points in the same fiber and will form the basic geometric objects of the theory developed below.
	The coincidence correspondence is the fundamental object of the theory.
	All subsequent constructions are derived from its irreducible
	decomposition.
	
	\subsection{The Coincidence Correspondence}
	\label{subsec:Coincidence-Correspondence}

	We now study the basic geometric properties of the coincidence correspondence $\mathcal C_Q = X\times_Y X$. Since $\mathcal C_Q$ is defined as a fiber product, many of its properties follow directly from standard facts in algebraic geometry. The following proposition records several fundamental features of $\mathcal C_Q$, including its dimension, its distinguished diagonal component, and the structure of its fibers over the \'etale  locus of the morphism  $Q$.
	
	\begin{proposition}
		\label{prop:coincidence-correspondence}
		
		Let $Q:X\to Y$
		be a finite dominant morphism between irreducible varieties and let
		$\mathcal C_Q=X\times_YX$
		denote its coincidence correspondence. Let $U\subseteq Y$ be the maximal open subset over which $Q$ is \'etale of degree $d$. Then
		\begin{enumerate} [label = (\roman*)]
			\item $\dim(\mathcal C_Q)=\dim(X)$.
			\item The diagonal $\Delta = \{(x,x):x\in X\}$ is a closed irreducible subvariety of $\mathcal C_Q$.
			\item 
			Let
			$Q_{\mathcal C}:=Q\circ p_1=Q\circ p_2:\mathcal C_Q\to Y$.
			For every $y\in U$, let
			$(\mathcal C_Q)_y:=Q_{\mathcal C}^{-1}(y).$
			Then $|(\mathcal C_Q)_y|=d^2$.
		\end{enumerate}
	\end{proposition}
	
	\begin{proof}
		
		Consider the projection
		$p_1:\mathcal C_Q=X\times_YX\longrightarrow X$.
		Since
		$(x,x)\in \mathcal C_Q$
		for every $x\in X$, the projection
		$p_1:\mathcal C_Q\to X$
		is surjective. Moreover, as finiteness is preserved under base change, $p_1$ is also finite.  
		Since $p_1$ is finite and surjective,
		$\dim(\mathcal C_Q)
		=
		\dim(X)$.
		The coincidence correspondence therefore has the same dimension as the source variety.
		
		Next we prove (ii). The diagonal $\Delta = \{(x,x) : x \in X\}$ is a closed subvariety of $X \times X$. For every $x \in X$, we have $Q(x) = Q(x)$, so
		\[
		\Delta \subseteq X \times_Y X = \mathcal{C}_Q.
		\]
		Hence, $\Delta$ is a closed subvariety of $\mathcal{C}_Q$.
		Moreover, the map $X \longrightarrow \Delta$ defined by $x \longmapsto (x,x)$ is an isomorphism. Since $X$ is irreducible, $\Delta$ is irreducible.
		
		Finally we prove (iii). Over $U$, the map $Q$ is a finite \'etale of degree $d$. For $y\in U$, the fiber $Q^{-1}(y) = \{x_1,\ldots,x_d\}$ consists of $d$ distinct points. Hence $(\mathcal C_Q)_y = Q^{-1}(y)\times Q^{-1}(y) = \{(x_i,x_j)\}_{i,j}$. The claim follows.
	\end{proof}
	
	\begin{remark}
		Part (iii) shows that over the \'etale locus, the coincidence correspondence
		records all ordered pairs of points lying in a common fiber of $Q$.
		This observation motivates the interpretation of $\mathcal C_Q$ as a space
		encoding the relations among the sheets of the finite covering defined by
		the morphism $Q$.
	\end{remark}

	\begin{definition}[Coincidence Component Arrangement]
		\label{def:coincidence-component-arrangement}
		Let $\mathcal C_Q = \Delta \cup D_1 \cup \cdots \cup D_r$ be the decomposition of $\mathcal C_Q$ into irreducible components, where $D_i\neq \Delta$. The collection $\Phi_Q = \{D_1,\ldots,D_r\}$ is called the \emph{coincidence component arrangement} associated to $Q$.
	\end{definition}
	
	\begin{example}
		Let
		\[
		Q:\mathbb A^1\setminus\{0\}\longrightarrow \mathbb A^1,
		\qquad
		Q(x)=x+x^{-1}.
		\]
		Then 
		\[
		Q(x)-Q(y) = \frac{(x-y)(xy-1)}{xy}.
		\]
		Hence $\mathcal C_Q = \Delta \cup D$, where $D=\{xy=1\}$. The non-diagonal component records the relation $x \mapsto x^{-1}$. 
	\end{example}
	
	\begin{example}
		Consider $Q(x)=x^2+x$. Then $Q(x)-Q(y) = (x-y)(x+y+1)$. Therefore, $\mathcal C_Q = \Delta \cup D$, where $D=\{x+y+1=0\}$. The component $D$  records the relation $x \mapsto -x-1$.
	\end{example}

	The terminology \emph{coincidence component arrangement} is motivated by the fact that, in the special case
	where $Q$ arises as the quotient map of a finite reflection group,
	the irreducible components $D_i$ recover the classical reflection
	arrangement. Indeed, if
	$\pi:V\to V/G$
	is the quotient map associated to a finite reflection group, then
	\[ V\times_{V/G}V
	=
	\bigcup_{g\in G}\Gamma_g,
	\qquad
	\Gamma_g=\{(v,g(v)):v\in V\}. \]
	Thus each group element determines an irreducible component of the coincidence correspondence.
	In particular, reflections correspond to distinguished components
	whose intersections with the diagonal recover the classical
	reflecting hyperplanes.
	These connections will be established in
	\Cref{sec:Finite-Galois-Covers}.
	More generally, the irreducible components of
	$\mathcal C_Q=X\times_YX$
	encode the multivalued inverse structure of the finite morphism $Q$.
	The collection $\Phi_Q$ provides the fundamental geometric data
	from which the constructions developed in the remainder of the paper
	are derived.
	
	\begin{definition}[Admissible 
		Coincidence Component]
		\label{def:admissible-coincidence-component}
		Let
		$D\in\Phi_Q$.
		We say that $D$ is an admissible coincidence component if
		\begin{enumerate}[label = (\roman*)]
			\item $D$ is generically smooth;
			\item the restrictions
			$p_1|_D:D\to X$
			and $p_2|_D:D\to X$
			are generically \'etale.
		\end{enumerate}
	\end{definition}
	
	Henceforth, all coincidence components are assumed admissible unless explicitly stated otherwise. Purely algebraic statements retain the algebraically closed characteristic-zero base field specified in their hypotheses. All statements involving analytification, holomorphic or biholomorphic germs, or complex-analytic loci are understood over $\mathbb C$. These conventions remain in force unless explicitly indicated otherwise.

	\begin{remark}
		For a general finite morphism
		$Q:X\to Y$,
		the coincidence correspondence
		$\mathcal C_Q=X\times_Y X$
		may contain singular or non-reduced components, particularly along the ramification locus.
		The admissibility hypothesis guarantees a dense open subset of each
		admissible component on which the component is smooth and both
		projections are \'etale. It does not exclude singularities or
		degenerations at special points, nor does it rule out scheme-theoretic
		non-reduced structure along special loci. The local transport
		constructions are carried out on selected regular branches at points
		where the hypotheses of the relevant local theorems hold. Unless
		scheme structure is explicitly relevant, coincidence components are
		viewed as their reduced irreducible supports. The study of singular
		correspondences and scheme-theoretic multiplicities beyond these
		selected regular branches lies beyond the scope of the present work.
	\end{remark}

	\section{Transport Symmetries}
	\label{sec:Transport-Symmetries}

	\paragraph{Analytic notation.}
	For a complex algebraic variety $Z$, we write $Z^{\mathrm{an}}$ for
	its associated complex analytic space. If $f:Z\to W$ is an algebraic
	morphism, then $f^{\mathrm{an}}:Z^{\mathrm{an}}\to W^{\mathrm{an}}$
	denotes the induced holomorphic map. In particular, writing
	$C_Q:=X\times_YX$, we use the canonical identification
	\(C_Q^{\mathrm{an}} \simeq X^{\mathrm{an}}\times_{Y^{\mathrm{an}}}X^{\mathrm{an}}.\)
	The symbols $p_1^{\mathrm{an}}$ and $p_2^{\mathrm{an}}$ denote the
	analytic projections from this fibre product, and $D^{\mathrm{an}}$
	denotes the analytification of a coincidence component
	$D\subseteq C_Q$. The coordinate interchange on the analytic fibre
	product is denoted by $\iota^{\mathrm{an}}$.
	
	\begin{remark}
		For the local transport-germ theory and the differential rank-two
		classification in Sections~\ref{sec:Transport-Symmetries}--
		\ref{sec:Rank-Two-Transport-Geometry}, the varieties and morphisms are
		complex algebraic and the associated local objects are taken in the
		complex analytic category. Algebraic lemmas stated over a general field
		of characteristic zero remain valid over that field, but their analytic
		interpretations are used only after specializing to $\mathbb C$.
	\end{remark}

	\subsection{Root Walls}\label{subsec:root-walls}
	
	\begin{definition}
		\label{def:root-wall}
		Let $D\in\Phi_Q$ be an admissible coincidence component. The associated \emph{root wall} is 
		\begin{equation}\label{eq:def-root-wall}
			H_D = D\cap\Delta,
		\end{equation}
		where $\Delta\subset X\times X$ denotes the diagonal.
	\end{definition}
	
	\begin{remark}
		\label{remark:root_wall}
		Since the diagonal
		$\Delta$
		is canonically isomorphic to \(X\), we shall often regard
		$H_D=D\cap\Delta$
		as a subvariety of \(X\) via the identification
		$\Delta \cong X$.
		Under this identification, the root wall
		$H_D=D\cap\Delta$
		may be regarded either as a subvariety of $X\times X$ or as a
		subvariety of $X$. Accordingly, expressions such as
		$x\in H_D$
		mean that
		$(x,x)\in H_D$.
	\end{remark}

	\begin{example}
		Consider
		\[
		Q:\mathbb A^1\setminus\{0\}\longrightarrow \mathbb A^1,
		\qquad
		Q(x)=x+x^{-1}.
		\]
		The coincidence correspondence is $\mathcal C_Q = \Delta \cup D$, where $D=\{xy=1\}$. The root wall is $H_D = D\cap\Delta = \{x=y,\ x^2=1\}$. Hence $H_D = \{(1,1),(-1,-1)\}$. These are precisely the fixed points of the involution $x\longmapsto x^{-1}$.
	\end{example}
	
	\begin{example}
		For $Q(x)=x^2+x$, the non-diagonal component is $D=\{x+y+1=0\}$. The root wall is $H_D = D\cap\Delta = \{x=y,\ 2x+1=0\}$. Thus $H_D = \left\{ \left(-\frac12,-\frac12\right) \right\}$. This is the fixed point of the involution $x\longmapsto -x-1$.
	\end{example}

	\begin{proposition}[Wall Regularity]
		\label{prop:wall-regularity}
		Let $Q:X\to Y$ be a finite morphism of complex varieties with
		smooth source $X$, and let $D\in\Phi_Q$ be an admissible
		coincidence component. Let $(x,x)\in H_D$ be a smooth point of $D$.
		Assume that $D$ and $\Delta$ intersect transversely at $(x,x)$.
		Then the scheme-theoretic intersection $H_D=D\cap\Delta$ is
		smooth and reduced of dimension zero at $(x,x)$. In particular,
		$T_{(x,x)}H_D=\{0\}$.
	\end{proposition}

	\begin{proof}
		Since $X$ is smooth, the diagonal
		$\Delta\subset X\times X$
		is a smooth subvariety.
		By hypothesis, $D$ is smooth at $(x,x)$.
		Since $D$ and $\Delta$ are smooth at $(x,x)$, their tangent spaces are well-defined. By the transversality assumption,
		\[
		T_{(x,x)}D + T_{(x,x)}\Delta = T_{(x,x)}(X\times X).
		\]
		The standard transversality criterion for smooth subvarieties therefore implies that the scheme-theoretic intersection $H_D=D\cap\Delta$ is smooth and reduced at $(x,x)$. Since $D$ and $\Delta$ both have dimension $n=\dim X$, transversality gives a zero-dimensional intersection there. Thus
		\[
		T_{(x,x)}H_D=T_{(x,x)}D \cap T_{(x,x)}\Delta=\{0\}.
		\]
	\end{proof}

	\begin{remark}
		The transversality assumption in \Cref{prop:wall-regularity} is a sufficient condition for smoothness of the root intersection \(H_D=D\cap\Delta\). Since \(\dim D=\dim\Delta=\dim X=n\), such an intersection is zero-dimensional. In the reflection-type setting considered subsequently, the fixed-point locus may instead be a hypersurface. In that case, \(D\) and \(\Delta\) intersect nontransversely, and smoothness of the root wall is treated through the fixed-point geometry of the associated local transport symmetry. The transversality assumption in the Rank-Two Reduction Theorem (\Cref{thm:Rank-Two-Reduction-Theorem}) concerns the intersection of two root walls and is independent of the condition considered here.
	\end{remark}

	\subsection{Local Transport Symmetries}
	\label{subsec:Local-Transport-Symmetries}
	
	Let $D\in\Phi_Q$ be an admissible coincidence component.
	Since the projections
	$p_1,p_2:D\to X$
	are generically \'etale, the component $D$ is locally the graph
	of a correspondence between sheets of the covering defined by $Q$.
	The following theorem makes this precise.

	\begin{theorem}[Local Transport Theorem]
		\label{thm:Local-Transport-Theorem}
		Let $Q:X\to Y$ be a finite morphism of smooth complex varieties, and
		let $D\in\Phi_Q$ be an admissible coincidence component. Choose a point
		$(x_0,y_0)$ in the regular locus of $D$ such that
		\(d(p_1^{\mathrm{an}})_{(x_0,y_0)}: T_{(x_0,y_0)}D^{\mathrm{an}}\longrightarrow T_{x_0}X^{\mathrm{an}}\)
		is an isomorphism. Then, in the complex analytic category, there are
		sufficiently small analytic neighborhoods $W$ of $(x_0,y_0)$ in
		$D^{\mathrm{an}}$ and $U$ of $x_0$ in $X^{\mathrm{an}}$ such that
		\(p_1^{\mathrm{an}}|_W:W\longrightarrow U\)
		is biholomorphic. The chosen inverse branch determines a holomorphic map
		$\tau_{D,(x_0,y_0),W}:U\to X^{\mathrm{an}}$, defined by
		$\tau_{D,(x_0,y_0),W}=p_2^{\mathrm{an}}|_W\circ
		(p_1^{\mathrm{an}}|_W)^{-1}$. Its graph is contained in
		$D^{\mathrm{an}}$, and it satisfies
		\(Q^{\mathrm{an}}\circ\tau_{D,(x_0,y_0),W} =Q^{\mathrm{an}}|_U.\)
		This map is unique among holomorphic maps whose graph passes through
		$(x_0,y_0)$ and is contained in the selected neighborhood $W$ after
		shrinking $U$.
		
		In the algebraic category, the corresponding assertion is understood
		after passing to suitable étale neighborhoods of $(x_0,y_0)$ and
		$x_0$. In particular, the theorem asserts uniqueness for a chosen
		point and a chosen local inverse branch; it does not assert that an
		irreducible component determines a single branch over an entire
		Zariski-open subset.
	\end{theorem}
	
	\begin{proof}
		Since $D$ is regular at $(x_0,y_0)$ and $X$ is smooth, both
		$D^{\mathrm{an}}$ and $X^{\mathrm{an}}$ are complex manifolds near the
		chosen points. The differential of $p_1^{\mathrm{an}}$ is an
		isomorphism there. The analytic inverse function theorem therefore
		gives neighborhoods $W$ of $(x_0,y_0)$ and $U$ of $x_0$ for which
		$p_1^{\mathrm{an}}|_W:W\to U$ is biholomorphic. Define
		\(\tau_{D,(x_0,y_0),W} =p_2^{\mathrm{an}}|_W\circ (p_1^{\mathrm{an}}|_W)^{-1}.\)
		By construction, for every $z\in U$,
		\((z,\tau_{D,(x_0,y_0),W}(z)) =(p_1^{\mathrm{an}}|_W)^{-1}(z) \in W\subset D^{\mathrm{an}}.\)
		Since $D^{\mathrm{an}}$ is contained in
		$X^{\mathrm{an}}\times_{Y^{\mathrm{an}}}X^{\mathrm{an}}$, this gives
		\(Q^{\mathrm{an}}(z) =Q^{\mathrm{an}}(\tau_{D,(x_0,y_0),W}(z)).\)
		
		If another holomorphic map has graph through $(x_0,y_0)$ and
		contained in $W$ after shrinking $U$, its graph must be the unique
		section of $p_1^{\mathrm{an}}|_W$ through that point. It therefore
		agrees with the selected inverse branch. The étale-local algebraic
		statement follows from the fact that $p_1|_D$ is étale at
		$(x_0,y_0)$.
	\end{proof}

	\begin{corollary}[Local Graph Property]
		\label{cor:local-graph-property}
		At every regular point $(x_0,y_0)\in D$ at which the differential of
		$p_1^{\mathrm{an}}$ is an isomorphism, a chosen local inverse branch
		of $p_1^{\mathrm{an}}|_{D^{\mathrm{an}}}$ is the graph of a unique
		local transport map. Thus every admissible coincidence component is
		locally a union of such graph branches over the locus where $p_1|_D$
		is étale; a component need not determine one globally single-valued
		transport map.
	\end{corollary}
	
	\begin{proof}
		This is the Local Transport Theorem applied at each selected regular
		point and to each selected local inverse branch.
	\end{proof}

	\begin{definition}[Primitive Transport Symmetry]
		\label{def:primitive-transport-symmetry}
		Let $D\in\Phi_Q$ be an admissible coincidence component. Choose a
		regular point $(x_0,y_0)\in D$ at which the differential of
		$p_1^{\mathrm{an}}$ is an isomorphism, and choose a sufficiently small
		neighborhood $W\subset D^{\mathrm{an}}$ on which
		$p_1^{\mathrm{an}}|_W$ is biholomorphic onto a neighborhood
		$U\subset X^{\mathrm{an}}$ of $x_0$. The map
		\(\tau_{D,(x_0,y_0),W} :=p_2^{\mathrm{an}}|_W\circ (p_1^{\mathrm{an}}|_W)^{-1}:U\to X^{\mathrm{an}}\)
		is called the \emph{primitive transport symmetry} determined by the
		component $D$ and the selected local branch.
	\end{definition}
	
	\begin{remark}
		A primitive transport symmetry is attached to a coincidence component
		together with a selected regular point and a selected local inverse
		branch of the first projection. Different points or different
		branches on the same component may give different transport germs.
		When the selected point lies on the root wall, so that $x_0=y_0=x$,
		the germ fixes $x$ and preserves $Q$. By
		Proposition~\ref{prop:local-finiteness-transport-germs}, it is
		automatically a finite-order biholomorphic germ and hence belongs to
		the local symmetry group introduced below; no separate assumption on
		the second projection is needed for this conclusion.
	\end{remark}
	
	\begin{definition}[Composite Transport Symmetry]
		\label{def:Composite-Transport-Symmetry}
		A \emph{composite transport symmetry} is any local symmetry obtained
		by composing primitive transport symmetries whenever the composition
		is defined.
	\end{definition}
	
	\begin{definition}[Local Symmetry Group]
		\label{def:local-symmetry-group}
		Let $Q:X\to Y$ be a finite morphism of smooth irreducible complex
		varieties and let $x\in X$. We write $G_x(Q)$ for the group of
		biholomorphic germs $\phi:(X^{\mathrm{an}},x)\to
		(X^{\mathrm{an}},x)$ satisfying $Q^{\mathrm{an}}\circ\phi=
		Q^{\mathrm{an}}$ as germs. The group operation is composition.
	\end{definition}
	
	\begin{lemma}[Finiteness of Homomorphisms from a Finite Algebra]
		\label{lem:finite-algebra-homomorphisms}
		Let $A$ be an integral domain and let $B$ be a commutative finite
		$A$-algebra. Then the set of $A$-algebra homomorphisms
		\(\operatorname{Hom}_{A\text{-alg}}(B,A)\)
		is finite.
	\end{lemma}
	
	\begin{proof}
		Choose algebra generators $b_1,\ldots,b_r$ of $B$ over $A$. Since
		$B$ is finite as an $A$-module, every $b_j$ is integral over $A$.
		Choose a monic polynomial $f_j(T)\in A[T]$ with $f_j(b_j)=0$.
		For every $A$-algebra homomorphism $\sigma:B\to A$, one has
		$f_j(\sigma(b_j))=0$. A nonzero polynomial over an integral domain
		has only finitely many roots in that domain. Thus each $\sigma(b_j)$
		has finitely many possible values. Since the $b_j$ generate $B$,
		$\sigma$ is determined by the finite tuple
		$(\sigma(b_1),\ldots,\sigma(b_r))$. Hence there are only finitely many
		such homomorphisms.
	\end{proof}
	
		\begin{proposition}[Finite Order and Automatic Invertibility of Local Transport Self-Germs]
		\label{prop:local-finiteness-transport-germs}
		Let $Q:X\to Y$ be a finite morphism of smooth irreducible complex
		varieties, and fix $x\in X$. Let $S_x(Q)$ be the set of all
		holomorphic germs $\phi:(X^{\mathrm{an}},x)\to
		(X^{\mathrm{an}},x)$ satisfying $Q^{\mathrm{an}}\circ\phi=
		Q^{\mathrm{an}}$, without assuming in advance that $\phi$ is
		biholomorphic. Then $S_x(Q)$ is finite, and every element of
		$S_x(Q)$ has finite order and is biholomorphic. In particular,
		$S_x(Q)=G_x(Q)$ and the local symmetry group is finite.
	\end{proposition}
	
	\begin{proof}
		Let
		\(C_Q^{\mathrm{an}} :=X^{\mathrm{an}}\times_{Y^{\mathrm{an}}}X^{\mathrm{an}}\) and 
		\(p_1^{\mathrm{an}}:C_Q^{\mathrm{an}}\longrightarrow X^{\mathrm{an}}.\)
		The map $p_1^{\mathrm{an}}$ is the analytification of a base change
		of the finite morphism $Q$, and hence is finite. Put
		\(A_x:=\mathcal O_{X^{\mathrm{an}},x}\) and 
		\(B_x:=\bigl((p_1^{\mathrm{an}})_*\mathcal O_{C_Q^{\mathrm{an}}}\bigr)_x.\)
		Then $B_x$ is a finite $A_x$-algebra, while $A_x$ is an integral
		domain because $X$ is smooth and irreducible.
		
		Every $\phi\in S_x(Q)$ defines a section germ
		$s_\phi(z)=(z,\phi(z))$ of $p_1^{\mathrm{an}}$ through
		$(x,x)$. Pullback along this section gives an $A_x$-algebra
		homomorphism $\sigma_\phi:B_x\to A_x$. For a finite analytic
		morphism, sections over a sufficiently small neighborhood correspond
		to homomorphisms from its finite direct-image algebra to the
		structure sheaf; passing to stalks preserves this correspondence.
		Thus distinct section germs give distinct homomorphisms, and
		$S_x(Q)$ injects into
		$\operatorname{Hom}_{A_x\text{-alg}}(B_x,A_x)$.
		This set is finite by
		Lemma~\ref{lem:finite-algebra-homomorphisms}. Hence $S_x(Q)$ is
		finite as a set. It is closed under composition, since the composite
		of two germs fixing $x$ and preserving $Q$ has the same property.
		
		Fix $\phi\in S_x(Q)$. Since $Q$ is finite and the ground field
		has characteristic zero, its differential has rank $\dim X$ on a
		dense open subset of $X$: equivalently, the finite extension of
		function fields over the image is generically separable. Differentiating
		$Q^{\mathrm{an}}\circ\phi=Q^{\mathrm{an}}$ gives
		$dQ_{\phi(z)}\circ d\phi_z=dQ_z$. At points in this dense open
		set, injectivity of $dQ_z$ forces $d\phi_z$ to be injective.
		Thus $\phi$ has full rank on a dense open subset and is dominant
		as a germ. More explicitly, every sufficiently small neighborhood
		of $x$ contains a point where $d\phi$ has full rank, and the inverse
		function theorem shows that the image of a smaller neighborhood
		contains a nonempty open set. A holomorphic function germ whose
		pullback by $\phi$ vanishes therefore vanishes on a nonempty open
		subset, and hence vanishes identically by the identity theorem.
		Thus the pullback on $A_x$ is injective. The same holds for every
		positive iterate $\phi^a$.
		
		Because $S_x(Q)$ is finite and contains every iterate, there are
		integers $0\leq a<b$ such that $\phi^a=\phi^b$. If $a=0$,
		then $\phi^b=\mathrm{id}$. If $a>0$, write
		$\phi^a=\phi^{b-a}\circ\phi^a$. Pullback gives
		$(\phi^a)^*=(\phi^a)^*\circ(\phi^{b-a})^*$.
		Since $(\phi^a)^*$ is injective, this implies
		$(\phi^{b-a})^*=\mathrm{id}$, and hence
		$\phi^{b-a}=\mathrm{id}$ as a germ. Thus every $\phi$
		has finite order. Its inverse is a positive iterate, so it is
		biholomorphic. Consequently $S_x(Q)=G_x(Q)$ and is a finite group.
	\end{proof}
	
\begin{lemma}[Faithfulness of the Differential Representation]
		\label{lem:faithful-local-differential}
		The differential representation $D_x:G_x(Q)\to
		\operatorname{GL}(T_xX^{\mathrm{an}})$, defined by
		$D_x(\phi)=d\phi_x$, is injective.
	\end{lemma}
	
	\begin{proof}
		Differentiation respects composition, so $D_x$ is a group
		homomorphism. It suffices to show that a finite-order holomorphic germ
		$g:(\mathbb C^n,0)\to(\mathbb C^n,0)$ with $dg_0=I$ is the identity.
		Suppose not. In local coordinates, let
		\(g(z)=z+P_d(z)+O(\lVert z\rVert^{d+1}),\)
		where $d\geq2$ and $P_d$ is the first nonzero homogeneous term in the
		Taylor expansion of $g(z)-z$. By induction under composition,
		\(g^{\circ m}(z)=z+mP_d(z)+O(\lVert z\rVert^{d+1})\)
		for every positive integer $m$. Since $G_x(Q)$ is finite, $g$ has
		finite order; choose $m$ with $g^{\circ m}=\operatorname{id}$.
		Then $mP_d=0$, which is impossible over $\mathbb C$ because
		$P_d\ne0$. Thus $g$ is the identity germ and $D_x$ has trivial
		kernel.
	\end{proof}
	
	\begin{theorem}[Simultaneous Local Analytic Linearization]
		\label{thm:simultaneous-local-analytic-linearization}
		Let $Q:X\to Y$ be a finite morphism of smooth irreducible
		complex varieties, fix $x\in X$, and let $G_x(Q)$ be the finite
		local symmetry group above. Write
		$\rho(g)=dg_x\in\operatorname{GL}(T_xX)$. There is a
		biholomorphic coordinate germ
		$\Psi:(X^{\mathrm{an}},x)\to(T_xX,0)$ with
		$d\Psi_x=I$ such that
		\[
		\Psi\circ g=\rho(g)\circ\Psi
		\qquad\text{for every }g\in G_x(Q).
		\]
		In particular, all selected root self-germs fixing $x$ are
		simultaneously conjugated to their linear differentials, and their
		fixed hypersurface germs become the corresponding linear fixed
		hyperplanes. If the selected generators have rank-one deviation
		from the identity, their differentials are complex pseudoreflections;
		the subgroup they generate is therefore a finite complex reflection
		group in its faithful tangent representation.
	\end{theorem}
	
	\begin{proof}
		Choose local coordinates identifying $(X^{\mathrm{an}},x)$ with
		$(T_xX,0)$ and with differential the identity. In these coordinates
		write each $g\in G_x(Q)$ as a holomorphic germ and put
		$\rho(g)=dg_0$. Define the averaging map
		\[
		\Psi(z)=\frac{1}{|G_x(Q)|}
		\sum_{g\in G_x(Q)}\rho(g)^{-1}g(z).
		\]
		Since $\rho$ is a representation,
		reindexing the sum gives
		$\Psi\circ h=\rho(h)\circ\Psi$ for every
		$h\in G_x(Q)$. Moreover, each summand has derivative
		$\rho(g)^{-1}\rho(g)=I$ at the origin, so
		$d\Psi_0=I$. The holomorphic inverse function theorem shows
		that $\Psi$ is a biholomorphic germ, proving simultaneous
		linearization. This is the standard finite-group averaging
		construction; see also \cite{Takamura2019Linearization}.
	\end{proof}
	
	\begin{corollary}[Involutivity from the Differential]
		\label{cor:involutivity-from-differential}
		Let $\phi\in G_x(Q)$. If $(d\phi_x)^2=I$, then
		$\phi^{\circ2}=\operatorname{id}$ as a germ. In particular, a
		biholomorphic transport germ fixing $x$ whose differential is a
		rank-one reflection with transverse eigenvalue $-1$ is an involution.
	\end{corollary}
	
	\begin{proof}
		The germ $\phi^{\circ2}$ belongs to $G_x(Q)$ and its differential is
		$(d\phi_x)^2=I$. Injectivity in
		Lemma~\ref{lem:faithful-local-differential} implies
		$\phi^{\circ2}=\operatorname{id}$.
	\end{proof}

	%
	%

	\begin{proposition}[Reflection-Type Transport Components]
		\label{prop:Reflection-Type-Transport-Components}
		Let $D\in\Phi_Q$ and let $x\in H_D$. Assume that $(x,x)$ is a
		smooth point of $D^{\mathrm{an}}$ and that
		$d(p_1^{\mathrm{an}})_{(x,x)}:
		T_{(x,x)}D^{\mathrm{an}}\to T_xX^{\mathrm{an}}$
		is an isomorphism. Choose the local transport branch through $(x,x)$
		supplied by the Local Transport Theorem. Assume that
		$\iota(D)=D$, where $\iota$ is the algebraic coordinate
		interchange on $X\times_YX$. Then the selected transport germ
		$\tau:(X^{\mathrm{an}},x)\to(X^{\mathrm{an}},x)$ is biholomorphic and
		satisfies
		\(\tau^{-1}=\tau,\qquad \tau^{\circ2}=\operatorname{id}.\)
	\end{proposition}
	
	\begin{proof}
		The invariance $\iota(D)=D$ gives
		$\iota^{\mathrm{an}}(D^{\mathrm{an}})=D^{\mathrm{an}}$.
		Choose the neighborhood $W$ in the Local Transport Theorem small
		enough that it is invariant under $\iota^{\mathrm{an}}$; this is
		possible by replacing an initial neighborhood by its intersection
		with its image under $\iota^{\mathrm{an}}$. Moreover,
		$p_2^{\mathrm{an}}=p_1^{\mathrm{an}}\circ\iota^{\mathrm{an}}$
		on the analytic fibre product. Hence both projections are
		biholomorphic on $W$. The inverse branch has graph obtained by
		interchanging the two coordinates in the graph of $\tau$. Since
		$W$ and $D^{\mathrm{an}}$ are invariant under
		$\iota^{\mathrm{an}}$, this inverse branch is again a section of
		$p_1^{\mathrm{an}}|_W$ through $(x,x)$. By uniqueness of the section
		through that point, it agrees with $\tau$. Hence
		$\tau^{-1}=\tau$.
	\end{proof}

	\begin{corollary}[Root Wall as Fixed Locus]
		Assume the hypotheses of
		Proposition~\ref{prop:Reflection-Type-Transport-Components}, and let
		$\tau:U\to X^{\mathrm{an}}$ be the selected local transport map.
		The smoothness hypothesis in that proposition ensures that the
		selected branch is the full local analytic germ of $D^{\mathrm{an}}$
		at $(x,x)$. After shrinking $U$ around $x$ if necessary, the fixed-point
	locus of $\tau$ in $U$ is exactly $H_D^{\mathrm{an}}\cap U$, where $H_D$ is
		viewed as a subvariety of $X$ through the identification
		$\Delta\cong X$:
		\(\operatorname{Fix}(\tau)\cap U =H_D^{\mathrm{an}}\cap U.\)
	\end{corollary}
	
	\begin{proof}
		Shrink $U$ so that, for every $z\in H_D^{\mathrm{an}}\cap U$,
		the diagonal point $(z,z)$ lies in the selected neighborhood
		$W\subset D^{\mathrm{an}}$. If $z\in H_D^{\mathrm{an}}\cap U$,
		then $(z,z)\in D^{\mathrm{an}}$ and is the point of $W$ over $z$.
		By uniqueness of the inverse branch,
		$(z,\tau(z))=(z,z)$, so $\tau(z)=z$.
		Conversely, if $\tau(z)=z$, then
		$(z,\tau(z))=(z,z)\in D^{\mathrm{an}}\cap\Delta^{\mathrm{an}}$,
		which means $z\in H_D^{\mathrm{an}}$ under the diagonal
		identification. This proves the equality.
	\end{proof}

	\begin{remark}
		This result explains the terminology ``root wall.'' For a selected
		reflection-type transport germ, the root wall is locally its
		fixed-point locus and plays the role of the classical reflecting
		hyperplane.
	\end{remark}
	
	When a selected primitive transport germ is an involution, we call it
	the \emph{transport reflection} associated with that selected branch.
	The preceding proposition gives one sufficient condition for
	involutivity. In the single-finite-morphism setting, the
	Involutivity-from-the-Differential corollary gives another: a
	biholomorphic local transport germ fixing a root point and having
	reflection differential with transverse eigenvalue $-1$ is
	automatically an involution. These statements concern selected germs
	at root points; they do not assert that every branch of every
	coincidence component is globally single-valued or involutive.

	\begin{example}[The quadratic cover]
		\label{ex:quadratic_cover}
		Consider $Q:\mathbb A^1\to\mathbb A^1$, $Q(x)=x^2$. The deck group is $G=\{e,g\}$, $g(x)=-x$.
		The coincidence correspondence is $\mathcal C_Q = \{(x,y) : x^2=y^2\} = \{x=y\}\cup\{x+y=0\}$. Hence $\mathcal C_Q = \Gamma_e\cup\Gamma_g$, where
		$\Gamma_g = \{(x,gx):x\in X\}$
		is the graph of the deck transformation $g$.
		The unique non-diagonal coincidence component is $D=\Gamma_g$. Its transport symmetry is $\tau_D(x)=-x$.
		The root wall is $H_D=\Gamma_g\cap\Delta=\{(0,0)\}$, which corresponds to the fixed-point locus $\operatorname{Fix}(g)=\{0\}$.
		Thus the coincidence correspondence recovers both the nontrivial deck transformation and the deck group $G\cong\mathbb Z/2\mathbb Z$. A general theory of finite Galois covers and their connections with coincidence-root framework will be presented in \Cref{sec:Finite-Galois-Covers}.
	\end{example}

	\section{Transport Root Theory}
	\label{sec:Transport-Root-Theory}
	
	\paragraph{Branch convention.}
	Throughout this section, every transport operator is attached to a
	selected local graph branch. More precisely, if $x\in H_D$ and the
	Local Transport Theorem applies at $(x,x)\in D^{\mathrm{an}}$, we
	choose the local branch through $(x,x)$ and denote its transport map
	by $\tau_{D,\mathfrak b}$. All differential data below are relative
	to this choice. We work at points where $(x,x)$ is a smooth point of
	$D^{\mathrm{an}}$ and the first projection is locally biholomorphic.
	After shrinking the neighborhoods, the selected branch is the local
	analytic germ of $D^{\mathrm{an}}$ at $(x,x)$, and the root wall
	$H_D^{\mathrm{an}}$ is locally the fixed-point locus of
	$\tau_{D,\mathfrak b}$. When two components meet at $x$, a branch is
	chosen for each component. If a component is singular at $(x,x)$,
	its different local branches must be treated separately.

	\subsection{Transport Operators}
	\label{subsec:Transport-Operators}
	
	The local transport symmetries associated to coincidence components induce natural linear actions on tangent and cotangent spaces.
	
	\begin{definition}[Transport Operators]
		\label{def:transport-operator}
		Let $D\in\Phi_Q$, let $x\in H_D$, and choose a local transport
		branch $\mathfrak b$ through $(x,x)$ as above. The differential
		$R_{D,\mathfrak b}(x):=(d\tau_{D,\mathfrak b})_x:T_xX\to T_xX$
		is called the \emph{transport operator} associated with the selected
		branch at $x$. Its dual $R_{D,\mathfrak b}^*(x):T_x^*X\to T_x^*X$
		is called the \emph{cotangent transport operator}. When the branch
		is fixed in the discussion, we write $R_D(x)$ and $R_D^*(x)$.
	\end{definition}
	
	The transport operators give the infinitesimal realization of the
	selected local coincidence branches. In particular, they are
	attached to local branches, not merely to the global irreducible
	component.
	
	\subsection{Transport Root Data}
	\label{subsec:Transport-Root-Data}
	
	\begin{proposition}[Tangent Space of the Root Wall as the Fixed Eigenspace]
		\label{prop:root-wall-fixed-eigenspace}
		Let $D$ be an admissible coincidence component and let $x\in H_D$.
		Assume that $(x,x)$ is a smooth point of $D^{\mathrm{an}}$, and that
		the selected local transport branch $\mathfrak b$ through $(x,x)$ is
		the full local analytic germ of $D^{\mathrm{an}}$ there (with the first
		projection locally biholomorphic). Assume also that $H_D$ is a smooth
		hypersurface near $x$. Write
		$\tau=\tau_{D,\mathfrak b}$ and $R_D(x)=d\tau_x$. Then, after
		shrinking the neighborhood, $H_D^{\mathrm{an}}$ is the fixed-point
		locus of $\tau$. If $\operatorname{rank}(I-R_D(x))=1$, then
		$T_xH_D=\ker(I-R_D(x))=\ker(R_D(x)-I)$.
	\end{proposition}
	
	\begin{proof}
		By construction, the selected branch is locally the graph of
		$\tau$. Thus $\tau(z)=z$ if and only if $(z,z)$ belongs to the
		selected local branch, which is precisely the local root wall after
		shrinking around $(x,x)$. Hence $H_D^{\mathrm{an}}$ is locally the
		fixed-point locus of $\tau$. Differentiating the identity
		$\tau|_{H_D}=\operatorname{id}_{H_D}$ gives
		$T_xH_D\subseteq\ker(I-R_D(x))$. Since $H_D$ is a smooth
		hypersurface, $\dim T_xH_D=\dim X-1$. The rank assumption and
		rank--nullity give $\dim\ker(I-R_D(x))=\dim X-1$, so the inclusion
		is an equality. Finally, $\ker(I-R_D(x))=\ker(R_D(x)-I)$.
	\end{proof}
	
	\begin{definition}[Regular Transport Root Point]
		\label{def:Regular-Transport-Root-Point}
		Let $D$ be an admissible coincidence component, let $x\in H_D$, and
		choose a local transport branch $\mathfrak b$ through $(x,x)$. Put
		$R_D(x)=d\tau_{D,\mathfrak b}|_x$. We call $x$ a
		\emph{regular transport root point for this branch} if $H_D$ is a
		smooth hypersurface at $x$ and $\operatorname{rank}(I-R_D(x))=1$.
	\end{definition}
	
	\begin{remark}[Automatic Splitting for Genuine Root Self-Germs]
		For an abstract rank-one endomorphism, the splitting
		$T_xX=\ker(I-R)\oplus\operatorname{Im}(I-R)$ need not hold:
		a nontrivial unipotent rank-one perturbation of the identity can have
		$\operatorname{Im}(I-R)\subseteq\ker(I-R)$. Such a matrix is
		an abstract linear-algebra possibility, but it cannot arise as the
		differential of a holomorphic self-germ fixing $x$ and preserving one
		finite morphism. By
		Proposition~\ref{prop:local-finiteness-transport-germs}, every such
		self-germ has finite order; its differential is therefore semisimple
		over $\mathbb C$, and the splitting is automatic. We retain the
		splitting explicitly in the normalized definitions below to display
		the precise linear-algebraic condition used in the formulas and to
		make the statements readable when viewed abstractly.
	\end{remark}
	
	\begin{remark}[Regularity Convention]
	The transport-root constructions are made at regular transport root
	points for selected local branches. The smoothness and rank-one
	conditions are pointwise conditions; without additional algebraicity
	or family hypotheses, we do not assert that their locus is Zariski
	open. For genuine root self-germs arising from a finite morphism,
	the splitting condition is automatic by
	Proposition~\ref{prop:local-finiteness-transport-germs}. We retain
	it explicitly in normalized definitions to display the precise
	linear-algebraic input and to distinguish this setting from abstract
	rank-one endomorphisms for which splitting can fail.
\end{remark}
	
	\begin{corollary}[No Genuine Tangency for Distinct Intrinsic Root Walls]
	\label{cor:no-genuine-tangency-common-morphism}
	Let $Q:X\to Y$ be a finite morphism of smooth irreducible complex
	varieties, and let two selected local transport self-germs
	$\tau_i,\tau_j\in G_x(Q)$ have rank-one deviations from the identity.
	Assume that their fixed-point loci are smooth hypersurface germs
	$H_i,H_j$ at $x$, as supplied by the selected full local branches.
	If $T_xH_i=T_xH_j$, then $H_i=H_j$ as analytic germs. In
	particular, two distinct intrinsic root-wall germs arising from one
	common finite morphism cannot have genuine tangency: if they are
	distinct and meet at $x$, their tangent hyperplanes are distinct and
	therefore transverse.
\end{corollary}

\begin{proof}
	By Theorem~\ref{thm:simultaneous-local-analytic-linearization},
	there is a biholomorphic coordinate germ $\Psi$ with
	$d\Psi_x=I$ that simultaneously conjugates every element of
	$G_x(Q)$ to its differential. Hence $\Psi$ sends the fixed-point
	germ of $\tau_k$ to the linear fixed hyperplane
	$\ker(I-d(\tau_k)_x)$ for $k=i,j$. Since $d\Psi_x=I$, the
	assumption $T_xH_i=T_xH_j$ says that these two linear hyperplanes
	coincide. Their fixed-point germs in the linearized coordinates
	are therefore the same hyperplane germ. Applying $\Psi^{-1}$ gives
	$H_i=H_j$ as analytic germs.
\end{proof}

	\begin{definition}[Transport Root]
		\label{def:Transport-Root}
		Let $D_i\in\Phi_Q$ and let $x\in H_{D_i}$ be a regular transport
		root point for a selected branch. The \emph{transport root line} is
		$L_i(x):=\operatorname{Ann}(\ker(I-R_i(x)))\subset T_x^*X$. A
		nonzero covector $\alpha_i(x)\in L_i(x)$ is called a
		\emph{transport root}.
	\end{definition}
	
	\begin{remark}
		By Proposition~\ref{prop:root-wall-fixed-eigenspace},
		$\ker(I-R_i(x))=T_xH_{D_i}$. Thus
		$L_i(x)=\operatorname{Ann}(T_xH_{D_i})=N_x^*(H_{D_i})$, the
		conormal line of the root wall. The image
		$C_i(x):=\operatorname{Im}(I-R_i(x))$ is also one-dimensional.
		For the transport self-germs arising from finite morphisms, the
		automatic finite-order result above makes $R_i(x)$ semisimple, so
		$T_xX=\ker(I-R_i(x))\oplus C_i(x)$. Consequently every nonzero
		$\alpha_i\in L_i(x)$ pairs nontrivially with every nonzero
		$\nu_i\in C_i(x)$, and there is a unique rescaling for which
		$\alpha_i(\nu_i)=1$. Nonsplitting remains possible for abstract
		rank-one matrices not arising from such a finite-morphism self-germ.
	\end{remark}
	
	\begin{definition}[Transport Cartan Coefficients]
		\label{def:Transport-Cartan-Coefficients}
		Let $D_i,D_j\in\Phi_Q$ and let $x\in H_{D_i}\cap H_{D_j}$ be a
		smooth intersection point. Choose a local transport branch through
		$(x,x)$ for each component, and assume that $x$ is a regular
		transport root point for both selected branches. Assume in addition
		that $T_xX=\ker(I-R_i(x))\oplus C_i(x)$ and
		$T_xX=\ker(I-R_j(x))\oplus C_j(x)$, where
		$C_k(x):=\operatorname{Im}(I-R_k(x))$. Choose nonzero roots
		$\alpha_i\in L_i(x)$ and $\alpha_j\in L_j(x)$. There are unique
		vectors $\nu_i\in C_i(x)$ and $\nu_j\in C_j(x)$ satisfying
		$\alpha_i(\nu_i)=1$ and $\alpha_j(\nu_j)=1$. The
		\emph{transport Cartan coefficients} are
		$A_{ij}(x):=\alpha_j(\nu_i)$ and
		$A_{ji}(x):=\alpha_i(\nu_j)$. We write
		$\Theta_{ij}(x):=A_{ij}(x)A_{ji}(x)$. For a fixed collection of
		components and selected branches at $x$, the matrix
		$A_Q(x):=(A_{ij}(x))$ is called the \emph{transport Cartan matrix}
		at $x$.
	\end{definition}
	
	\begin{remark}
		The vector $\nu_i$ is the normalized transverse direction associated
		with the selected transport branch, and plays the role of a classical
		coroot direction. It is introduced to form the Cartan pairings; it is
		not an additional choice of geometric structure on $X$.
	\end{remark}
	
	\begin{remark}
		All pairings are taken in the same tangent-cotangent fibre at $x$;
		therefore no extensions of roots or auxiliary connections are
		required. Once normalized roots and transverse vectors are chosen,
		the coefficients are defined by their pointwise pairings. Their
		rescaling behaviour is recorded in the Transformation Law below.
	\end{remark}
	
	\begin{proposition}[Transformation Law]
		\label{prop:Transformation-Law}
		Fix selected local transport branches for $D_i$ and $D_j$ through
		$(x,x)$, and suppose $x$ is a regular transport root point for both, with the
		splitting condition in Definition~\ref{def:Transport-Cartan-Coefficients}.
		Let $A_{ij}(x)=\alpha_j(\nu_i)$ and
		$A_{ji}(x)=\alpha_i(\nu_j)$ be the corresponding coefficients.
		If the normalized data are rescaled by
		$\nu_i\mapsto\lambda_i\nu_i$,
		$\alpha_i\mapsto\lambda_i^{-1}\alpha_i$ and similarly for
		$j$, where $\lambda_i,\lambda_j$ are nonzero, then
		$A_{ij}\mapsto\lambda_i\lambda_j^{-1}A_{ij}$ and
		$A_{ji}\mapsto\lambda_j\lambda_i^{-1}A_{ji}$. Consequently,
		$\Theta_{ij}=A_{ij}A_{ji}$ is unchanged.
	\end{proposition}
	
	\begin{proof}
		The asserted transformations follow directly from
		$A_{ij}=\alpha_j(\nu_i)$ and $A_{ji}=\alpha_i(\nu_j)$:
		$A_{ij}$ is multiplied by $\lambda_i\lambda_j^{-1}$ and
		$A_{ji}$ by $\lambda_j\lambda_i^{-1}$. Their product is
		therefore unchanged.
	\end{proof}

	\begin{remark}
		The individual coefficients $A_{ij}$ depend on the normalized root
		choices. Their product $\Theta_{ij}=A_{ij}A_{ji}$ is invariant under
		those rescalings, but it is still attached to the selected pair of
		local branches at $x$. Thus $\Theta_{ij}$ is a branchwise local
		invariant, not an invariant of the two global irreducible components
		without a branch choice.
	\end{remark}

	\subsection{Transport Trace Forms}
	\label{subsec:Transport-Trace-Forms}
	
	The transport operators associated with selected local coincidence
	branches give rise to trace-theoretic invariants. For fixed branches,
	the trace form is unchanged by a change of local coordinates; changing
	the selected branches may change its value.
	
	\begin{definition}[Transport Trace Form]
		Let $D_i,D_j\in\Phi_Q$, let $x\in H_{D_i}\cap H_{D_j}$ be a smooth
		intersection point, and choose a local transport branch through
		$(x,x)$ for each component. With the corresponding operators
		$R_i(x)$ and $R_j(x)$, the \emph{transport trace form} is
		$B_{ij}(x):=\operatorname{tr}((I-R_i(x))(I-R_j(x)))$.
	\end{definition}
	
	\begin{proposition}[Intrinsic Transport Trace Form]
		\label{prop:Intrinsic-TransportTrace-Form}
		For fixed selected branches, the quantity $B_{ij}(x)$ is independent
		of the choice of local coordinates on $X$.
	\end{proposition}
	
	\begin{proof}
		A change of local coordinates conjugates both operators by the same
		invertible linear map, say $R_i\mapsto P^{-1}R_iP$ and
		$R_j\mapsto P^{-1}R_jP$. It therefore conjugates their product
		$(I-R_i)(I-R_j)$ by $P$. The trace is invariant under conjugation,
		so $B_{ij}(x)$ is unchanged.
	\end{proof}

	\begin{lemma}[Rank-One Reflection Formula]
		\label{lem:rank-one-reflection-formula}
		
		Let $x$ be a regular transport root point, so that
		$H_{D_i}$ is a smooth hypersurface at $x$ and
		$\operatorname{rank}(I-R_i)=1$. Let
		$R_i:T_xX\to T_xX$ be a finite-order transport operator
		over a field $k$ of characteristic zero, and let
		$\alpha_i\in N_x^*(H_{D_i})$ be a nonzero transport root.
		Then there exist $\nu_i\in\operatorname{Im}(I-R_i)$,
		$\zeta_i\in k$, and a positive integer $m$ such that
		$\alpha_i(\nu_i)=1$, $\zeta_i\neq1$,
		$R_i^m=I$, $\zeta_i^m=1$, and
		$R_i\nu_i=\zeta_i\nu_i$. Moreover,
		$(I-R_i)(v)=(1-\zeta_i)\alpha_i(v)\nu_i$
		for every $v\in T_xX$.
		
	\end{lemma}
	
	\begin{proof}
		By Proposition~\ref{prop:root-wall-fixed-eigenspace},
		$\ker(I-R_i)=T_xH_{D_i}$. Since $H_{D_i}$ is a
		hypersurface at $x$ and $\alpha_i$ is a nonzero
		conormal covector, we also have
		$\ker\alpha_i=T_xH_{D_i}$. Consequently,
		$\ker\alpha_i=\ker(I-R_i)$.
		Write $S=I-R_i$, and choose a positive integer $m$
		such that $R_i^m=I$.
		We first show that $\operatorname{Im}S\cap\ker S=\{0\}$.
		Suppose $w\in\operatorname{Im}S\cap\ker S$.
		Then $w=Su$ for some $u\in T_xX$, and $R_iw=w$.
		It follows that
		$\sum_{r=0}^{m-1}R_i^rw=mw$.
		On the other hand, since $w=(I-R_i)u$, the same
		sum is
		$\sum_{r=0}^{m-1}R_i^r(I-R_i)u
		=(I-R_i^m)u=0$.
		Thus $mw=0$. Since $k$ has characteristic zero,
		$w=0$.
		
		By regularity, $\operatorname{rank}S=1$, so
		$\operatorname{Im}S$ is one-dimensional and nonzero.
		Choose a nonzero $u_0\in\operatorname{Im}S$.
		The intersection identity and
		$\ker\alpha_i=\ker S$ imply that
		$\alpha_i(u_0)\neq0$.
		Define $\nu_i=\alpha_i(u_0)^{-1}u_0$.
		Then $\nu_i\in\operatorname{Im}S$ and
		$\alpha_i(\nu_i)=1$.
		Since $R_i$ commutes with $S$, the image of $S$
		is invariant under $R_i$. Its one-dimensionality
		therefore gives a scalar $\zeta_i\in k$ such that
		$R_i\nu_i=\zeta_i\nu_i$.
		Moreover, $\zeta_i\neq1$, since otherwise
		$\nu_i\in\operatorname{Im}S\cap\ker S$, contradicting
		$\alpha_i(\nu_i)=1$.
		
		Since $R_i\nu_i=\zeta_i\nu_i$ and $R_i^m=I$, we have
		$\nu_i=R_i^m\nu_i=\zeta_i^m\nu_i$.
		Since $\nu_i\neq0$, we obtain $\zeta_i^m=1$.
		Finally, for any $v\in T_xX$, the vector
		$v-\alpha_i(v)\nu_i$ belongs to
		$\ker\alpha_i=\ker S$.
		Consequently,
		$(I-R_i)(v)=(I-R_i)(\alpha_i(v)\nu_i)
		=(1-\zeta_i)\alpha_i(v)\nu_i$.
		This proves the assertion.
	\end{proof}
	
	\begin{remark}
		The tensor $\nu_i \otimes \alpha_i \in T_xX \otimes T_x^*X$ is identified with the rank-one endomorphism $T_xX \longrightarrow T_xX$ defined by $v \longmapsto \alpha_i(v)\,\nu_i$. 
		Thus the Rank-One Reflection Formula may be written equivalently as $(I-R_i)(v) = (1-\zeta_i)\alpha_i(v)\,\nu_i$ for all $v \in T_xX$.
	\end{remark}
	
	\begin{remark}[Finite Order from a Common Finite Morphism]
		Suppose $Q:X\to Y$ is a finite morphism as in
		Proposition~\ref{prop:local-finiteness-transport-germs}, and a
		selected transport branch through $(x,x)$ has invertible
		differential. The inverse function theorem makes its transport map a
		biholomorphic germ fixing $x$ and preserving $Q$. It therefore belongs
		to the finite group $G_x(Q)$, so the germ, and hence its differential,
		has finite order. In particular, for a regular rank-one branch whose
		transverse eigenvalue is nonzero, the differential is invertible and
		this argument applies. The finite-order hypothesis in the Rank-One
		Reflection Formula is therefore automatic for such biholomorphic
		selected branches of a common finite morphism.
	\end{remark}

	The following theorem identifies the transport trace form with a Cartan-type pairing whenever the transport operators behave as finite-order reflections.
	\begin{theorem}[Finite-Order Transport Trace Formula]
		\label{thm:Finite-Order-Transport-Trace-Formula}
		Let $x\in H_{D_i}\cap H_{D_j}$ be a smooth intersection point, and
		choose a local transport branch for each component. Assume that the
		corresponding operators $R_i$ and $R_j$ have finite orders, and that
		$x$ is a regular transport root point for both branches. Let
		$\zeta_i\ne1$ and $\zeta_j\ne1$ be their respective transverse
		eigenvalues, and choose normalized pairs
		$(\alpha_i,\nu_i)$ and $(\alpha_j,\nu_j)$ as above. Then
		$B_{ij}(x)=(1-\zeta_i)(1-\zeta_j)A_{ij}(x)A_{ji}(x)$.
		In particular, if $\zeta_i=\zeta_j=-1$, then
		$B_{ij}(x)=4A_{ij}(x)A_{ji}(x)$.
	\end{theorem}
	
	\begin{proof}
		By the Rank-One Reflection Formula,
		$I-R_i=(1-\zeta_i)\nu_i\otimes\alpha_i$ and
		$I-R_j=(1-\zeta_j)\nu_j\otimes\alpha_j$. Under the identification
		of $u\otimes f$ with the map $v\mapsto f(v)u$, one has
		$(\nu_i\otimes\alpha_i)(\nu_j\otimes\alpha_j)
		=\alpha_i(\nu_j)(\nu_i\otimes\alpha_j)$. Using
		$\operatorname{tr}(u\otimes f)=f(u)$, it follows that
		$B_{ij}(x)=(1-\zeta_i)(1-\zeta_j)\alpha_i(\nu_j)\alpha_j(\nu_i)$,
		which is the asserted formula by the definitions of $A_{ij}$ and
		$A_{ji}$.
	\end{proof}
	
	\begin{corollary}[Trace Formula for the Transport Invariant $\Theta_{ij}$]
		\label{cor:Trace-Theoretic-Expression}
		Under the hypotheses of the Finite-Order Transport Trace Formula,
		$\Theta_{ij}(x)=B_{ij}(x)/((1-\zeta_i)(1-\zeta_j))$.
	\end{corollary}
	
	\begin{proof}
		This follows immediately from
		$B_{ij}(x)=(1-\zeta_i)(1-\zeta_j)A_{ij}(x)A_{ji}(x)$ and
		$\Theta_{ij}(x)=A_{ij}(x)A_{ji}(x)$.
	\end{proof}

	\begin{proposition}[Compatibility with Classical Cartan Data]
		\label{prop:classical-cartan-compatibility}
		Let $V_{\mathbb R}$ be a finite-dimensional real Euclidean vector
		space, and let $W\subset O(V_{\mathbb R})$ be a finite Weyl group
		with root system $\Phi\subset V_{\mathbb R}$. Let
		$V:=V_{\mathbb R}\otimes_{\mathbb R}\mathbb C$ be the
		complexification, extend the $W$-action complex-linearly, and denote
		by $\langle\cdot,\cdot\rangle$ the complex-bilinear extension of
		the Euclidean inner product. We regard $W$ as acting on $V$ and use
		the complex affine quotient $V/W:=\operatorname{Spec}\mathbb C[V]^W$
		and its quotient morphism $Q:V\to V/W$. Let
		$s_i,s_j$ be distinct simple reflections, with corresponding simple
		roots $\beta_i,\beta_j$ and coroots $\beta_i^\vee,\beta_j^\vee$.
		The extended bilinear form identifies $V$ and $V^*$, and we use the
		convention
		$\beta_i^\vee=2\beta_i/\langle\beta_i,\beta_i\rangle$ and
		$s_i(v)=v-\langle v,\beta_i^\vee\rangle\beta_i$
		(and analogously for $j$). For the specific normalized transport
		choices $\alpha_i^{\mathrm{tr}}=\beta_i/2$, $\nu_i=\beta_i^\vee$
		and $\alpha_j^{\mathrm{tr}}=\beta_j/2$, $\nu_j=\beta_j^\vee$,
		the transport Cartan coefficients are
		$A_{ij}=\frac12\langle\beta_j,\beta_i^\vee\rangle$ and
		$A_{ji}=\frac12\langle\beta_i,\beta_j^\vee\rangle$.
		Consequently,
		$\Theta_{ij}=\frac14\langle\beta_j,\beta_i^\vee\rangle
		\langle\beta_i,\beta_j^\vee\rangle$.
		\end{proposition}
	
	\begin{proof}
		For the quotient map $Q:V\to V/W$, the coincidence components are
		the graphs of elements of $W$. The components associated with $s_i$
		and $s_j$ are therefore the graphs of these reflections. Let $H_i$
		be the reflecting hyperplane of $s_i$. Since $s_i$ fixes $H_i$
		pointwise, $T_xH_i=\ker(I-s_i)$, while
		$\operatorname{Im}(I-s_i)$ is the line spanned by $\beta_i^\vee$.
		Choose $\nu_i=\beta_i^\vee$ and $\alpha_i=\beta_i/2$; the identity
		$\beta_i(\beta_i^\vee)=2$ gives $\alpha_i(\nu_i)=1$. Hence
		$A_{ij}=\alpha_j(\nu_i)=\frac12\langle\beta_j,\beta_i^\vee\rangle$.
		The analogous calculation gives the expression for $A_{ji}$, and
		multiplication gives the stated formula for $\Theta_{ij}$.
	\end{proof}
	
	\begin{remark}[Selected Simple Reflections and Normalization]
		This compatibility result is a special case of the more general
		reconstruction theory developed in
		\Cref{sec:Finite-Galois-Covers}. For a finite Galois quotient,
		all coincidence components are graphs of all deck transformations.
		In the Weyl-group setting, the graphs of the simple reflections form
		a distinguished generating subset, not the full set of non-diagonal
		components. The displayed individual Cartan coefficients are recovered
		for the specific normalization in
		Proposition~\ref{prop:classical-cartan-compatibility}; reciprocal
		rescaling of the normalized root pairs changes the individual
		coefficients but leaves $\Theta_{ij}$ unchanged.
	\end{remark}

	\begin{corollary}[Classical Coxeter Recovery]
		\label{cor:classical-coxeter-recovery}
		In the setting of Proposition~\ref{prop:classical-cartan-compatibility},
		let $i\ne j$ and let $m_{ij}\ge2$ be the order of $s_is_j$. If
		$\vartheta_{ij}$ is the Euclidean angle between the corresponding
		simple roots, then $\Theta_{ij}=\cos^2(\vartheta_{ij})
		=\cos^2(\pi/m_{ij})$.
	\end{corollary}
	
	\begin{proof}
		Write $C_{ij}=\langle\beta_j,\beta_i^\vee\rangle$ and
		$C_{ji}=\langle\beta_i,\beta_j^\vee\rangle$. Since
		$\beta_i^\vee=2\beta_i/\langle\beta_i,\beta_i\rangle$ and
		$\beta_j^\vee=2\beta_j/\langle\beta_j,\beta_j\rangle$,
		one obtains
		$C_{ij}C_{ji}=4\langle\beta_i,\beta_j\rangle^2/
		(\langle\beta_i,\beta_i\rangle\langle\beta_j,\beta_j\rangle)$.
		Proposition~\ref{prop:classical-cartan-compatibility} therefore gives
		$\Theta_{ij}=\langle\beta_i,\beta_j\rangle^2/
		(\langle\beta_i,\beta_i\rangle\langle\beta_j,\beta_j\rangle)
		=\cos^2(\vartheta_{ij})$. For distinct simple reflections in a finite
		Weyl group, the angle between the simple roots is
		$\vartheta_{ij}=\pi-\pi/m_{ij}$ (with $m_{ij}=2$ in the
		orthogonal case). Thus
		$\cos^2(\vartheta_{ij})=\cos^2(\pi/m_{ij})$, as claimed.
	\end{proof}

	\begin{remark}[From Classical to Nonlinear Root Geometry]
		For distinct simple reflections in a finite Weyl group, the possible
		orders are $m_{ij}=2,3,4,6$, and the corresponding values of
		$\Theta_{ij}=\cos^2(\pi/m_{ij})$ are, respectively,
		$0,\frac14,\frac12,\frac34$. The value $1$ corresponds to
		$m_{ij}=1$, which occurs when the two reflections coincide; it is
		excluded when $i\ne j$. Along edges of the simply-laced Dynkin
		diagrams $A_n,D_n,E_n$, one has $\Theta_{ij}=\frac14$; along the
		multiple edges of $B_n,C_n$ one has $\Theta_{ij}=\frac12$, and along
		the triple edge of $G_2$ one has $\Theta_{ij}=\frac34$. Non-adjacent
		simple roots are orthogonal and give $\Theta_{ij}=0$.
		
		The unrestricted transport formalism is broader than the common
		finite-cover reflection setting. When two selected branches do not
		both define biholomorphic germs fixing the same point for one common
		finite morphism, the finiteness argument of
		Proposition~\ref{prop:local-finiteness-transport-germs} does not apply
		to their product. In that broader setting, the branchwise invariant
		$\Theta_{ij}$ is not required to belong to the finite Weyl-group
		list and may vary with the local geometry. By contrast, the common
		finite-morphism reflection setting is subject to the additional
		finite-order restrictions developed in the next section.
	\end{remark}
	
	\begin{remark}[Transport Angles and Coxeter Data]
		Whenever $0\le\Theta_{ij}\le1$, define the real
		\emph{transport angle} by
		$\theta_{ij}:=\arccos(\sqrt{\Theta_{ij}})$. In the classical
		Weyl-group setting with $i\ne j$, this recovers
		$\theta_{ij}=\pi/m_{ij}$; the angle between the simple-root vectors
		is its supplement, $\vartheta_{ij}=\pi-\pi/m_{ij}$. For selected
		rank-one self-germs fixing the same point and preserving one common
		finite morphism, the Hermitian-form argument in
		Proposition~\ref{prop:strict-theta-common-morphism} gives the
		non-strict bound $0\le\Theta_{ij}\le1$; under its distinctness and
		transversality hypotheses, the bound is strict. Thus a value outside
		$[0,1]$ can occur only outside the common-finite-morphism
		rank-one setting, for example for mixed-chart or other
		non-common-morphism pairs. In that broader setting the real transport
		angle need not be defined. When it is defined, it is a local,
		branchwise quantity. The collection of such invariants may be viewed
		as a nonlinear analogue of the classical Coxeter data; see also
		\cite{humphreys1990reflection} for the classical theory.
		\end{remark}

	\section{Rank-Two Transport Geometry}
	\label{sec:Rank-Two-Transport-Geometry}

	\subsection{Rank-Two Reduction Theory}
	\label{subsec:Rank-Two-Reduction-Theory}
	
	The interaction of two coincidence components is governed by the pair of
	transport operators $R_i$ and $R_j$. Although these operators act on the
	full tangent space $T_xX$, most directions are geometrically irrelevant. In particular,
	the common tangent directions to the two root walls are fixed by both
	transport operators. The following theorem shows that, after factoring
	out these fixed directions, all nontrivial transport geometry is
	concentrated in a canonical two-dimensional quotient.

	\begin{theorem}[Rank-Two Reduction Theorem]
		\label{thm:Rank-Two-Reduction-Theorem}
		Let $D_i,D_j\in\Phi_Q$, choose local transport branches through
		$(x,x)$ for both components, and let $R_i,R_j:T_xX\to T_xX$ be the
		corresponding transport operators. Suppose that
		$x\in H_{D_i}\cap H_{D_j}$ is a smooth transverse intersection point
		and a regular transport root point for both selected branches. Put
		$F_{ij}:=T_xH_{D_i}\cap T_xH_{D_j}$ and $n:=\dim X$. Then:
		\begin{enumerate}
			\item[(i)] $\dim F_{ij}=n-2$;
			\item[(ii)] every vector in $F_{ij}$ is fixed by $R_i$, $R_j$, and
			$M_{ij}:=R_iR_j$;
			\item[(iii)] $M_{ij}$ preserves $F_{ij}$ and therefore induces a
			linear operator $\overline M_{ij}:T_xX/F_{ij}\to T_xX/F_{ij}$;
			\item[(iv)] $\dim(T_xX/F_{ij})=2$.
		\end{enumerate}
	\end{theorem}
	
	\begin{proof}
		Since the two smooth hypersurfaces intersect transversely,
		$T_xH_{D_i}+T_xH_{D_j}=T_xX$. Hence
		$\dim F_{ij}=(n-1)+(n-1)-n=n-2$. If $v\in F_{ij}$, then
		$v\in T_xH_{D_i}$ and $v\in T_xH_{D_j}$. By
		Proposition~\ref{prop:root-wall-fixed-eigenspace}, the selected
		transport operators fix these tangent spaces pointwise. Thus
		$R_i(v)=R_j(v)=v$, and consequently $M_{ij}(v)=v$. In particular,
		$F_{ij}$ is invariant under $M_{ij}$, so the induced quotient operator
		is well-defined. Finally, $\dim(T_xX/F_{ij})=n-(n-2)=2$.
	\end{proof}
	
	The theorem isolates the rank-two spectral data of the interaction of
	$D_i$ and $D_j$ in the induced operator
	$\overline M_{ij}:T_xX/F_{ij}\to T_xX/F_{ij}$
	on a canonical two-dimensional quotient. The characteristic polynomial
	and spectrum of this quotient are determined by the normalized pairwise
	data. In a mixed-chart setting, however, the quotient need not determine
	the full Jordan form of $M_{ij}$ on $T_xX$: the extension of the quotient
	action by the identity action on $F_{ij}$ can carry additional nilpotent
	coupling, as the example following the mixed-chart formula illustrates.

	\begin{theorem}[Rank--Two Characteristic Polynomial]
		\label{thm:rank-two-characteristic-polynomial}
		Assume the hypotheses of the Rank-Two Reduction Theorem. Assume in
		addition that the selected branches satisfy the splitting condition in
		Definition~\ref{def:Transport-Cartan-Coefficients}, and choose the
		corresponding normalized transport data $(\alpha_i,\nu_i)$ and
		$(\alpha_j,\nu_j)$. Thus $\alpha_i(\nu_i)=\alpha_j(\nu_j)=1$,
		$\nu_i\in\operatorname{Im}(I-R_i)$, and
		$\nu_j\in\operatorname{Im}(I-R_j)$. Since each image line is invariant
		under its transport operator, write $R_i\nu_i=\zeta_i\nu_i$ and
		$R_j\nu_j=\zeta_j\nu_j$. Put
		$A_{ij}:=\alpha_j(\nu_i)$,
		$A_{ji}:=\alpha_i(\nu_j)$, and
		$\Theta_{ij}:=A_{ij}A_{ji}$. Then the characteristic polynomial of
		$\overline M_{ij}$ is
		\begin{equation}
			\label{eq:characteristic_polynomial_reduced_transport_operator}
			p_{ij}(t)=t^2-\Bigl(\zeta_i+\zeta_j+(1-\zeta_i)(1-\zeta_j)\Theta_{ij}\Bigr)t
			+\zeta_i\zeta_j.
		\end{equation}
		In particular, the required normalized data and rank-one formulas are
		available whenever both selected transport operators have finite order,
		by Lemma~\ref{lem:rank-one-reflection-formula}.
	\end{theorem}
	
	\begin{proof}
		The image lines used in the statement are invariant because
		$R_i(I-R_i)=(I-R_i)R_i$ and
		$R_j(I-R_j)=(I-R_j)R_j$. Thus each transport operator preserves
		its corresponding one-dimensional image line. Define rank-one
		endomorphisms $\Pi_i,\Pi_j$ by
		$\Pi_i(v):=\alpha_i(v)\nu_i$ and
		$\Pi_j(v):=\alpha_j(v)\nu_j$. The fixed-space identity
		$\ker\alpha_i=\ker(I-R_i)$ and the normalization
		$\alpha_i(\nu_i)=1$ imply
		$R_i=I-(1-\zeta_i)\Pi_i$; similarly,
		$R_j=I-(1-\zeta_j)\Pi_j$. These identities are also the direct
		rank-one consequence of Lemma~\ref{lem:rank-one-reflection-formula}
		when finite order is assumed.
		
		For a rank-one endomorphism of the form $v\mapsto\alpha(v)\nu$,
		its trace is $\alpha(\nu)$. Therefore
		$\operatorname{tr}(\Pi_i)=\operatorname{tr}(\Pi_j)=1$. Moreover,
		$\Pi_i\Pi_j(v)=\alpha_j(v)\alpha_i(\nu_j)\nu_i$, and hence
		$\operatorname{tr}(\Pi_i\Pi_j)=\alpha_i(\nu_j)\alpha_j(\nu_i)
		=A_{ji}A_{ij}=\Theta_{ij}$. Expanding
		$M_{ij}=R_iR_j$ now gives
		$\operatorname{tr}(M_{ij})=n-2+\zeta_i+\zeta_j+
		(1-\zeta_i)(1-\zeta_j)\Theta_{ij}$.
		
		The decomposition $T_xX=\ker\alpha_i\oplus k\nu_i$ shows that
		$R_i$ acts as the identity on $\ker\alpha_i$ and as multiplication
		by $\zeta_i$ on $k\nu_i$. Thus $\det R_i=\zeta_i$; similarly,
		$\det R_j=\zeta_j$, so $\det M_{ij}=\zeta_i\zeta_j$.
		Since $M_{ij}$ fixes $F_{ij}$ pointwise, a basis adapted to
		$F_{ij}$ makes its matrix block upper triangular with identity block
		on $F_{ij}$ and quotient block $\overline M_{ij}$. Consequently,
		$\operatorname{tr}(\overline M_{ij})=
		\operatorname{tr}(M_{ij})-(n-2)$ and
		$\det(\overline M_{ij})=\det(M_{ij})$. Substitution into the
		characteristic polynomial of a two-dimensional operator proves the
		formula.
	\end{proof}
	
	\begin{corollary}[Conjugacy Reconstruction]
		\label{cor:conjugacy-reconstruction}
		Under the hypotheses of the Rank--Two Characteristic Polynomial
		Theorem, if the characteristic polynomial of $\overline M_{ij}$ has
		distinct roots, then $\Theta_{ij}$ determines $\overline M_{ij}$ up
		to conjugacy, provided $\zeta_i$ and $\zeta_j$ are fixed.
	\end{corollary}
	
	\begin{proof}
		The characteristic polynomial is determined by
		$\zeta_i,\zeta_j$, and $\Theta_{ij}$. Over the algebraically closed
		field $\mathbb C$, a two-dimensional operator with two distinct
		eigenvalues is diagonalizable, and its conjugacy class is determined
		by those eigenvalues.
	\end{proof}
	
	\begin{remark}
		The proof uses traces and determinants on the full tangent space and
		then passes to the quotient. It does not require the image vectors
		$\nu_i$ and $\nu_j$ to be linearly independent modulo $F_{ij}$.
		Transversality of the walls gives independence of their conormal
		covectors, not independence of the image directions. For example, on
		$\mathbb C^2$ the involutions
		$R_i=\begin{pmatrix}1&0\\0&-1\end{pmatrix}$ and
		$R_j=\begin{pmatrix}1&0\\2&-1\end{pmatrix}$ have distinct fixed
		lines, but $\operatorname{Im}(I-R_i)=\operatorname{Im}(I-R_j)$.
		With $\alpha_i(x,y)=y$, $\alpha_j(x,y)=y-x$, and
		$\nu_i=\nu_j=(0,1)$, both normalizations hold and
		$\Theta_{ij}=1$. The reduced product is a nontrivial unipotent
		matrix. Thus the basis argument based on $[\nu_i],[\nu_j]$ is not
		valid in general, whereas the trace argument above remains valid.
	\end{remark}

	\subsection{Rank-Two Transport Classification}
	\label{subsec:Rank-Two-Transport-Classification}
	
	The Rank-Two Reduction Theorem shows that the nontrivial transport
	dynamics of a pair of selected branches are represented by
	$\overline M_{ij}$ on a two-dimensional quotient. We first record the
	Cartan normalization and the discriminant classification, and then
	separate the unrestricted repeated-root cases from the additional
	rigidity supplied by finite order.
	
	\begin{definition}[Rank--Two Transport Cartan Matrix]
		\label{def:Rank--Two-Transport-Cartan-Matrix}
		Let $\Theta_{ij}=A_{ij}A_{ji}$. The associated rank--two transport
		Cartan matrix is
		$\mathcal C_{ij}=\begin{pmatrix}2A_{ii}&2A_{ij}\\
			2A_{ji}&2A_{jj}\end{pmatrix}$. Since the normalization gives
		$A_{ii}=A_{jj}=1$, this is
		$\mathcal C_{ij}=\begin{pmatrix}2&2A_{ij}\\2A_{ji}&2\end{pmatrix}$.
		Its determinant is
		\begin{equation}
			\label{eq:cartan-determinant}
			\det(\mathcal C_{ij})=4(1-\Theta_{ij}).
		\end{equation}
	\end{definition}
	
	\begin{remark}
		The sign in the off-diagonal entries is fixed by the Weyl-group
		compatibility established in Section~\ref{subsec:Finite-Reflection-Groups}:
		there $A_{ij}=\frac12\langle\alpha_j,\alpha_i^\vee\rangle$.
		Thus the classical Cartan entry is $2A_{ij}$, not $-2A_{ij}$.
		The determinant formula is unchanged by this sign correction because
		it depends on the product of the off-diagonal entries.
	\end{remark}
	
	\begin{remark}[Discriminant Formula]
		In the involutive case $\zeta_i=\zeta_j=-1$, the characteristic
		polynomial is $p_{ij}(t)=t^2-(4\Theta_{ij}-2)t+1$. Its discriminant is
		\begin{equation}
			\label{eq:discriminant-formula}
			\Delta=(4\Theta_{ij}-2)^2-4=16\Theta_{ij}(\Theta_{ij}-1).
		\end{equation}
		Assume that $\Theta_{ij}\in\mathbb R$. If
		$0<\Theta_{ij}<1$, then $\Delta<0$ and the eigenvalues are a
		complex-conjugate pair on the unit circle; this is the elliptic case.
		If $\Theta_{ij}<0$ or $\Theta_{ij}>1$, then $\Delta>0$ and the
		eigenvalues are distinct real reciprocals; this is the hyperbolic case.
		If $\Theta_{ij}=0$ or $\Theta_{ij}=1$, the characteristic
		polynomial has a repeated eigenvalue, respectively $-1$ or $1$.
		At $\Theta_{ij}=0$, the operator may be scalar or have a nontrivial
		Jordan part. At $\Theta_{ij}=1$, transversality excludes the scalar
		identity, so the operator is necessarily a nontrivial unipotent.
		In this two-dimensional projective setting, that non-scalar repeated-root
		case is parabolic.
	\end{remark}
	
	\begin{definition}[Rank--Two Transport Type]
		Assume $\Theta_{ij}\in\mathbb R$ and $\zeta_i=\zeta_j=-1$.
		The reduced transport operator is called \emph{elliptic} if
		$0<\Theta_{ij}<1$, \emph{hyperbolic} if
		$\Theta_{ij}<0$ or $\Theta_{ij}>1$, and of
		\emph{repeated-root boundary type} if $\Theta_{ij}\in\{0,1\}$.
		A repeated-root boundary operator is called \emph{parabolic} only
		when it is non-scalar (equivalently, when its Jordan part is nonzero).
	\end{definition}
	
	\begin{remark}[Relation with Classical Cartan Type]
		When the transport Cartan matrix satisfies the usual real generalized
		Cartan sign and symmetrizability conditions, its off-diagonal product
		is nonnegative. In that setting, $0<\Theta_{ij}<1$ is the
		irreducible positive-definite rank--two range, $\Theta_{ij}=1$ is
		the degenerate boundary, and $\Theta_{ij}>1$ is the indefinite
		range. The value $\Theta_{ij}=0$
		corresponds to the reducible rank--two value when both off-diagonal
		coefficients vanish. Values $\Theta_{ij}<0$ are outside the usual
		real symmetrizable Cartan setting, although their reduced dynamics
		are hyperbolic by the discriminant test. The real sign and
		symmetrizability conditions alone do not imply that the associated
		reflection product has finite order. The classical finite/affine/
		indefinite generalized-Cartan classification additionally uses the
		usual integrality conditions on the Cartan entries. Thus the spectral
		trichotomy and the classical Cartan-type terminology should not be
		identified without the relevant sign, symmetrizability, and
		integrality hypotheses.
	\end{remark}
	
	\begin{theorem}[Rank--Two Classification]
		\label{thm:rank-two-classification}
		Assume that the transverse rank--two quotient and its
		characteristic polynomial are supplied either by the
		single--morphism hypotheses of the Rank--Two Reduction Theorem
		and the Rank--Two Characteristic Polynomial Theorem, or by the
		mixed--chart hypotheses of
		Proposition~\ref{prop:mixed-rank-two-characteristic-polynomial}.
		In either setting, suppose that
		$\zeta_i=\zeta_j=-1$ and $\Theta_{ij}\in\mathbb R$.
		Then the spectral type of $\overline M_{ij}$ is determined by
		$\Theta_{ij}$ as follows:
		\begin{center}
			\renewcommand{\arraystretch}{1.2}
			\begin{tabular}{c|c|c|c|c}
				\toprule
				$\Theta_{ij}$ & $\det(\mathcal C_{ij})$ & $\Delta$
				& Eigenvalue type & Projective type \\
				\midrule
				$0<\Theta_{ij}<1$ & $0<\det(\mathcal C_{ij})<4$ & $\Delta<0$
				& unit-circle pair & elliptic \\
				$\Theta_{ij}=1$ & $\det(\mathcal C_{ij})=0$ & $\Delta=0$
				& repeated $+1$ & nontrivial Jordan (parabolic) \\
				$\Theta_{ij}=0$ & $\det(\mathcal C_{ij})=4$ & $\Delta=0$
				& repeated $-1$ & projective identity or parabolic \\
				$\Theta_{ij}>1$ & $\det(\mathcal C_{ij})<0$ & $\Delta>0$
				& distinct real reciprocal pair & hyperbolic \\
				$\Theta_{ij}<0$ & $\det(\mathcal C_{ij})>4$ & $\Delta>0$
				& distinct real reciprocal pair & hyperbolic \\
				\bottomrule
			\end{tabular}
		\end{center}
		At $\Theta_{ij}=0$, the repeated-root operator is either the scalar
		$-I$ or a nontrivial Jordan block. At $\Theta_{ij}=1$, the scalar
		identity is excluded by transversality, so the operator is necessarily
		a nontrivial Jordan block.
	\end{theorem}
	
	\begin{proof}
		The determinant identity gives
		$\det(\mathcal C_{ij})=4(1-\Theta_{ij})$, while the discriminant
		identity gives $\Delta=16\Theta_{ij}(\Theta_{ij}-1)$. Hence
		$\Delta<0$ precisely for $0<\Theta_{ij}<1$,
		$\Delta=0$ precisely for $\Theta_{ij}\in\{0,1\}$, and
		$\Delta>0$ precisely for $\Theta_{ij}<0$ or $\Theta_{ij}>1$.
		Because the characteristic polynomial has real coefficients and
		constant term $1$, the negative-discriminant case gives conjugate
		unit-modulus roots, while the positive-discriminant case gives
		distinct real reciprocal roots. At the two zero-discriminant values, the repeated root is $-1$
		when $\Theta_{ij}=0$ and $1$ when $\Theta_{ij}=1$. At
		$\Theta_{ij}=0$, Cayley--Hamilton gives the scalar-or-Jordan
		alternative. At $\Theta_{ij}=1$, the scalar identity is excluded
		by transversality: the fixed lines of the induced involutions
		$\overline R_i$ and $\overline R_j$ are distinct. If
		$\overline R_i\overline R_j=I$, then
		$\overline R_j=\overline R_i^{-1}=\overline R_i$, a
		contradiction. Thus the repeated-root operator at $\Theta_{ij}=1$
		is necessarily non-scalar and has a nontrivial Jordan part.
	\end{proof}

	\begin{corollary}[Transverse Unipotent Boundary]
		\label{cor:transverse-unipotent-boundary}
		Under the hypotheses of the Rank--Two Classification Theorem, if
		$\Theta_{ij}=1$, then
		$\overline M_{ij}=I+N$ for a nonzero nilpotent operator
		$N$ with $N^2=0$. In particular, $\overline M_{ij}$ has
		infinite order.
	\end{corollary}

	\begin{proof}
		The characteristic polynomial is $(t-1)^2$. As shown in the
		proof of the Rank--Two Classification Theorem, transversality
		forces $\overline M_{ij}\ne I$. The Cayley--Hamilton theorem
		therefore gives $N:=\overline M_{ij}-I\ne0$ and $N^2=0$.
		For every positive integer $q$, the binomial formula gives
		$(I+N)^q=I+qN\ne I$ over $\mathbb C$, so the operator
		has infinite order.
	\end{proof}
	
	\begin{corollary}[Cartan Reconstruction]
		\label{cor:Cartan-Reconstruction}
		In the involutive case $\zeta_i=\zeta_j=-1$, the determinant of the
		transport Cartan matrix uniquely determines $\Theta_{ij}$ and hence
		the characteristic polynomial and spectrum of $\overline M_{ij}$.
	\end{corollary}
	
	\begin{proof}
		The determinant identity gives
		$\Theta_{ij}=1-\frac14\det(\mathcal C_{ij})$.
		The characteristic polynomial is then determined by
		$\Theta_{ij}$ through the Rank--Two Characteristic Polynomial
		Theorem.
	\end{proof}

	\subsubsection{Finite-Order Spectral Classification}
	\label{subsec:Finite-Order-Spectral-Classification}
	
	\begin{theorem}[Finite-Order Spectral Classification]
		\label{thm:Finite-Order-Spectral-Classification}
		Assume $\zeta_i=\zeta_j=-1$ and that the reduced operator
		$\overline M_{ij}$ has exact finite order $m_{ij}\geq2$. Then its eigenvalues
		are $e^{2\pi i k/m_{ij}}$ and $e^{-2\pi i k/m_{ij}}$ for an integer
		$k$ with $1\leq k<m_{ij}$ and $\gcd(k,m_{ij})=1$. Moreover,
		$\Theta_{ij}=\cos^2(\pi k/m_{ij})$.
	\end{theorem}
	
	\begin{proof}
		A finite-order complex matrix is diagonalizable, since its minimal
		polynomial divides $t^{m_{ij}}-1$, which has no repeated roots.
		The determinant of $\overline M_{ij}$ is $1$, so its eigenvalues
		are $\lambda$ and $\lambda^{-1}$. Since the matrix has exact
		order $m_{ij}$, $\lambda$ is a primitive $m_{ij}$-th root of unity.
		Thus $\lambda=e^{2\pi i k/m_{ij}}$ for some $k$ coprime to
		$m_{ij}$. The trace formula gives
		$-2+4\Theta_{ij}=2\cos(2\pi k/m_{ij})$, and therefore
		$\Theta_{ij}=\cos^2(\pi k/m_{ij})$.
	\end{proof}
	
	\begin{proposition}[Repeated-Root Rigidity Under Finite Order]
		\label{prop:finite-order-repeated-root-rigidity}
		Assume the hypotheses of the Rank--Two Classification Theorem and
		that $\overline M_{ij}$ has finite order. If $\Theta_{ij}=0$,
		then $\overline M_{ij}=-I$. The value $\Theta_{ij}=1$
		cannot occur.
	\end{proposition}

	\begin{proof}
		At $\Theta_{ij}=0$, the characteristic polynomial is
		$(t+1)^2$. A finite-order complex matrix is diagonalizable, so
		the repeated eigenvalue forces $\overline M_{ij}=-I$.
		At $\Theta_{ij}=1$, finite order and the characteristic
		polynomial $(t-1)^2$ would force $\overline M_{ij}=I$.
		This contradicts
		\Cref{cor:transverse-unipotent-boundary}, which shows that the
		transverse product is nontrivially unipotent. Hence
		$\Theta_{ij}=1$ is impossible for a finite-order reduced product.
	\end{proof}
	
	\begin{corollary}[Finite Order from a Common Finite Morphism]
		\label{cor:common-finite-morphism-rank-two-order}
		Let $Q:X\to Y$ be a finite morphism of smooth irreducible complex
		varieties. Suppose two selected local transport germs
		$\tau_i,\tau_j$ fix the same point $x$, preserve $Q$, and have
		rank-one reflection differentials $R_i,R_j$ (eigenvalues $1$ on
		their fixed hyperplanes and $-1$ on the transverse lines). If the
		corresponding root walls are distinct and transverse at $x$, then
		$\overline M_{ij}$ has finite order $m_{ij}\geq2$. Consequently,
		$\Theta_{ij}=\cos^2(\pi k/m_{ij})$ for some $k$ coprime to
		$m_{ij}$, with $1\leq k<m_{ij}$. In particular,
		$\Theta_{ij}<1$, so the $+1$ repeated-root boundary cannot occur
		for a transverse pair arising from one common finite morphism.
	\end{corollary}
	
	\begin{proof}
		The reflection differentials are invertible, so the inverse function
		theorem makes $\tau_i$ and $\tau_j$ biholomorphic germs fixing $x$.
		Since they preserve $Q$, both belong to the finite group
		$G_x(Q)$ of local symmetry germs from
		Proposition~\ref{prop:local-finiteness-transport-germs}.
		Their product therefore has finite order, and so do its differential
		$M_{ij}=R_iR_j$ and the induced quotient operator
		$\overline M_{ij}$. Thus $m_{ij}<\infty$.
		
		It remains to exclude $\overline M_{ij}=I$. Suppose it were the
		identity and put $N:=M_{ij}-I$. Since $M_{ij}$ fixes $F_{ij}$
		pointwise and induces the identity on the quotient, $N$ maps $T_xX$
		into $F_{ij}$ and vanishes on $F_{ij}$; hence $N^2=0$. If
		$M_{ij}^q=I$ for some positive integer $q$, then
		$(I+N)^q=I+qN=I$. Over $\mathbb C$ this implies $N=0$, so
		$R_iR_j=I$. Since $R_i^2=R_j^2=I$, this gives $R_i=R_j$ and hence
		their fixed hyperplanes coincide, contrary to transversality of the
		two root walls. Therefore $\overline M_{ij}\ne I$, and its order
		satisfies $m_{ij}\geq2$. The finite-order spectral classification
		now gives the asserted value of $\Theta_{ij}$. Thus the $+1$
		unipotent boundary in the general transverse classification belongs
		to the mixed--chart setting and cannot occur for two branches
		preserving one common finite morphism.
	\end{proof}

	\begin{proposition}[Strict Cartan Bound for a Common Finite Morphism]
		\label{prop:strict-theta-common-morphism}
		Let $Q:X\to Y$ be a finite morphism of smooth irreducible complex
		varieties. Suppose that two selected local transport self-germs
		$\tau_i,\tau_j$ fix the same point $x$, preserve $Q$, and have
		rank-one deviations from the identity with transverse eigenvalues
		$\zeta_i,\zeta_j\ne1$. Assume that their fixed hyperplanes are
		distinct and transverse at $x$. For the normalized transport data
		and invariant $\Theta_{ij}$ defined above, one has
		\[
		\boxed{0\leq\Theta_{ij}<1.}
		\]
		No involutivity or reality assumption on $\zeta_i,\zeta_j$ is
		required.
	\end{proposition}

	\begin{proof}
		By Proposition~\ref{prop:local-finiteness-transport-germs},
		$G_x(Q)$ is a finite group, and its differential representation
		$\rho:G_x(Q)\to\operatorname{GL}(T_xX)$ is faithful by
		Lemma~\ref{lem:faithful-local-differential}. Choose any positive-
		definite Hermitian form $h_0$ on $T_xX$ and average it over the
		finite group:
		\[
		h(u,v)=\frac{1}{|G_x(Q)|}
		\sum_{g\in G_x(Q)}h_0(\rho(g)u,\rho(g)v).
		\]
		Then $h$ is positive definite and $G_x(Q)$-invariant. In
		particular, both $R_i=d(\tau_i)_x$ and $R_j=d(\tau_j)_x$
		are unitary with respect to $h$.

		Put $C_k=\operatorname{Im}(I-R_k)$ for $k=i,j$. Since $R_k$
		is unitary and has a rank-one deviation from the identity, its
		fixed hyperplane and its transverse eigenline $C_k$ are
		$h$-orthogonal. The rank-one projector
		\[
		\Pi_k=\frac{I-R_k}{1-\zeta_k}
		\]
		is therefore the $h$-orthogonal projection onto $C_k$. Choose
		nonzero $u_k\in C_k$ and, taking $h$ linear in its first
		argument, choose the normalized conormal
		\[
		\alpha_k(v)=\frac{h(v,u_k)}{h(u_k,u_k)},
		\qquad \nu_k=u_k.
		\]
		These choices satisfy $\alpha_k(\nu_k)=1$, so the rescaling
		invariance of $\Theta_{ij}$ allows us to use them in its
		calculation. The rank-one trace identity gives
		\[
		\Theta_{ij}=\operatorname{tr}(\Pi_i\Pi_j)
		=\frac{|h(u_i,u_j)|^2}
		{h(u_i,u_i)h(u_j,u_j)}.
		\]
		The Cauchy--Schwarz inequality implies
		$0\leq\Theta_{ij}\leq1$. Equality holds precisely when
		$C_i=C_j$. Since each fixed hyperplane is the $h$-orthogonal
		complement of its transverse line, this would make the two fixed
		hyperplanes equal, contrary to their transversality. Thus
		$\Theta_{ij}<1$.
	\end{proof}

	\begin{corollary}[Primitive Transport Type]
		\label{cor:Primitive-Transport-Type}
		Under the hypotheses of the Finite-Order Spectral Classification
		Theorem, if
		$\operatorname{Spec}(\overline M_{ij})=
		\{e^{2\pi i/m_{ij}},e^{-2\pi i/m_{ij}}\}$,
		then $\Theta_{ij}=\cos^2(\pi/m_{ij})$.
	\end{corollary}
	
	The primitive finite-order values include
	\begin{center}
		\renewcommand{\arraystretch}{1.2}
		\setlength{\tabcolsep}{10pt}
		\begin{tabular}{r|ccccccc}
			\toprule
			$m_{ij}$ & $2$ & $3$ & $4$ & $5$ & $6$ & $7$ & $8$ \\
			\midrule
			$\Theta_{ij}$ & $0$ & $\frac14$ & $\frac12$
			& $\frac{3+\sqrt5}{8}$ & $\frac34$
			& $\cos^2(\pi/7)$ & $\frac{2+\sqrt2}{4}$ \\
			\bottomrule
		\end{tabular}
	\end{center}
	The values $m_{ij}=2,3,4,6$ give the crystallographic rank-two
	possibilities. The values $m_{ij}=5,7,8,\ldots$ give non-crystallographic
	dihedral values; they are not new Coxeter groups, but they are
	additional possibilities within the transport framework.
	
	\begin{remark}[The Pentagonal Transport Type]
		For $m_{ij}=5$, one has
		$\Theta_{ij}=\cos^2(\pi/5)=(3+\sqrt5)/8$. This is the first
		non-crystallographic dihedral value in the sequence.
	\end{remark}
	
	\begin{corollary}[Finite-Order Transport Geometry]
		\label{cor:Finite-Order-Transport-Geometry}
		Under the hypotheses of the Finite-Order Spectral Classification
		Theorem, $0\leq\Theta_{ij}<1$. The value $\Theta_{ij}=0$ occurs
		for $m_{ij}=2$ and gives $\overline M_{ij}=-I$; for
		$m_{ij}\geq3$ one has $0<\Theta_{ij}<1$. Thus hyperbolic and
		nontrivial parabolic reduced dynamics do not occur in this finite-order
		involutive setting.
	\end{corollary}
	
	\begin{proof}
		The formula $\Theta_{ij}=\cos^2(\pi k/m_{ij})$ gives
		$0\leq\Theta_{ij}<1$, since $1\leq k<m_{ij}$. It vanishes
		exactly when $m_{ij}=2$ and $k=1$. The repeated-root rigidity
		proposition gives $\overline M_{ij}=-I$ in that case. For
		$m_{ij}\geq3$, the cosine cannot vanish for a primitive exponent
		with $k=m_{ij}/2$, so $0<\Theta_{ij}<1$. Finite-order matrices
		are diagonalizable, excluding nontrivial Jordan cases.
	\end{proof}
	
	\begin{remark}[Nonlinear Cartan Fields]
		For selected branches and normalized transport data, the individual
		coefficients $A_{ij}(x)=\alpha_j(\nu_i)$ may vary with the point
		$x\in H_{D_i}\cap H_{D_j}$. Their product $\Theta_{ij}(x)$ is
		invariant under the rescalings in Proposition~\ref{prop:Transformation-Law},
		but remains attached to the selected pair of branches at $x$.
		Accordingly, any resulting field of rank-two transport Cartan data is
		branchwise and pointwise; it should not be interpreted as a single
		global Cartan matrix without additional compatibility data.
	\end{remark}

	\subsection{Transport Graphs}
	\label{subsec:Transport-Graphs}
	
	The rank--two invariants associated with selected coincidence branches
	may be organized into graph-theoretic structures analogous to the
	classical Dynkin and Coxeter diagrams. Since the normalized transport
	data are defined at a common point and depend on selected local branches,
	the graphs in this subsection are local and branchwise. They are not,
	without additional compatibility data, global graphs attached solely to
	the irreducible components of the coincidence correspondence.
	
	\begin{definition}[Rank--Two-Admissible Local Branch System]
		\label{def:rank-two-admissible-branch-system}
		Let $x\in X$ and let $I$ be a finite index set. A family
		$\mathfrak B=((D_i,\mathfrak b_i))_{i\in I}$ is called a
		\emph{rank--two-admissible local branch system at $x$} if the following
		conditions hold:
		\begin{enumerate}[label=\emph{(\roman*)}]
			\item For every $i\in I$, $D_i\in\Phi_Q$, $x\in H_{D_i}$,
			and $\mathfrak b_i$ is a selected local transport branch through
			$(x,x)$ for which $x$ is a regular transport root point.
			\item Each selected branch satisfies the splitting condition
			$T_xX=\ker(I-R_i)\oplus\operatorname{Im}(I-R_i)$,
			where $R_i$ is its transport operator. For the present
			single-morphism branch system this condition follows from
			Proposition~\ref{prop:local-finiteness-transport-germs}; it is
			retained to display the exact linear-algebraic input.
			\item The root walls are pairwise transverse at $x$: for
			$i\ne j$,
			$T_xH_{D_i}+T_xH_{D_j}=T_xX$.
		\end{enumerate}
		For each $i$, choose normalized transport data
		$(\alpha_i,\nu_i)$ with
		$\alpha_i(\nu_i)=1$ and
		$\nu_i\in\operatorname{Im}(I-R_i)$. Since the image line is
		$R_i$-invariant, there is a scalar $\zeta_i\ne1$ such that
		$R_i\nu_i=\zeta_i\nu_i$. The inequality $\zeta_i\ne1$
		follows because otherwise the image line would lie in
		$\ker(I-R_i)$, contradicting the splitting. For $i\ne j$, put
		\[
		A_{ij}:=\alpha_j(\nu_i),\qquad
		A_{ji}:=\alpha_i(\nu_j),\qquad
		\Theta_{ij}:=A_{ij}A_{ji}.
		\]
		The scalar $\Theta_{ij}$ is independent of the rescaling of the
		normalized root data, but it remains attached to $x$ and to the selected
		pair of branches.
	\end{definition}
	
	\subsubsection{Transport Dynkin Graphs}
	\label{subsubsec:Transport-Dynkin-Graphs}
	
	\begin{definition}[Transport Dynkin Graph]
		\label{def:Transport-Dynkin-Graph}
		Let $\mathfrak B=((D_i,\mathfrak b_i))_{i\in I}$ be a
		rank--two-admissible local branch system at $x$. Its
		\emph{transport Dynkin graph} $\Gamma^{\mathrm{Dyn}}_{x,\mathfrak B}$
		is the undirected labelled graph defined by:
		\begin{enumerate}[label=\emph{(\roman*)}]
			\item the vertex set is $I$;
			\item the vertex $i$ is decorated by its transverse eigenvalue
			$\zeta_i$;
			\item distinct vertices $i,j$ are joined by an edge exactly when
			$\Theta_{ij}\ne0$, and that edge is labelled by $\Theta_{ij}$.
		\end{enumerate}
		Because every pair in $\mathfrak B$ is rank--two-admissible, a missing
		edge means precisely that $\Theta_{ij}=0$; it does not mean that the
		rank--two invariant is undefined.
	\end{definition}
	
	The vertex decorations are necessary: outside the involutive case,
	the characteristic polynomial depends on $\zeta_i$ and $\zeta_j$ as
	well as on $\Theta_{ij}$. The edge labels record the normalized
	rank--two interaction; the individual Cartan coefficients
	$A_{ij}$ and $A_{ji}$ are not recorded separately.
	
	\begin{theorem}[Dynkin Reconstruction]
		\label{thm:dynkin-reconstruction}
		Let $\mathfrak B$ be a rank--two-admissible local branch system at $x$.
		For every pair of distinct vertices $i,j$, the transport Dynkin graph,
		including its vertex decorations and edge labels, determines:
		\begin{enumerate}[label=\emph{(\roman*)}]
			\item the Cartan determinant
			$\det(\mathcal C_{ij})=4(1-\Theta_{ij})$;
			\item the characteristic polynomial of the reduced transport
			operator $\overline M_{ij}$,
			$p_{ij}(t)=t^2-
			\Bigl(\zeta_i+\zeta_j+
			(1-\zeta_i)(1-\zeta_j)\Theta_{ij}\Bigr)t
			+\zeta_i\zeta_j$;
			\item the spectrum of $\overline M_{ij}$, as the multiset of
			roots of this polynomial.
		\end{enumerate}
		Here $\Theta_{ij}$ is read from the edge label when an edge is
		present and is set equal to $0$ when the pair is not joined.
	\end{theorem}
	
	\begin{proof}
		The Cartan determinant identity follows from
		\Cref{eq:cartan-determinant}. The characteristic-polynomial formula
		is exactly the conclusion of the Rank--Two Characteristic Polynomial
		Theorem, whose hypotheses hold for every pair in
		$\mathfrak B$ by definition. The spectrum is the multiset of roots of
		the characteristic polynomial.
	\end{proof}
	
	\begin{corollary}[Spectral Type from the Transport Dynkin Graph]
		\label{cor:dynkin-graph-spectral-type}
		Assume in addition that $\zeta_i=-1$ for every vertex and that all
		edge labels are real. Then the transport Dynkin graph determines,
		for each pair, whether the reduced operator is elliptic, hyperbolic,
		or of repeated-root boundary type, according to
		\Cref{thm:rank-two-classification}.
		Because the branch system is associated with one common finite
		morphism, every selected involutive branch is a biholomorphic local
		symmetry and the common-morphism finite-order corollary applies.
		Consequently $0\leq\Theta_{ij}<1$ for every pair in this graph.
		The hyperbolic alternative and the transverse $+1$ repeated-root
		boundary are therefore excluded in this common-morphism involutive
		setting. If $\Theta_{ij}=0$, the repeated-root operator is
		necessarily $-I$; if $0<\Theta_{ij}<1$, the reduced operator is
		elliptic. The transverse $+1$ unipotent boundary described by
		\Cref{cor:transverse-unipotent-boundary} belongs instead to the
		mixed-chart setting, where no common finite-morphism symmetry group
		is available.
		\end{corollary}
	
	\begin{proof}
		For a non-edge, $\Theta_{ij}=0$ by definition; for an edge, its real
		label is $\Theta_{ij}$. Thus the graph determines the value of
		$\Theta_{ij}$ for every pair. Since all transverse eigenvalues are $-1$, the Rank--Two
		Classification Theorem gives the spectral alternatives. The strict
		common-morphism bound in
		Proposition~\ref{prop:strict-theta-common-morphism} excludes both
		$\Theta_{ij}>1$ and $\Theta_{ij}=1$. When $\Theta_{ij}=0$,
		finite-order repeated-root rigidity gives $\overline M_{ij}=-I$;
		for $0<\Theta_{ij}<1$, the reduced operator is elliptic.
	\end{proof}
	
	\subsubsection{Transport Coxeter Graphs}
	\label{subsubsec:Transport-Coxeter-Graphs}
	
	The transport Dynkin graph records the normalized rank--two invariant.
	A second graph records the exact orders of the reduced pairwise
	transport products. By definition, its label $m_{ij}$ is the exact
	order of the quotient operator $\overline M_{ij}$. For an involutive
	pair arising from one common finite morphism, this order agrees with
	the order of the full differential product and of the product germ;
	this follows from the finite local symmetry group and the argument
	given in Remark~\ref{rem:coxeter-pairwise-faithfulness} below.
	For mixed-chart pairs, no such common-morphism finite-order conclusion
	is available in general.
	
	\begin{definition}[Transport Coxeter Graph]
		\label{def:Transport-Coxeter-Graph}
		Let $\mathfrak B=((D_i,\mathfrak b_i))_{i\in I}$ be a
		rank--two-admissible local branch system at $x$. For each vertex $i$,
		write $r_i=\operatorname{ord}(R_i)$ when a positive power of $R_i$
		is the identity, and put $r_i=\infty$ otherwise. For each pair
		$i\ne j$, define the exact reduced order
		\[
		m_{ij}:=\operatorname{ord}(\overline M_{ij})\in
		\mathbb N_{\ge2}\cup\{\infty\},
		\]
		where $m_{ij}=\infty$ if no positive power of $\overline M_{ij}$
		is the identity. The \emph{transport Coxeter graph}
		$\Gamma^{\mathrm{Cox}}_{x,\mathfrak B}$ is the undirected labelled
		graph with the following data:
		\begin{enumerate}[label=\emph{(\roman*)}]
			\item the vertex set is $I$; vertex $i$ is decorated by
			$(\zeta_i,r_i)$;
			\item there is no edge between $i$ and $j$ when $m_{ij}=2$;
			\item if $3\le m_{ij}<\infty$, there is an edge labelled
			$m_{ij}$;
			\item if $m_{ij}=\infty$, there is an edge labelled $\infty$.
		\end{enumerate}
		Thus a missing edge has the usual Coxeter-graph meaning $m_{ij}=2$;
		an infinite-order pair is explicitly recorded by an $\infty$-labelled
		edge.
	\end{definition}
	
	For completeness, $m_{ij}=1$ cannot occur in this setting. By
	rank--two-admissibility, in particular the splitting and pairwise
	transversality conditions in Definition~\ref{def:rank-two-admissible-branch-system},
	the induced operators $\overline R_i$ and $\overline R_j$ on
	$T_xX/F_{ij}$ have distinct one-dimensional fixed spaces, namely
	$T_xH_{D_i}/F_{ij}$ and $T_xH_{D_j}/F_{ij}$. If
	$\overline R_i\overline R_j=I$, then
	$\overline R_j=\overline R_i^{-1}$, so the two fixed spaces would
	coincide, contrary to transversality. The order label is symmetric in
	$i,j$: if either product has finite order, both factors are invertible
	and the two products are similar; if neither has finite order, both
	labels are $\infty$.
	
	\begin{proposition}[Spectral Data and the Transport Coxeter Label]
		\label{prop:transport-Coxeter-graph}
		Let $i\ne j$ be a pair in a rank--two-admissible local branch system,
		and suppose $m_{ij}<\infty$. If the two eigenvalues of
		$\overline M_{ij}$, counted with algebraic multiplicity, are
		$\lambda,\mu$, then
		\[
		\lambda\mu=\zeta_i\zeta_j,
		\qquad
		\Theta_{ij}=
		\frac{\lambda+\mu-(\zeta_i+\zeta_j)}
		{(1-\zeta_i)(1-\zeta_j)}.
		\]
		Moreover, $\lambda$ and $\mu$ are roots of unity and
		$\operatorname{lcm}(\operatorname{ord}(\lambda),
		\operatorname{ord}(\mu))=m_{ij}$.
	\end{proposition}
	
	\begin{proof}
		The determinant and trace of $\overline M_{ij}$ are given by the
		Rank--Two Characteristic Polynomial Theorem:
		\[
		\det(\overline M_{ij})=\zeta_i\zeta_j,
		\qquad
		\operatorname{tr}(\overline M_{ij})=
		\zeta_i+\zeta_j+(1-\zeta_i)(1-\zeta_j)\Theta_{ij}.
		\]
		Since the eigenvalues are $\lambda,\mu$, their product is the
		determinant and their sum is the trace. Solving the trace identity
		for $\Theta_{ij}$ gives the formula. Finally, finite order implies
		that $\overline M_{ij}$ is diagonalizable and that both eigenvalues
		are roots of unity. Its exact order is the least common multiple of
		their multiplicative orders.
	\end{proof}
	
	\begin{corollary}[Finite-Order Values Recorded by the Transport Coxeter Graph]
		\label{cor:transport-coxeter-finite-values}
		For a pair with finite label $m_{ij}$, the vertex decorations and the
		edge label restrict the possible values of $\Theta_{ij}$ to those
		obtained from roots of unity $\lambda,\mu$ satisfying
		$\lambda\mu=\zeta_i\zeta_j$ and 
		$\operatorname{lcm}(\operatorname{ord}(\lambda),
		\operatorname{ord}(\mu))=m_{ij}$.
		If the spectral pair $\{\lambda,\mu\}$ is also specified, the value
		of $\Theta_{ij}$ is uniquely recovered by
		\Cref{prop:transport-Coxeter-graph}.
	\end{corollary}
	
	\begin{proof}
		This is an immediate consequence of
		\Cref{prop:transport-Coxeter-graph} and the exact-order condition in
		the proposition.
	\end{proof}
	
	\begin{corollary}[Classical Coxeter Recovery]
		\label{cor:transport-graph-classical-coxeter-recovery}
		Assume that $\zeta_i=-1$ for every vertex, that every reduced pair
		product has finite order $m_{ij}$, and that every pair has primitive
		transport type, namely that its two eigenvalues, counted with
		algebraic multiplicity, are
		$e^{2\pi i/m_{ij}}$ and $ e^{-2\pi i/m_{ij}}$.
		Then
		$\Theta_{ij}=\cos^2\!\left(\frac{\pi}{m_{ij}}\right)$.
		With the convention that pairs with $m_{ij}=2$ are not joined and
		pairs with $m_{ij}\ge3$ are joined by an edge labelled $m_{ij}$,
		$\Gamma^{\mathrm{Cox}}_{x,\mathfrak B}$ is the Coxeter graph
		with pairwise labels $m_{ij}$, namely the graph of the corresponding
		abstract Coxeter presentation.
	\end{corollary}
	
	\begin{proof}
		The rank--one formula gives
		$R_i(v)=v-2\alpha_i(v)\nu_i$ when $\zeta_i=-1$;
		hence $R_i^2=I$ and $r_i=2$. Its induced operator
		$\overline R_i$ is an involutive reflection on the two-dimensional
		quotient associated with a pair. The primitive spectral condition
		and the Finite--Order Spectral Classification Theorem give
		$\Theta_{ij}=\cos^2(\pi/m_{ij})$. The standard Coxeter-graph
		convention omits the $m_{ij}=2$ pairs and labels the remaining edges
		by $m_{ij}\ge3$, exactly as in the definition above.
	\end{proof}
	
	\begin{remark}[Pairwise Coxeter Labels and Faithfulness]
		\label{rem:coxeter-pairwise-faithfulness}
		When every vertex operator is an involution, each finite label
		$m_{ij}$ specifies the pairwise Coxeter relation
		$(s_is_j)^{m_{ij}}=1$, while a label $m_{ij}=\infty$ imposes
		no relation on that pair. Together with $s_i^2=1$, these data
		determine the associated abstract Coxeter presentation. In the present
		single--morphism setting these pairwise relations are realized by
		the full tangent-space differentials, not merely by their quotient
		operators. Indeed, let $M_{ij}=R_iR_j$ and let $m_{ij}$ be the
		exact order of its induced quotient operator. The common-morphism
		finite-order result shows that $M_{ij}$ has finite order. Since
		$M_{ij}$ fixes $F_{ij}$ pointwise and $\overline M_{ij}^{m_{ij}}=I$,
		the operator $S=M_{ij}^{m_{ij}}-I$ maps $T_xX$ into $F_{ij}$ and
		vanishes on $F_{ij}$; hence $S^2=0$. The operator $I+S=M_{ij}^{m_{ij}}$
		has finite order, so $(I+S)^q=I+qS=I$ for some positive integer
		$q$, forcing $S=0$ in characteristic zero. Thus the full differential
		product has exact order $m_{ij}$. Faithfulness of the differential
		representation then gives the same order for the product germ.
		The resulting representation of the abstract Coxeter presentation
		need not be faithful: additional relations involving three or more
		generators may hold. In the classical Weyl-group quotient, the
		identification of the acting group and its faithful action comes
		from the given Weyl-group action and the Galois reconstruction
		theorem, not from the pairwise labels alone.
	\end{remark}
	
		\begin{remark}[Locality and Branch Dependence of the Transport Graphs]
		The two graphs are attached to the point $x$ and to the selected local
		branch system $\mathfrak B$. The Dynkin edge labels are unchanged by
		rescaling normalized root data, but may vary with $x$ or with the
		selected branches. The Coxeter labels record exact orders on the
		corresponding two-dimensional quotients. Neither graph is asserted to
		be a global invariant of the unlabelled component set $\Phi_Q$ without
		additional compatibility of the local branch data.

		There is, however, a constancy statement for coherent holomorphic
		families in the common-morphism setting. Let $Q:X\to Y$ be a fixed
		finite morphism of smooth irreducible complex varieties, let $S$ be a
		connected complex manifold, and let $x:S\to X$ be holomorphic. Suppose
		the two selected branches are represented by coherent holomorphic
		families of local self-germs at $x(s)$, each fixing $x(s)$ and
		preserving this same morphism $Q$. Assume that their differential
		deviations from the identity have rank one and that their fixed
		hyperplanes are transverse for every $s\in S$. The differentials define
		holomorphic endomorphisms $R_i,R_j$ of the bundle $E=x^*T_X$. By
		Proposition~\ref{prop:local-finiteness-transport-germs}, each
		$R_k(s)$ has finite order and is semisimple. Thus its eigenvalues are
		$1$ with multiplicity $\dim X-1$ and one transverse eigenvalue
		$\zeta_k(s)\ne1$. The latter is the holomorphic function
		$\zeta_k=\operatorname{tr}(R_k)-(\dim X-1)$, so the normalized
		rank-one projectors $\Pi_k=(I-R_k)/(1-\zeta_k)$ are holomorphic as
		well. The rank-one trace identity gives
		$\Theta_{ij}(s)=\operatorname{tr}(\Pi_i(s)\Pi_j(s))$, hence
		$\Theta_{ij}$ is holomorphic on $S$. Proposition~\ref{prop:strict-theta-common-morphism}
		gives $0\leq\Theta_{ij}(s)<1$ at every point. Since a real-valued
		holomorphic function on a connected complex manifold is constant,
		$\Theta_{ij}$ is constant on $S$. This does not preclude variation
		for independently selected pointwise branches, mixed-chart pairs, or
		families that are not coherent and holomorphic or do not preserve one
		fixed finite morphism.
		\end{remark}

	\subsection{Root Reconstruction}
	\label{subsec:Root-Reconstruction}
	
	We now state precisely which root-geometric data can be reconstructed
	from the coincidence correspondence after selecting the local branches
	and regular points required in the preceding sections.

	\begin{theorem}[Branchwise Reconstruction of Local Transport Data]
		\label{thm:Root-Reconstruction}
		Let $Q:X\to Y$ be a finite morphism of smooth complex varieties,
		and let $\mathcal C_Q=X\times_YX$ be its coincidence
		correspondence.

		For each admissible coincidence component $D\in\Phi_Q$ and
		each point $x\in H_D$ under consideration, choose a local graph
		branch $\mathfrak b_{D,x}$ through $(x,x)$ to which the Local
		Transport Theorem (\Cref{thm:Local-Transport-Theorem}) applies.
		Assume that $x$ is a regular transport root point for this
		selected branch. Let $\mathfrak B$ denote the resulting collection
		of branch choices. For notational simplicity, we suppress the
		branch subscripts when $\mathfrak B$ is fixed.

		Then the coincidence correspondence $\mathcal C_Q$, together
		with the selected branch system $\mathfrak B$, determines the
		following local transport-geometric data:

		\begin{enumerate}[label=\emph{(\roman*)}]
			\item The selected local transport symmetry germs
			$\tau_{D,\mathfrak b_{D,x}}$;

			\item The root walls $H_D=D\cap\Delta$;

			\item The branchwise transport operators
			$R_D(x)=d\tau_{D,\mathfrak b_{D,x}}|_x$;

			\item The transport root lines
			$L_D(x)=\operatorname{Ann}(\ker(I-R_D(x)))$,
			which coincide with the conormal lines $N_x^*(H_D)$;

			\item The transport trace forms $B_{ij}(x)$ for selected
			pairs of branches at smooth common points of regular
			root walls;

			\item The normalized transport invariants
			$\Theta_{ij}(x)=A_{ij}(x)A_{ji}(x)$, whenever normalized
			transport root data are available for the selected
			branch pair;

			\item The individual transport Cartan coefficients
			$A_{ij}(x)$ and $A_{ji}(x)$, after a choice of normalized
			transport root data.
		\end{enumerate}

		The root walls and conormal root lines are intrinsic to
		$\mathcal C_Q$. The transport germs, operators, and trace
		forms are reconstructed relative to the selected branch system;
		for fixed branches, they are well-defined geometric objects,
		independent of coordinate representations; their matrix
		representatives transform naturally under coordinate changes.
		The normalized invariant $\Theta_{ij}$ is
		independent of the choice of normalized root generators, but
		may depend on the selected branch pair. The individual Cartan
		coefficients also depend on the normalization of the root data.

		Moreover, whenever the hypotheses of
		\Cref{cor:Trace-Theoretic-Expression} hold, the invariant
		$\Theta_{ij}(x)$ can be recovered directly from the
		transport trace form.
	\end{theorem}
	
	\begin{proof}
		For each selected local branch $\mathfrak b_{D,x}$, the Local
		Transport Theorem (\Cref{thm:Local-Transport-Theorem}) gives a
		unique local transport symmetry germ
		$\tau_{D,\mathfrak b_{D,x}}$ near $(x,x)$.
		Thus the germ is reconstructed from the coincidence correspondence
		together with the selected branch. We do not assert that an
		unlabelled irreducible component determines one unique transport
		germ over all of its regular locus.
		
		The associated root wall is $H_D=D\cap\Delta$.
		Differentiating the local transport symmetry at a regular
		transport root point $x$ determines the transport operator
		$R_D(x)=d\tau_{D,\mathfrak b_{D,x}}|_x$.
		
		By Definition~\ref{def:Regular-Transport-Root-Point},
		the root wall is a smooth hypersurface at $x$ and
		$\operatorname{rank}(I-R_D(x))=1$. Proposition
		\ref{prop:root-wall-fixed-eigenspace} therefore gives
		$T_xH_D=\ker(I-R_D(x))$.
		
		It follows that the transport root line satisfies
		$L_D(x)=\operatorname{Ann}(\ker(I-R_D(x)))
		=\operatorname{Ann}(T_xH_D)=N_x^*(H_D)$.
		Hence the transport root line is intrinsically determined
		by the root wall.
		
		Now consider two selected local branches associated with
		admissible coincidence components $D_i,D_j$ and a smooth
		common point $x\in H_{D_i}\cap H_{D_j}$ at which both
		transport operators are defined. These operators determine the
		transport trace form
		$B_{ij}(x)=
		\operatorname{tr}((I-R_i(x))(I-R_j(x)))$.
		By the Intrinsic Transport Trace Form Proposition
		(\Cref{prop:Intrinsic-TransportTrace-Form}), this
		quantity is independent of the choice of local coordinates.
		
		Whenever normalized transport root data
		$(\alpha_i,\nu_i)$ and $(\alpha_j,\nu_j)$ are available,
		the Cartan coefficients are given by
		$A_{ij}(x)=\alpha_j(x)(\nu_i(x))$ and
		$A_{ji}(x)=\alpha_i(x)(\nu_j(x))$.
		
		By the Transformation Law
		(\Cref{prop:Transformation-Law}), rescaling the normalized
		root data changes the individual coefficients but leaves
		their product $\Theta_{ij}(x)=A_{ij}(x)A_{ji}(x)$
		unchanged. Thus, for the fixed selected branch pair,
		$\Theta_{ij}$ is independent of the normalized root generators,
		whereas the individual coefficients $A_{ij}$ and $A_{ji}$ depend
		on their normalization. The invariant may nevertheless vary when the
		selected branches are changed.
		
		Finally, under the hypotheses of the Trace Formula
		(\Cref{cor:Trace-Theoretic-Expression}), we have
		$\Theta_{ij}(x)=
		B_{ij}(x)/((1-\zeta_i)(1-\zeta_j))$.
		Without those hypotheses, the invariant is obtained
		directly from its definition as $A_{ij}(x)A_{ji}(x)$.
		
		Therefore, all the stated local transport-geometric
		structures are reconstructed from the coincidence
		correspondence together with the selected local branch system.
		The individual Cartan coefficients require the additional choice
		of normalized transport root data.
	\end{proof}
	
	\begin{remark}[What a Reflection Arrangement Determines]
		A reflection arrangement determines the reflecting hyperplanes,
		but it does not, by itself, determine a preferred root system or
		Cartan matrix. For example, the root systems of types $B_n$ and
		$C_n$ ($n\geq3$) have the same Weyl-group action and the same
		reflecting hyperplanes
		$\{x_i=0\}$ and $\{x_i\pm x_j=0:i<j\}$, while their
		root-length assignments and associated Cartan data differ
		(see, for example, \cite[Section~3.3]{humphreys1990reflection}).
		Thus additional root, coroot, or lattice normalization is needed
		to specify individual Cartan entries. In the present setting, the
		coincidence correspondence retains the labelled deck transformations
		for a Galois quotient, while normalized pairwise invariants such as
		$\Theta_{ij}$ are independent of reciprocal rescaling. Individual
		Cartan coefficients still require the normalization stated in
		Proposition~\ref{prop:classical-cartan-compatibility}.
	\end{remark}

	\section{Finite Galois Covers and Group Reconstruction}
	\label{sec:Finite-Galois-Covers}
	
	The coincidence-root framework admits a particularly transparent description for finite Galois covers. In this setting, the irreducible components of the coincidence correspondence are indexed by deck transformations, and the associated transport maps recover the deck group. The root wall corresponding to a deck transformation is its fixed-point locus. These fixed-point loci need not all be hypersurfaces: they may have higher codimension or may be empty. The codimension-one members are the ones that recover the classical reflection arrangement in the linear reflection-group setting.
	
	\begin{theorem}[Galois Reconstruction Theorem]
		\label{thm:galois-reconstruction}
		Let $Q:X\to Y$ be a finite Galois quotient of irreducible smooth varieties over an algebraically closed field of characteristic zero, with faithful deck transformation group $G=\operatorname{Deck}(Q)$ and $Y\cong X/G$. Then the following statements hold.
		\begin{enumerate}[label=\emph{(\roman*)}]
			\item The reduced coincidence correspondence is the union of the graphs of the deck transformations: $(X\times_Y X)_{\mathrm{red}}=\bigcup_{g\in G}\Gamma_g$, where $\Gamma_g=\{(x,gx):x\in X\}$.
			\item The irreducible components of the coincidence correspondence are precisely these graphs. In particular, $\Phi_Q=\{\Gamma_g:g\in G,\ g\ne e\}$, while $\Gamma_e=\Delta_X$.
			\item For every $g\ne e$, the root wall $H_g=\Gamma_g\cap\Delta_X$ is canonically identified, under $\Delta_X\cong X$, with the fixed-point locus $\operatorname{Fix}(g)$. Thus a root wall may be empty or may have codimension greater than one.
			\item The global transport map associated with the component $D=\Gamma_g$ is $\tau_D=g$.
			\item The transport group is $\mathcal T_Q=\langle\tau_D:D\in\Phi_Q\rangle=\operatorname{Deck}(Q)$.
		\end{enumerate}
	\end{theorem}
	
	\begin{proof}
		Since $Y\cong X/G$ is the finite-group quotient, two geometric points of $X$ have the same image in $Y$ if and only if they lie in the same $G$-orbit. Hence a pair $(x_1,x_2)$ belongs to the underlying space of $X\times_YX$ precisely when $x_2=gx_1$ for some $g\in G$. This proves the equality of underlying closed sets in (i), and therefore the asserted equality of their reduced closed subvarieties.
		
		For each $g\in G$, the graph $\Gamma_g$ is closed and irreducible, since it is isomorphic to the irreducible variety $X$. The graphs are pairwise distinct because the action of $G$ is faithful. They all have dimension $\dim X$, so none is contained in another. Thus they are exactly the irreducible components of the reduced coincidence correspondence. Since $\Gamma_e=\Delta_X$, this proves (ii).
		
		For (iii), a point $x$ satisfies $(x,x)\in\Gamma_g$ if and only if $g(x)=x$. More precisely, the intersection of the graph of $g$ with the diagonal is the equalizer of $g$ and the identity, and hence is canonically isomorphic to the fixed-point locus $\operatorname{Fix}(g)$.
		
		For (iv), the first projection $p_1|_{\Gamma_g}:\Gamma_g\to X$ is an isomorphism, with inverse $x\mapsto(x,gx)$. Therefore the transport map $p_2\circ(p_1|_{\Gamma_g})^{-1}$ is exactly $g$.
		
		Finally, as $g$ ranges over the nonidentity elements of $G$, the maps $\tau_{\Gamma_g}$ range over all nonidentity deck transformations. They generate $G$, proving (v).
	\end{proof}
	
	\begin{corollary}[Deck Group Reconstruction]
		\label{cor:Deck-Group-Reconstruction}
		The coincidence correspondence of a finite Galois quotient determines its deck transformation group, together with its action on $X$.
	\end{corollary}
	\begin{proof}
		By \Cref{thm:galois-reconstruction}, each non-diagonal irreducible component is the graph of a unique deck transformation, and its transport map is that transformation. Hence the group generated by the transport maps is exactly $\operatorname{Deck}(Q)$.
	\end{proof}
	
	The incidence geometry of the coincidence components is likewise described by fixed-point loci.
	
	\begin{proposition}[Intersection Geometry]
		\label{prop:Intersection-Geometry}
		For $g,h\in G$, the intersection $\Gamma_g\cap\Gamma_h$ is canonically isomorphic to $\operatorname{Fix}(h^{-1}g)$, where the latter denotes the fixed-point locus of $h^{-1}g$.
	\end{proposition}
	\begin{proof}
		A point of $\Gamma_g\cap\Gamma_h$ has the form $(x,gx)=(x,hx)$, which is equivalent to $h^{-1}g(x)=x$. The map $x\mapsto(x,gx)$ therefore identifies the fixed-point locus of $h^{-1}g$ with the intersection of the two graphs. Its inverse is the restriction of the first projection.
	\end{proof}
	
	\begin{corollary}
		The pairwise incidence geometry of the coincidence components is determined by the fixed-point loci $\{\operatorname{Fix}(h^{-1}g):g,h\in G\}$.
	\end{corollary}
	
	\begin{remark}
		The quadratic cover of \Cref{ex:quadratic_cover} is one of the simplest illustrations of the Galois Reconstruction Theorem. More generally, a nontrivial deck transformation need not have a fixed point. For example, in an unramified Galois cover the nonidentity deck transformations act without fixed points, so their associated root walls are empty.
	\end{remark}
	
	\begin{remark}[Codimension-One Root Walls and Reflection Arrangements]
		For a finite linear group $G\subseteq\operatorname{GL}(V)$, a nonidentity element $s\in G$ whose fixed space has codimension one is called a pseudoreflection (or complex reflection). The union of the fixed hyperplanes of all such elements is the reflection arrangement of $G$. By the Galois Reconstruction Theorem, it is exactly the union of the codimension-one root walls among the family $\{H_g:g\ne e\}$. The other nontrivial elements may contribute higher-codimension fixed loci or empty root walls; they are not themselves reflections merely because they determine coincidence components.
	\end{remark}
	
	\subsection{Finite Reflection Groups}
	\label{subsec:Finite-Reflection-Groups}
	
	\begin{corollary}[Weyl Group Recovery]
		Let $\pi:V\to V/W$ be the quotient map associated with a finite Weyl group $W$. Then the following statements hold:
		\begin{enumerate}[label=\emph{(\roman*)}]
			\item The reduced coincidence correspondence is $(\mathcal C_\pi)_{\mathrm{red}}=\bigcup_{w\in W}\Gamma_w$.
			\item The union of the codimension-one root walls is the classical reflecting-hyperplane arrangement of $W$; the full family of root walls is $\{\operatorname{Fix}(w):w\in W,\ w\ne e\}$.
			\item The transport symmetries corresponding to codimension-one root walls are precisely the Weyl reflections.
			\item The transport group satisfies $\mathcal T_\pi=W$.
			\item For simple reflections $s_i,s_j$ with corresponding simple roots $\alpha_i,\alpha_j$ and coroots $\alpha_i^\vee,\alpha_j^\vee$, under the specific normalization $\alpha_i^{\mathrm{tr}}=\alpha_i/2$, $\nu_i=\alpha_i^\vee$ (and similarly for $j$), the transport Cartan coefficients satisfy $A_{ij}=\frac12\langle\alpha_j,\alpha_i^\vee\rangle$.
		\end{enumerate}
	\end{corollary}
	
	\begin{proof}
		The quotient map $V\to V/W$ is a finite Galois quotient with deck group $W$. Statements (i) and (iv) follow from \Cref{thm:galois-reconstruction}. For a finite real reflection group, a nonidentity element whose fixed space is a hyperplane is a reflection: it fixes that hyperplane pointwise and acts by $-1$ on its one-dimensional orthogonal complement. Conversely, every Weyl reflection fixes its reflecting hyperplane pointwise. This proves (ii) and (iii).
		
		For (v), let $s_i,s_j\in W$ be simple reflections. Their associated transport maps are $s_i$ and $s_j$ by \Cref{thm:galois-reconstruction}. Comparing the transport reflection formula with the classical reflection formula gives $A_{ij}=\frac12\langle\alpha_j,\alpha_i^\vee\rangle$, as required.
	\end{proof}
	
	\subsection{Complex Reflection Groups}
	\label{subsec:Complex-Reflection-Groups}
	
	\begin{corollary}[Shephard--Todd Recovery]
		\label{cor:Shephard--Todd-Recovery}
		Let $G\subseteq\operatorname{GL}(V)$ be a finite complex reflection group and let $\pi:V\to V/G$ be the quotient map. Then the coincidence correspondence determines:
		\begin{enumerate}[label=\emph{(\roman*)}]
			\item the full family of root walls $H_g\cong\operatorname{Fix}(g)$ for $g\ne e$;
			\item the complex reflection arrangement, obtained as the union of the codimension-one root walls;
			\item the transport operator $R_g=g$ for every $g\in G$, and hence the full group $G$;
			\item the normalized rank-two complex Cartan invariants associated with selected pairs of pseudoreflection components whenever the hypotheses for the normalized rank-two construction are satisfied.
		\end{enumerate}
	\end{corollary}
	
	\begin{proof}
		By the Shephard--Todd--Chevalley theorem, $V/G$ is smooth, and $\pi:V\to V/G$ is a finite Galois quotient with deck group $G$. The Galois Reconstruction Theorem identifies every non-diagonal coincidence component with the graph of a unique element of $G$, identifies its root wall with that element's fixed-point locus, and recovers its transport map. In a linear action, the codimension-one fixed spaces are precisely the fixed hyperplanes of pseudoreflections. Their union is the complex reflection arrangement, proving (i)--(iii).
		
		For a selected pseudoreflection $s$, its transport operator is $R_s=s$. Its fixed hyperplane and its one-dimensional transverse eigenspace provide the rank-one data used in the complex Cartan formula. Applying that formula to any selected pair for which the normalized rank-two hypotheses hold recovers the corresponding invariant $\Theta_{ij}=A_{ij}A_{ji}$. This proves (iv).
	\end{proof}
	
	\begin{remark}[Finite-Order Transport Symmetries]
		\label{remark:Finite-Order-Transport-Symmetries}
		Unlike the Weyl-group case, pseudoreflections in a complex reflection group need not be involutions. If $D_s=\Gamma_s$ and $s\in G$ is a pseudoreflection of order $m$, then its associated transport map satisfies $\tau_{D_s}=s$ and $\tau_{D_s}^{\,m}=\operatorname{id}$. More generally, every transport map arising from a deck transformation has finite order, whether or not that transformation is a pseudoreflection.
	\end{remark}
	
	\begin{remark}
		The coincidence correspondence determines both the full deck group and its fixed-point data. The reflection arrangement is only the codimension-one part of that data. Thus the coincidence correspondence contains strictly more information than the reflection arrangement alone.
	\end{remark}

	\subsection{Finite Monodromy and Transport Rigidity}
	\label{subsec:Finite-Monodromy-and-Transport-Rigidity}
	
	The preceding examples share a common feature: the transport symmetries arise from a finite group acting on the fibers of a finite cover. In the Galois case, \Cref{thm:galois-reconstruction} identifies the transport group with the deck transformation group itself:
	$\mathcal T_Q=\operatorname{Deck}(Q)$.
	Therefore, every transport symmetry has finite order, every transport operator has finite-order linearization at regular transport root points, and the rank-two transport invariants are constrained by the finite-order spectral classification theorem.
	
	This phenomenon is not restricted to Galois covers. More generally, whenever transport symmetries arise as global regular automorphisms preserving a finite morphism, they are forced to act faithfully on the finite \'etale fibers of the cover. As a result, the transport group embeds into a finite permutation group.
	
	The following section formalizes this observation and establishes a rigidity theorem for finite morphisms. It shows that global regular transport symmetries cannot produce continuously varying rank-two transport invariants. This degree-bounded restriction is specifically rank-two. It should be distinguished from the all-dimensional local analytic linearization theorem of Section~\ref{sec:Transport-Symmetries}: at each fixed point, the finite local symmetry group of one morphism is simultaneously linearizable in every dimension. Neither result gives a global algebraic linearization of the source or classifies how these local representations vary in families.

	\section{Finite-Monodromy Rigidity}
	\label{sec:finite-monodromy-rigidity}
	
	The finite Galois-cover case suggests a more general rigidity phenomenon.
	For a finite morphism, any global regular automorphism preserving the
	morphism acts on its fibres. The key point is that the action on a generic
	\'etale fibre is not merely faithful: it is free. Consequently, the order
	of the global transport group divides the degree of the morphism. This
	stronger divisibility statement yields a degree-dependent restriction on
	the rank-two transport invariants.
	
	Throughout this section, varieties are irreducible and defined over an
	algebraically closed field of characteristic zero. A \emph{global regular
		transport symmetry} of a finite morphism $Q:X\to Y$ means an automorphism
	$g\in\operatorname{Aut}(X)$ satisfying $Q\circ g=Q$.
	
	The results below concern global automorphisms. They complement the local
	finiteness theorem of Section~\ref{sec:Transport-Symmetries}, which
	applies to biholomorphic symmetry germs fixing a common point and
	preserving one finite morphism. The global result has the additional
	feature that the degree of the morphism bounds the possible orders.
	
	\subsection{Free Action on the Generic \'Etale Fibre}
	\label{subsec:Faithfulness-on-the-Generic-Etale-Fiber}
	
	\begin{lemma}[Free Action of Global Regular Transport Symmetries]
		\label{lem:faithfulness-transport}
		Let $Q:X\to Y$ be a finite dominant morphism of degree $d$ between
		irreducible varieties. Let $G\leq\operatorname{Aut}(X)$ be a subgroup
		such that $Q\circ g=Q$ for every $g\in G$. Choose a nonempty dense open
		subset $U\subseteq Y$ over which $Q_U:Q^{-1}(U)\to U$ is finite
		\'etale of degree $d$. Then the action of $G$ on every geometric fibre
		of $Q_U$ is free. In particular, $G$ is finite and $|G|$ divides $d$.
	\end{lemma}
	
	\begin{proof}
		The open subset $X_U:=Q^{-1}(U)$ is nonempty and open in the irreducible
		variety $X$, hence is irreducible and connected. Each $g\in G$ restricts
		to an automorphism of the connected finite \'etale cover
		$Q_U:X_U\to U$.
		
		Fix a geometric point $\bar u\to U$ and a point $\bar x$ of the geometric
		fibre $F:=Q_U^{-1}(\bar u)$. Suppose that $g(\bar x)=\bar x$. An
		automorphism of a connected finite \'etale cover is determined by its
		value at one geometric point: equivalently, the associated fundamental-
		group set is transitive, and an equivariant self-map fixing one point is
		the identity. Thus $g$ is the identity on $X_U$. Since $X_U$ is dense
		in $X$ and $g$ is a regular automorphism, it follows that $g=\mathrm{id}_X$.
		Therefore every stabilizer for the action on $F$ is trivial, so the
		action is free.
		
		The fibre $F$ has $d$ points. The free action gives an injective
		permutation representation $G\hookrightarrow\operatorname{Perm}(F)\cong
		S_d$, so $G$ is finite. Every orbit has cardinality $|G|$, and the orbits
		partition $F$. Hence $|G|$ divides $d$.
	\end{proof}
	
	\begin{theorem}[Finite-Monodromy Rigidity]
		\label{thm:finite-monodromy-rigidity}
		Let $Q:X\to Y$ be a finite dominant morphism of degree $d$ between
		irreducible varieties. Suppose that selected coincidence components
		define global regular transport symmetries
		$\tau_1,\ldots,\tau_r\in\operatorname{Aut}(X)$ satisfying
		$Q\circ\tau_i=Q$. Put
		$G_{\mathrm{tr}}:=\langle\tau_1,\ldots,\tau_r\rangle$.
		Then:
		\begin{enumerate}[label=\emph{(\roman*)}]
			\item $G_{\mathrm{tr}}$ acts freely on every geometric fibre over a
			nonempty open subset of the \'etale locus of $Q$;
			\item $G_{\mathrm{tr}}$ is finite and $|G_{\mathrm{tr}}|$ divides $d$;
			\item every element of $G_{\mathrm{tr}}$ has finite order. In
			particular, for each pair $i,j$, the global product
			$\tau_i\tau_j$ has an exact finite order $m_{ij}$ satisfying
			$m_{ij}\mid |G_{\mathrm{tr}}|$ and hence $m_{ij}\mid d$;
			\item if $x$ is fixed by both $\tau_i$ and $\tau_j$, then the
			linearized product $R_i(x)R_j(x)$ has finite order. More precisely, if
			$r_{ij}(x)$ denotes its exact order, then
			$r_{ij}(x)\mid m_{ij}$.
		\end{enumerate}
		In particular, at such a common fixed point every eigenvalue of
		$R_i(x)R_j(x)$ is a root of unity. The converse is not asserted: root-of-unity
		eigenvalues alone do not imply finite order in the presence of nontrivial
		Jordan blocks.
	\end{theorem}
	
	\begin{proof}
		Apply Lemma~\ref{lem:faithfulness-transport} to
		$G_{\mathrm{tr}}$. This gives (i) and (ii). Since $G_{\mathrm{tr}}$ is
		finite, every element has finite order. By Lagrange's theorem, the exact
		order $m_{ij}$ of $\tau_i\tau_j$ divides $|G_{\mathrm{tr}}|$, and (ii)
		gives $m_{ij}\mid d$. This proves (iii).
		
		Now suppose $x$ is fixed by both $\tau_i$ and $\tau_j$. The product
		$\tau_i\circ\tau_j$ fixes $x$, and the chain rule gives
		$d(\tau_i\circ\tau_j)_x=R_i(x)R_j(x)$. Since
		$(\tau_i\circ\tau_j)^{m_{ij}}=\mathrm{id}_X$, differentiation at $x$
		gives $(R_i(x)R_j(x))^{m_{ij}}=I$. Thus the exact order
		$r_{ij}(x)$ divides $m_{ij}$, proving (iv). Every eigenvalue is
		therefore an $m_{ij}$-th root of unity.
	\end{proof}
	
	\begin{corollary}[Finite Pairwise Order]
		\label{cor:finite-pairwise-order}
		Under the hypotheses of \Cref{thm:finite-monodromy-rigidity}, the exact
		order $m_{ij}$ of every global product $\tau_i\tau_j$ divides $d$.
		If $x$ is a common fixed point of $\tau_i$ and $\tau_j$, then the exact
		order of $R_i(x)R_j(x)$ also divides $m_{ij}$.
	\end{corollary}
	
	\begin{proof}
		This is the divisibility assertion in
		\Cref{thm:finite-monodromy-rigidity}.
	\end{proof}
	
	\begin{corollary}[Degree-Bounded Rank-Two Rigidity]
		\label{cor:rank-two-rigidity}
		Assume the hypotheses of \Cref{thm:finite-monodromy-rigidity}.
		Suppose that $\tau_i$ and $\tau_j$ are distinct involutions fixing a
		common point $x$, and that their differentials $R_i(x)$ and $R_j(x)$
		are involutive rank-one reflections: each has eigenvalue $1$ on its
		fixed hyperplane $T_xH_{D_i}$ or $T_xH_{D_j}$, respectively, and
		eigenvalue $-1$ on the transverse line. Assume that the two root walls
		are smooth hypersurfaces, distinct, and transverse at $x$, and that
		$x$ is a regular transport root point for both selected branches.
		Since the operators are involutions, the normalized pairs
		$(\alpha_i,\nu_i)$ and $(\alpha_j,\nu_j)$ are supplied by the
		Rank-One Reflection Formula
		(Lemma~\ref{lem:rank-one-reflection-formula}). Put
		$\Theta_{ij}:=\alpha_j(\nu_i)\alpha_i(\nu_j)$, and let
		$\overline M_{ij}$ be the operator induced on the two-dimensional
		quotient by the Rank-Two Reduction Theorem. If $n_{ij}$ is its exact
		order, then $n_{ij}\geq2$, $2n_{ij}\mid d$, and, over the ground
		field $\mathbb K$, there is a primitive $n_{ij}$-th root of unity
		$\xi\in\mathbb K$ such that
		\[
		\Theta_{ij}=\frac{2+\xi+\xi^{-1}}{4}.
		\]
		When $\mathbb K=\mathbb C$, writing $\xi=e^{2\pi i k/n_{ij}}$
		recovers $\Theta_{ij}=\cos^2(\pi k/n_{ij})$ with
		$1\leq k<n_{ij}$ and $\gcd(k,n_{ij})=1$. Consequently, for fixed
		degree $d$ the invariant belongs to the finite set
		\[
		\mathcal S_d^{\mathrm{sharp}}:=
		\left\{\frac{2+\xi+\xi^{-1}}4:
		\begin{array}{l}
		 n\geq2,\ 2n\mid d,\\[-2pt]
		 \xi\in\mathbb K\text{ a primitive }n\text{-th root of unity}
		\end{array}\right\}.
		\]
	\end{corollary}

	\begin{proof}
		Let $m_{ij}$ be the exact order of the global product $\tau_i\tau_j$,
		and let $r_{ij}$ be the exact order of its differential
		$M_{ij}:=R_i(x)R_j(x)$. Let $n_{ij}$ be the exact order of the induced
		quotient operator $\overline M_{ij}$. By
		\Cref{thm:finite-monodromy-rigidity}, $r_{ij}\mid m_{ij}$ and
		$m_{ij}\mid d$. Since passing to an invariant quotient cannot increase
		the order, $n_{ij}\mid r_{ij}$.

		The subgroup $H:=\langle\tau_i,\tau_j\rangle$ is dihedral of order
		$2m_{ij}$. Indeed, with $r:=\tau_i\tau_j$ one has
		$\tau_i r\tau_i=r^{-1}$, and every element of $H$ is of the form
		$r^a$ or $\tau_i r^a$ for $0\leq a<m_{ij}$. These two sets are
		disjoint because $\tau_i,\tau_j$ are distinct; hence $|H|=2m_{ij}$.
		Since every element of $H$ preserves the same finite morphism, the
		free-action lemma \Cref{lem:faithfulness-transport} gives
		$2m_{ij}=|H|\mid d$. Therefore
		$2n_{ij}\mid 2m_{ij}\mid d$.

		It remains to show that $\overline M_{ij}\neq I$. Put
		$F_{ij}:=T_xH_{D_i}\cap T_xH_{D_j}$. By the Rank-Two Reduction
		Theorem, $M_{ij}$ fixes $F_{ij}$ pointwise. If
		$\overline M_{ij}=I$, then $N:=M_{ij}-I$ maps $T_xX$ into $F_{ij}$
		and vanishes on $F_{ij}$, so $N^2=0$. Since $M_{ij}$ has finite
		order, say $M_{ij}^q=I$, we have $(I+N)^q=I+qN=I$. The ground field
		has characteristic zero, so $N=0$ and $M_{ij}=I$. As $R_i^2=R_j^2=I$,
		this implies $R_i=R_j$. Their fixed hyperplanes would then coincide,
		contrary to the assumed transversality of the two root walls. Hence
		$\overline M_{ij}\neq I$ and $n_{ij}\geq2$.

		The determinant of $\overline M_{ij}$ is $1$, because each rank-one
		reflection has determinant $-1$ and the quotient determinant is the
		determinant of the full product after removing its identity action on
		$F_{ij}$. Since $\overline M_{ij}$ has exact finite order $n_{ij}$
		over an algebraically closed field of characteristic zero, it is
		diagonalizable. Its eigenvalues are $\xi$ and $\xi^{-1}$ for a
		primitive $n_{ij}$-th root of unity $\xi\in\mathbb K$. The
		involutive rank-two trace formula is
		$\operatorname{tr}(\overline M_{ij})=4\Theta_{ij}-2$; comparing
		with $\xi+\xi^{-1}$ gives the displayed expression. This proves
		the corollary.
	\end{proof}
	
	\begin{remark}[Scope of Single-Morphism Rigidity]
		For a fixed finite morphism of degree $d$, the normalized rank-two
		invariants arising from global involutive transport symmetries with
		smooth transverse reflection walls belong to the finite set
		$\mathcal S_d^{\mathrm{sharp}}$ defined in
		\Cref{cor:rank-two-rigidity}. In particular, they cannot vary continuously while
		remaining in this class on a connected parameter space on which the
		normalized invariants vary continuously. Separately, the local symmetry
		group at each fixed point is simultaneously analytically linearizable
		in every dimension. This local statement does not imply a global
		algebraic linearization, nor does it classify the variation of local
		representations in families. It also does not constrain a collection
		of germs arising from different finite morphisms, as in the
		transport-atlas setting of the next section.
	\end{remark}

	\section{Transport Atlases and Local Finite Covers}
	\label{sec:transport-atlases-and-Local-Finite-Covers}
	
	The finite-morphism rigidity results of Sections~\ref{sec:Transport-Symmetries}
	and~\ref{sec:Rank-Two-Transport-Geometry} apply when the transport
	branches under consideration preserve one common finite morphism. In
	particular, two biholomorphic transport germs fixing the same point and
	preserving the same finite morphism belong to its finite local symmetry
	group. In the involutive rank-one setting, this gives the finite-order
	restriction on the associated rank-two invariant established in
	\Cref{cor:common-finite-morphism-rank-two-order}.
	
	We now consider collections of local finite covers in which different
	transport walls may be represented by different morphisms. In this
	situation the two selected transport germs need not preserve a common
	finite morphism, so the common-morphism finite-order argument is not
	available. The construction below is deliberately an unglued collection
	of local finite-cover data: no overlap identifications or cocycle
	conditions are imposed. This weaker structure is sufficient for
	pointwise comparisons of transport data on the common ambient variety
	whenever the required normalizations exist.
	
	\begin{definition}[Transport Atlas]
		\label{def:Transport-Atlas}
		Let $X$ be a smooth irreducible variety over an algebraically closed
		field of characteristic zero. (For the holomorphic branch constructions
		used below, take the base field to be $\mathbb C$; an analogous
		definition can be made directly for smooth complex analytic spaces with finite holomorphic maps.)  A
		\emph{transport atlas} on $X$ consists of an open cover
		$X=\bigcup_{\alpha\in A}U_\alpha$ and, for each index $\alpha$, a
		finite dominant morphism $Q_\alpha:U_\alpha\to Y_\alpha$, where
		$Y_\alpha$ is smooth and irreducible. A chart-labelled coincidence
		component is a choice of a non-diagonal irreducible component
		$D_{\alpha,i}$ of
		$\mathcal C_{Q_\alpha}=U_\alpha\times_{Y_\alpha}U_\alpha$.
		Its root wall is $H_{\alpha,i}:=D_{\alpha,i}\cap\Delta_{U_\alpha}$,
		viewed as a subvariety of $U_\alpha\subseteq X$. Whenever the local
		transport theorem applies, a selected local graph branch of
		$D_{\alpha,i}$ determines the corresponding local transport germ.
		We say that this chart-labelled wall is \emph{represented by}
		$Q_\alpha$.
	\end{definition}
	
	The term \emph{atlas} is used here in this deliberately weak sense.
	The definition records which local finite morphism represents each
	selected coincidence component; it does not include transition maps,
	identifications of coincidence correspondences on overlaps, or gluing
	conditions. In particular, a chart-labelled wall is not asserted to
	extend to a global coincidence component of a single morphism on $X$.
	
	\begin{remark}[Walls from Different Charts]
		Let $H_i$ and $H_j$ be chart-labelled transport walls meeting at
		$x\in H_i\cap H_j$. They may be represented by distinct finite
		morphisms $Q_\alpha:U_\alpha\to Y_\alpha$ and
		$Q_\beta:U_\beta\to Y_\beta$. Their selected transport germs then
		need not preserve a common finite morphism. The local finiteness theorem
		still applies to any pair that does arise from one common chart and
		consists of biholomorphic germs fixing $x$; the atlas definition does
		not remove that common-morphism hypothesis when it happens to hold.
	\end{remark}
	
	\begin{remark}[Scope of the Atlas]
		The transport atlas is a bookkeeping structure for local finite covers
		and their selected coincidence branches. We do not claim that the
		assignment sending an open subset to its finite covers is a sheaf: in
		general, restricting a finite morphism to an arbitrary open subset of
		its source need not remain a finite morphism to the same target. A
		sheaf, groupoid, or stack of finite covers would require additional
		restriction and descent data and is left for future work.
	\end{remark}
	
	\begin{proposition}[Local Survival of Transport Data]
		\label{prop:atlas-local-data}
		Let $(U_\alpha,Q_\alpha)$ be a chart of a transport atlas, let
		$D_{\alpha,i}\subseteq\mathcal C_{Q_\alpha}$ be an admissible
		coincidence component, and choose a local graph branch through
		$(x,x)\in D_{\alpha,i}\cap\Delta_{U_\alpha}$. Suppose that $x$ is a
		regular transport root point for this branch. Then the selected local
		transport germ is automatically a finite-order biholomorphic self-germ
		of $Q_\alpha$ by
		Proposition~\ref{prop:local-finiteness-transport-germs}. Its
		differential $R_i(x)$ is semisimple, and the splitting
		$T_xX=\ker(I-R_i(x))\oplus\operatorname{Im}(I-R_i(x))$
		holds. The root wall $H_i$, the conormal line
		$L_i(x)=N_x^*(H_i)$, and for every chosen nonzero
		$\alpha_i\in L_i(x)$ the unique normalized vector
		$\nu_i\in\operatorname{Im}(I-R_i(x))$ satisfying
		$\alpha_i(\nu_i)=1$ are therefore defined.
	\end{proposition}
	
	\begin{proof}
		The selected graph branch gives a holomorphic self-germ fixing $x$
		and preserving $Q_\alpha$. The automatic finite-order proposition
		therefore applies. Its differential is semisimple, so the image of
		$I-R_i(x)$ complements its kernel. Regularity and
		\Cref{prop:root-wall-fixed-eigenspace} identify the kernel with
		$T_xH_i$, hence the conormal line is $L_i(x)=N_x^*(H_i)$. The
		restriction of a nonzero conormal covector to the complementary
		image line is nonzero, giving the unique normalized vector.
	\end{proof}
	
	\begin{remark}[Automatic Splitting in Every Atlas Chart]
		Each selected root self-germ fixes its base point and preserves the
		finite morphism assigned to its own chart. It is therefore finite order
		and has a semisimple differential, even when two selected germs come
		from different charts. The product of two such germs from different
		charts need not have finite order, because they need not belong to one
		common local symmetry group.
	\end{remark}
	
	\begin{proposition}[No Automatic Common-Morphism Bound]
		\label{prop:atlas-rigidity-failure}
		The axioms of a transport atlas do not require two chart-labelled
		transport walls meeting at a point to arise from one common finite
		morphism. Consequently, the common-morphism finite-order argument does
		not, by itself, impose a finite-order restriction on the product of
		transport germs represented by distinct charts.
	\end{proposition}
	
	\begin{proof}
		By definition, each chart-labelled wall is extracted from the
		coincidence correspondence of its assigned morphism. No compatibility
		or identification of these morphisms on overlaps is part of the atlas
		data. If the two walls are represented in different charts, the atlas data
		by themselves do not identify a common finite morphism preserved by
		both selected germs. The local and global finite-monodromy theorems
		therefore cannot be applied to their product solely from the atlas
		axioms. Additional structure could of course identify the two germs
		with symmetries of a common finite morphism.
		The example below shows that this absence can lead to a genuinely
		variable rank-two invariant.
	\end{proof}
	
	\begin{definition}[Regular Transport Intersection Point]
		\label{def:Regular-Transport-Intersection-Point}
		Let $H_i$ and $H_j$ be chart-labelled transport walls represented by
		charts $(U_i,Q_i)$ and $(U_j,Q_j)$, respectively, with selected local
		branches through $x\in H_i\cap H_j$. We call $x$ a
		\emph{regular transport intersection point} if $x\in U_i\cap U_j$
		and $x$ is a regular transport root point for each selected branch.
	\end{definition}
	
	\begin{definition}[Cartan-Regular Atlas Intersection Point]
		\label{def:Cartan-Regular-Atlas-Intersection-Point}
		A regular transport intersection point $x$ for selected branches $i$
		and $j$ is called \emph{Cartan-regular} if both splitting conditions
		$T_xX=\ker(I-R_i(x))\oplus\operatorname{Im}(I-R_i(x))$ and
		$T_xX=\ker(I-R_j(x))\oplus\operatorname{Im}(I-R_j(x))$
		hold.
	\end{definition}
	
	\begin{proposition}[Mixed Cartan Coefficients]
		\label{prop:mixed-cartan}
		Let $H_i$ and $H_j$ be transport walls represented by possibly
		distinct charts of a transport atlas, and let $x\in H_i\cap H_j$ be
		a Cartan-regular atlas intersection point. Choose nonzero transport
		root covectors $\alpha_i\in N_x^*(H_i)$ and
		$\alpha_j\in N_x^*(H_j)$. There are unique vectors
		$\nu_i\in\operatorname{Im}(I-R_i(x))$ and
		$\nu_j\in\operatorname{Im}(I-R_j(x))$ satisfying
		$\alpha_i(\nu_i)=\alpha_j(\nu_j)=1$. The mixed coefficients
		$A_{ij}(x):=\alpha_j(\nu_i)$ and
		$A_{ji}(x):=\alpha_i(\nu_j)$ are therefore defined for these
		normalized choices. Their product
		$\Theta_{ij}(x):=A_{ij}(x)A_{ji}(x)$ is independent of rescaling
		the two nonzero root covectors.
	\end{proposition}
	
	\begin{proof}
		Cartan-regularity supplies the two splittings. Applying
		\Cref{prop:atlas-local-data} to each branch gives the unique normalized
		vectors. Since both charts are open subsets of the same smooth variety
		$X$ and $x$ belongs to their intersection, all four objects lie in the
		same tangent-cotangent fibre at $x$. Thus the two pairings are defined.
		If $\alpha_i$ is replaced by $c_i\alpha_i$, its normalized vector is
		replaced by $c_i^{-1}\nu_i$, and similarly for $j$. Hence
		$A_{ij}$ is multiplied by $c_j/c_i$ and $A_{ji}$ by $c_i/c_j$;
		their product is unchanged.
	\end{proof}
	
	\begin{corollary}[Mixed Rank-Two Invariant]
		\label{cor:mixed-theta}
		Under the hypotheses of \Cref{prop:mixed-cartan}, the scalar
		$\Theta_{ij}(x)$ is well-defined independently of the choices of
		nonzero conormal generators. Along a locus of Cartan-regular
		intersection points, suppose that the normalized pairs
		$(\alpha_i,\nu_i)$ and $(\alpha_j,\nu_j)$ can be chosen
		locally and continuously (respectively, locally and holomorphically).
		Then
		$\Theta_{ij}$ defines a continuous (respectively, holomorphic)
		function there. This hypothesis includes the variation of the
		rank-one image lines as well as the conormal root lines. Since the
		invariant is independent of generator rescaling, these local
		functions agree on overlaps.
	\end{corollary}
	
	\begin{proof}
		The independence of conormal rescaling is part of
		\Cref{prop:mixed-cartan}. By hypothesis, both normalized pairs
		vary continuously (respectively, holomorphically). Therefore the
		pairings defining $A_{ij}$ and $A_{ji}$, and hence their product,
		have the same regularity.
	\end{proof}

	\begin{proposition}[Rank--Two Formula for Branches from Different Charts]
		\label{prop:mixed-rank-two-characteristic-polynomial}
		Let $X$ be a smooth complex variety, and let
		$(U_i,Q_i:U_i\to Y_i)$ and $(U_j,Q_j:U_j\to Y_j)$ be
		two charts of a weak transport atlas, with
		$x\in U_i\cap U_j$. Choose selected local transport branches
		fixing $x$, with regular rank-one transport differentials
		$R_i,R_j$ and smooth root walls $H_i,H_j$ meeting transversely
		at $x$. Assume that both splitting conditions hold, and choose
		normalized data $(\alpha_i,\nu_i)$ and $(\alpha_j,\nu_j)$,
		with $\alpha_i\in N_x^*H_i$, $\alpha_j\in N_x^*H_j$,
		$\nu_i\in\operatorname{Im}(I-R_i)$,
		$\nu_j\in\operatorname{Im}(I-R_j)$,
		$\alpha_i(\nu_i)=\alpha_j(\nu_j)=1$, and
		$R_i\nu_i=\zeta_i\nu_i$,
		$R_j\nu_j=\zeta_j\nu_j$. Set
		$A_{ij}=\alpha_j(\nu_i)$,
		$A_{ji}=\alpha_i(\nu_j)$, and
		$\Theta_{ij}=A_{ij}A_{ji}$. Then
		$F_{ij}:=T_xH_i\cap T_xH_j$ is fixed pointwise by
		$M_{ij}:=R_iR_j$, and the induced operator
		$\overline M_{ij}$ on $T_xX/F_{ij}$ has characteristic
		polynomial
		\[
		 p_{ij}(t)=t^2-
		\Bigl(\zeta_i+\zeta_j+(1-\zeta_i)(1-\zeta_j)
		\Theta_{ij}\Bigr)t+\zeta_i\zeta_j.
		\]
	\end{proposition}

	\begin{proof}
		Put $n=\dim X$ and $F_{ij}=T_xH_i\cap T_xH_j$. Transversality
		gives $\dim F_{ij}=n-2$. Both operators fix their respective wall
		tangent hyperplanes pointwise, so their product fixes $F_{ij}$ and
		induces an operator on $T_xX/F_{ij}$. Define
		$\Pi_i=\nu_i\otimes\alpha_i$ and
		$\Pi_j=\nu_j\otimes\alpha_j$. The splitting conditions
		and the normalization give
		$R_i=I-(1-\zeta_i)\Pi_i$ and
		$R_j=I-(1-\zeta_j)\Pi_j$, with
		$\det(R_i)=\zeta_i$ and $\det(R_j)=\zeta_j$,
		$\operatorname{tr}(\Pi_i)=\operatorname{tr}(\Pi_j)=1$ and
		$\operatorname{tr}(\Pi_i\Pi_j)
		=\alpha_i(\nu_j)\alpha_j(\nu_i)=\Theta_{ij}$.
		Consequently, writing $a_i=1-\zeta_i$ and
		$a_j=1-\zeta_j$, we obtain
		\[
		\operatorname{tr}(M_{ij})
		=n-a_i-a_j+a_ia_j\Theta_{ij},
		\qquad
		\det(M_{ij})=\zeta_i\zeta_j.
		\]
		Since $M_{ij}$ is the identity on $F_{ij}$, the quotient trace is
		the full trace minus $n-2$, while the quotient determinant equals
		the full determinant. Thus
		\[
		\operatorname{tr}(\overline M_{ij})
		=\zeta_i+\zeta_j+(1-\zeta_i)(1-\zeta_j)\Theta_{ij},
		\qquad
		\det(\overline M_{ij})=\zeta_i\zeta_j.
		\]
		These are precisely the coefficients of the stated characteristic
		polynomial. The argument uses only the pointwise linear data on the
		common source and does not require the branches to preserve one
		common finite morphism.
	\end{proof}

	\begin{remark}[Scope of the Mixed Rank--Two Formula]
		The proposition extends the characteristic-polynomial formula to
		branches supplied by different finite morphisms on a common source.
		It does not imply that the two germs preserve one common finite
		morphism or belong to one common local symmetry group. Accordingly,
		the finite-monodromy and finite-order conclusions for a single
		morphism cannot be applied to a mixed-chart pair without an
		additional common-morphism hypothesis.
	\end{remark}

	\begin{example}[A Mixed-Chart Jordan Extension Not Detected by the Quotient]
		\label{ex:mixed-chart-jordan-extension}
		This example shows that the two-dimensional quotient determines the
		pairwise spectral data but need not determine the full Jordan form of
		the tangent-space product. Let $X=\mathbb A^3_{\mathbb C}$ and,
		for $c\in\mathbb C$, define
		\[
		Q_1(x,y,z)=(x^2,y-x,z),\qquad
		Q_2(x,y,z)=(y^2,x-y,z-cy).
		\]
		Both are finite of degree two: each is a linear polynomial
		automorphism followed by a coordinate-square map. Their nontrivial
		deck involutions are
		\[
		\tau_1(x,y,z)=(-x,y-2x,z),\qquad
		\tau_2(x,y,z)=(x-2y,-y,z-2cy).
		\]
		The corresponding root walls are $H_1=\{x=0\}$ and
		$H_2=\{y=0\}$, which meet transversely along
		$\Sigma=\{x=y=0\}$. At the origin, take
		$\alpha_1=dx$, $\alpha_2=dy$ and
		$\nu_1=(1,1,0)$, $\nu_2=(1,1,c)$. These vectors span
		$\operatorname{Im}(I-R_1)$ and $\operatorname{Im}(I-R_2)$,
		respectively, and satisfy the required self-normalizations. Hence
		$A_{12}=A_{21}=1$ and $\Theta_{12}=1$.
		The differentials and their product are
		\[
		R_1=\begin{pmatrix}-1&0&0\\-2&1&0\\0&0&1\end{pmatrix},
		\qquad
		R_2=\begin{pmatrix}1&-2&0\\0&-1&0\\0&-2c&1\end{pmatrix},
		\]
		\[
		M_c=R_1R_2=
		\begin{pmatrix}-1&2&0\\-2&3&0\\0&-2c&1\end{pmatrix}.
		\]
		The common fixed subspace is $F=\mathbb C e_3$, and the induced
		quotient operator is, independently of $c$,
		$\left(\begin{smallmatrix}-1&2\\-2&3\end{smallmatrix}\right)$,
		a nontrivial unipotent with characteristic polynomial $(t-1)^2$.
		On the full tangent space put $N_c=M_c-I$. Direct calculation gives
		$N_0^2=0$, whereas for $c\ne0$ one has
		$N_c^2\ne0$ and $N_c^3=0$. Thus the full Jordan forms are
		$J_2(1)\oplus J_1(1)$ for $c=0$ and $J_3(1)$ for $c\ne0$,
		although the quotient operator is the same. Each individual germ
		is an involution preserving its own finite morphism; their product
		need not preserve a common finite morphism. This is therefore a
		mixed-chart phenomenon and does not contradict the finite-order
		rigidity for a single morphism.
	\end{example}
	
	\subsection{A Variable Transport Invariant}
	\label{sec:atlas-variable-example}
	
	We now give an explicit transport atlas in which the mixed rank-two
	invariant varies continuously. The two walls will be represented by
	different finite morphisms, so this example does not contradict the
	common-finite-morphism rigidity results.
	
	\begin{example}[A Variable Rank--Two Transport Invariant]
		\label{ex:transport-atlas-z2}
		Let $X=\mathbb A^3_{\mathbb C}$ with coordinates $(x,y,z)$. All maps
		below are defined over $\mathbb R$; the real parameter line will be
		used when discussing the elliptic, parabolic, and hyperbolic ranges.
		Consider $Q_1:X\to\mathbb A^3$ given by
		$Q_1(x,y,z)=(x^2,y-zx,z)$. This is finite of degree two: it is the
		composition of the polynomial automorphism
		$(x,y,z)\mapsto(x,y-zx,z)$ with the finite map
		$(u,v,w)\mapsto(u^2,v,w)$.
		
		The coincidence equations for $Q_1$ are
		$x_1^2=x_2^2$, $y_1-z_1x_1=y_2-z_2x_2$, and $z_1=z_2$.
		Besides the diagonal, the non-diagonal component is the graph
		$D_1=\{x_2=-x_1,\ y_2=y_1-2zx_1,\ z_2=z_1\}$ of the involution
		$\tau_1(x,y,z)=(-x,y-2zx,z)$. Direct substitution gives
		$Q_1\circ\tau_1=Q_1$. Its fixed-point locus, and hence its root wall,
		is $H_1=\{x=0\}$.
		
		Similarly, let $Q_2:X\to\mathbb A^3$ be
		$Q_2(x,y,z)=(x-zy,y^2,z)$. It is finite of degree two, since it is
		the composition of the polynomial automorphism
		$(x,y,z)\mapsto(x-zy,y,z)$ with $(u,v,w)\mapsto(u,v^2,w)$.
		Its non-diagonal coincidence component is the graph
		$D_2=\{y_2=-y_1,\ x_2=x_1-2zy_1,\ z_2=z_1\}$ of the involution
		$\tau_2(x,y,z)=(x-2zy,-y,z)$. Again, direct substitution gives
		$Q_2\circ\tau_2=Q_2$, and its root wall is $H_2=\{y=0\}$.
		
		Thus the two chart entries, both with underlying open set $X$ but
		with different finite morphisms $Q_1$ and $Q_2$, define a transport
		atlas. Their root walls meet transversely along
		$\Sigma=H_1\cap H_2=\{x=y=0\}$. At a point $p=(0,0,z)$, the common
		fixed tangent direction is $F_{12}=T_pH_1\cap T_pH_2
		=\operatorname{span}\{\partial_z\}$. The rank-two quotient
		$T_pX/F_{12}$ is identified with the normal $(x,y)$-plane.
		
		In the quotient basis $(\partial_x,\partial_y)$, the induced transport
		operators are
		$\overline R_1(z)=\left(\begin{smallmatrix}-1&0\\-2z&1\end{smallmatrix}\right)$
		and
		$\overline R_2(z)=\left(\begin{smallmatrix}1&-2z\\0&-1\end{smallmatrix}\right)$.
		Both are involutions. Their transverse image directions are generated
		by $\nu_1=\partial_x+z\partial_y$ and
		$\nu_2=z\partial_x+\partial_y$, respectively. Choosing the conormal
		roots $\alpha_1=dx$ and $\alpha_2=dy$ gives
		$\alpha_1(\nu_1)=\alpha_2(\nu_2)=1$. Therefore
		$A_{12}(z)=\alpha_2(\nu_1)=z$ and
		$A_{21}(z)=\alpha_1(\nu_2)=z$, so the normalized invariant is
		$\Theta_{12}(z)=z^2$.
		
		For completeness, the product on the quotient is
		$\overline R_1(z)\overline R_2(z)
		=\left(\begin{smallmatrix}-1&2z\\-2z&4z^2-1\end{smallmatrix}\right)$;
		it has determinant one and trace $4z^2-2$. For real $z$ with $|z|>1$, it has distinct real eigenvalues off the unit circle
		and therefore has infinite order. Consequently, at such a point the
		two selected germs cannot both preserve any common finite morphism: if
		they did, their product would belong to the finite local symmetry
		group of that morphism. This exhibits explicitly why the
		common-morphism finite-order conclusion cannot be applied to the pair
		$(\tau_1,\tau_2)$.
	\end{example}
	
	\begin{theorem}[Nonconstant Cartan Invariant on a Transport Atlas]
		\label{thm:example-non-constant-cartan-matrix}
		On the real intersection curve $\Sigma(\mathbb R)$ in
		\Cref{ex:transport-atlas-z2}, the mixed invariant is the nonconstant
		continuous function $\Theta_{12}(z)=z^2$. More precisely, for
		$0<|z|<1$ the reduced product is elliptic; for $|z|=1$ it is
		parabolic; and for $|z|>1$ it is hyperbolic. At $z=0$ it equals
		$-I$, the scalar repeated-root boundary case (projectively the
		identity).
	\end{theorem}
	
	\begin{proof}
		The example gives $\Theta_{12}(z)=z^2$. For $0<|z|<1$ this lies
		strictly between zero and one, so the Rank--Two Classification Theorem
		gives elliptic dynamics. If $|z|=1$, then $\Theta_{12}=1$ and the
		quotient product has trace two and determinant one. The explicit
		product matrix in Example~\ref{ex:transport-atlas-z2} has nonzero
		off-diagonal entries at $z=\pm1$, so it is not the identity; it
		therefore has a nontrivial Jordan part and is parabolic. If $|z|>1$, then $\Theta_{12}>1$, and the quotient product
		has distinct real reciprocal eigenvalues, so it is hyperbolic. At
		$z=0$, the product is $-I$, the scalar repeated-root case at
		$\Theta_{12}=0$.
	\end{proof}
	
	\begin{remark}[Why the Example Does Not Contradict Single-Morphism Rigidity]
		The wall $H_1$ is represented by $Q_1$ and the wall $H_2$ by $Q_2$.
		The selected involutions preserve their respective morphisms, but the
		atlas does not provide one finite morphism preserved by both. For
		$|z|>1$, the quotient differential of their product has infinite
		order, confirming that the pair need not generate a finite transport
		group. The variable invariant is therefore compatible with the
		single-morphism local and global rigidity theorems.
	\end{remark}
	
	\section{Beyond Reflection Groups}
	\label{sec:Beyond-Reflection-Groups}
	
	The preceding sections identify finite Weyl groups and finite complex
	reflection groups as important sources of coincidence geometry. They
	also distinguish two settings. For a single finite morphism, selected
	biholomorphic transport branches belong to a finite local symmetry
	group, and their rank--two invariants satisfy the finite-order
	restrictions established in Sections~\ref{sec:Rank-Two-Transport-Geometry}
	and~\ref{sec:finite-monodromy-rigidity}. A transport atlas, by contrast,
	may assign different finite morphisms to different charts. The mixed
	invariant of branches from distinct charts need not satisfy the
	common-morphism finite-order conclusion.
	
	To compare such an atlas with a linear reflection model, the notion of
	equivalence must retain the relative position of all chart data. In
	particular, allowing an unrelated change of source coordinates in each
	chart would erase that relative information. We first record the
	functoriality of the branchwise Cartan invariant for a single
	coincidence correspondence. We then define an atlas isomorphism using
	one common source isomorphism for all charts, and prove that a
	nonconstant mixed Cartan invariant obstructs simultaneous
	linearization to a common-vector-space reflection model.
	
	We do not use curvature in chosen coordinates, or constancy of a
	coordinate expression for a root direction, as an obstruction under
	arbitrary variety isomorphisms. Such properties are not invariant in
	the category considered here.
	
	\subsection{Coincidence Isomorphisms}
	\label{subsec:Coincidence-Isomorphisms}
	
	\begin{definition}[Coincidence Isomorphism]
		\label{def:Coincidence-Isomorphisms}
		Let
		\[
		Q_1:X_1\to Y_1,
		\qquad
		Q_2:X_2\to Y_2
		\]
		be finite morphisms. A \emph{coincidence isomorphism} between $Q_1$
		and $Q_2$ is a pair of variety isomorphisms
		\[
		\phi:X_1\to X_2,
		\qquad
		\psi:Y_1\to Y_2,
		\]
		such that $Q_2\circ\phi=\psi\circ Q_1$. The isomorphism is understood
		to carry the selected coincidence components and local branches under
		consideration to their corresponding components and branches.
	\end{definition}
	
	\begin{proposition}[Functoriality]
		\label{prop:functoriality}
		Let $(\phi,\psi)$ be a coincidence isomorphism from $Q_1$ to $Q_2$.
		Then:
		\begin{enumerate}[label=(\roman*)]
			\item the induced map on self-fibre products gives a bijection between
			the irreducible coincidence components of $Q_1$ and those of $Q_2$;
			\item corresponding components have corresponding root walls, and
			selected local graph branches are carried to the corresponding local
			graph branches;
			\item at a regular root point, the differential of $\phi$ conjugates
			the corresponding transport operators and identifies their kernel
			hyperplanes, image lines, and conormal root lines;
			\item for any selected Cartan-regular pair of branches at a common
			point, the normalized invariant $\Theta_{ij}$ is preserved.
		\end{enumerate}
	\end{proposition}
	
	\begin{proof}
		The identity $Q_2\circ\phi=\psi\circ Q_1$ induces an isomorphism
		\[
		\Phi:X_1\times_{Y_1}X_1\longrightarrow X_2\times_{Y_2}X_2,
		\qquad
		\Phi(x_1,x_2)=(\phi(x_1),\phi(x_2)).
		\]
		An isomorphism preserves irreducible components, proving (i). It also
		preserves the diagonal, so it carries
		$D\cap\Delta_{X_1}$ to $\Phi(D)\cap\Delta_{X_2}$; this proves the
		root-wall assertion in (ii). A local graph branch of a transport germ
		$\tau$ is carried to the graph of $\phi\circ\tau\circ\phi^{-1}$.
		
		Fix a point $x$ and write $P=d\phi_x$. For corresponding transport
		operators one has $R_i'=P R_iP^{-1}$. Hence
		$P(\ker(I-R_i))=\ker(I-R_i')$ and
		$P(\operatorname{Im}(I-R_i))=\operatorname{Im}(I-R_i')$.
		The dual map carries the conormal root line to the corresponding
		conormal root line. This proves (iii).
		
		For the final assertion, let $(\alpha_i,\nu_i)$ and
		$(\alpha_j,\nu_j)$ be normalized pairs at $x$. Their transported
		normalized pairs are
		\[
		\nu_i'=P\nu_i,\quad \nu_j'=P\nu_j,
		\qquad
		\alpha_i'=\alpha_i\circ P^{-1},\quad
		\alpha_j'=\alpha_j\circ P^{-1}.
		\]
		They remain normalized, and
		$\alpha_j'(\nu_i')=\alpha_j(\nu_i)$ and
		$\alpha_i'(\nu_j')=\alpha_i(\nu_j)$. Thus both Cartan coefficients,
		and in particular their product $\Theta_{ij}$, are preserved. If
		different normalized generators are chosen, the Transformation Law
		\Cref{prop:Transformation-Law} gives the same product. This proves
		(iv).
	\end{proof}
	
	\subsection{Single-Morphism Linearization}
	\label{subsec:Linearizable-Coincidence-Geometry}
	
	\begin{definition}[Linearizable Coincidence Geometry]
		\label{def:Linearizable-Coincidence-Geometry}
		A coincidence geometry of a single finite morphism is called
		\emph{linearizable} if it is coincidence-isomorphic to the quotient
		morphism $\pi:V\to V/G$, where $G\subseteq\operatorname{GL}(V)$ is a
		finite complex reflection group.
	\end{definition}
	
	\begin{corollary}[Branchwise Intrinsic Cartan Invariant]
		\label{cor:intrinsic-cartan}
		For a specified pair of regular transport branches at a specified
		Cartan-regular intersection point, the normalized scalar
		$\Theta_{ij}=A_{ij}A_{ji}$ is invariant under coincidence isomorphisms
		and under changes of normalized root generators. Under the hypotheses
		of the Finite-Order Transport Trace Formula, it is also given by
		\[
		\Theta_{ij}(x)=\frac{B_{ij}(x)}{(1-\zeta_i)(1-\zeta_j)}.
		\]
		The branch and point data are part of the invariant: without specifying
		the selected local branches, this statement does not assign a unique
		scalar to an unlabelled pair of global components.
	\end{corollary}
	
	\begin{proof}
		Invariance under coincidence isomorphisms is part~(iv) of Proposition~\ref{prop:functoriality}, and invariance under changes of normalized generators is the Transformation Law in Proposition~\ref{prop:Transformation-Law}. The displayed trace expression is
		Corollary~\ref{cor:Trace-Theoretic-Expression}.
	\end{proof}
	
	\begin{remark}[Curvature and Coordinate Constancy Are Not Intrinsic]
		A nonlinear variety isomorphism can turn a hyperplane into a curved
		hypersurface in the chosen affine coordinates and can turn a constant
		transverse line into a varying line field. Thus neither apparent
		curvature of a wall nor constancy of a coordinate representative of a
		root direction is, by itself, an obstruction to coincidence
		linearization.
		
		For an explicit example, let $\pi:\mathbb A^2\to\mathbb A^2$ be the
		order-two reflection quotient $\pi(u,v)=(u,v^2)$, with reflection
		wall $H=\{v=0\}$ and transverse direction $\mathbb C\partial_v$.
		The polynomial map
		\[
		\phi(u,v)=\bigl(u+(v+u^2)^2,\ v+u^2\bigr)
		\]
		is an automorphism of $\mathbb A^2$. The coincidence-isomorphic map
		$Q'=\pi\circ\phi^{-1}$ has root wall $\phi(H)$, parametrized by
		$(u+u^4,u^2)$, which is curved in the displayed affine coordinates.
		Moreover,
		$d\phi_{(u,0)}(\partial_v)=(2u^2,1)$, so the transported transverse
		line varies with $u$. The underlying finite morphisms are nevertheless
		coincidence-isomorphic. What is preserved is the line subbundle as
		transported by the differential, not its constancy in a chosen affine
		trivialization.
	\end{remark}
	
	\subsection{Atlas Isomorphisms and Simultaneous Linearization}
	\label{subsec:Atlas-Linearization}
	
	We now use the weak transport-atlas structure of
	Section~\ref{sec:transport-atlases-and-Local-Finite-Covers}. Recall
	that its chart morphisms need not be restrictions of one common finite
	morphism, and no target-side overlap or descent data are assumed.
	Nevertheless, all charts have the same source variety. An equivalence
	that is to preserve mixed invariants between different charts must
	therefore use one and the same source isomorphism for every chart.
	
	\begin{definition}[Atlas Coincidence Isomorphism]
		\label{def:atlas-coincidence-isomorphism}
		Let
		\[
		\mathfrak A=\bigl(X,\{(U_\alpha,Q_\alpha:U_\alpha\to Y_\alpha)\},
		\{D_{\alpha,i}\}\bigr),
		\qquad
		\mathfrak A'=\bigl(X',\{(U'_{\alpha'},Q'_{\alpha'}:U'_{\alpha'}\to
		Y'_{\alpha'})\},\{D'_{\alpha',i'}\}\bigr)
		\]
		be weak transport atlases with their selected chart-labelled
		coincidence components. An \emph{atlas coincidence isomorphism} consists
		of a variety isomorphism $\phi:X\to X'$, a bijection $\sigma$ of the
		chart index sets, and, for each $\alpha$, a variety isomorphism
		$\psi_\alpha:Y_\alpha\to Y'_{\sigma(\alpha)}$ such that
		$\phi(U_\alpha)=U'_{\sigma(\alpha)}$ and
		\[
		Q'_{\sigma(\alpha)}\circ\phi|_{U_\alpha}
		=\psi_\alpha\circ Q_\alpha.
		\]
		The selected component labels are required to correspond under the
		induced isomorphisms of the chartwise self-fibre products. In
		particular, the source isomorphism $\phi$ is common to all charts;
		independent source isomorphisms in different charts do not define an
		atlas coincidence isomorphism.
	\end{definition}
	
	\begin{proposition}[Functoriality of Mixed Cartan Invariants]
		\label{prop:atlas-theta-functoriality}
		An atlas coincidence isomorphism preserves every mixed Cartan invariant
		formed from two selected branches at a common point, including when
		the branches belong to different charts. More precisely, if the
		selected branches $i$ and $j$ meet at $x$, then
		\[
		\Theta'_{\sigma(\alpha),i';\,\sigma(\beta),j'}(\phi(x))
		=\Theta_{\alpha,i;\,\beta,j}(x).
		\]
	\end{proposition}
	
	\begin{proof}
		The common source isomorphism $\phi$ induces the same linear isomorphism
		$P=d\phi_x$ on the tangent space for both chart-labelled branches.
		Each selected transport operator is conjugated by $P$, while its
		normalized transverse vector is sent by $P$ and its normalized
		conormal covector by the dual map. Thus both cross-pairings are
		preserved, exactly as in the proof of part~(iv) of Proposition~\ref{prop:functoriality}. Their product is therefore preserved.
		The fact that the branches belong to different charts does not affect
		the argument: they are compared in the same tangent-cotangent fibre
		because the source atlas has one common variety $X$.
	\end{proof}
	
	\begin{definition}[Common-Vector-Space Linear Reflection Atlas]
		\label{def:linear-reflection-atlas}
		A weak transport atlas is a \emph{common-vector-space linear reflection
			atlas} if its source is a finite-dimensional complex vector space $V$
		and, for every chart $\alpha$, its chart domain is an open subset
		$V_\alpha\subseteq V$ invariant under a finite complex reflection
		group $G_\alpha\subseteq\operatorname{GL}(V)$, with chart morphism
		the restriction of the quotient map
		$\pi_\alpha|_{V_\alpha}:V_\alpha\to V_\alpha/G_\alpha$.
		The chart domains form the given open cover of $V$. Every selected
		rank-one chart branch is required to be induced by a pseudoreflection
		in its chart group. An atlas is \emph{simultaneously linearizable}
		if it is atlas-coincidence-isomorphic to such a local quotient atlas.
	\end{definition}
	
	\begin{theorem}[Variable Cartan Invariant Obstructs Simultaneous Linearization]
		\label{thm:linearization-obstruction}
		Let $\mathfrak A$ be a weak transport atlas on a smooth complex
		variety. Suppose two selected chart-labelled rank-one branches, say
		$(\alpha,i)$ and $(\beta,j)$,
		have walls meeting along a connected positive-dimensional smooth
		complex-analytic locus $S$. Assume that the branches vary
		holomorphically along $S$, that both Cartan splittings hold at every
		point of $S$, and that the mixed invariant
		$\Theta_{\alpha,i;\beta,j}:S\to\mathbb C$ is nonconstant. Then
		$\mathfrak A$ is not simultaneously linearizable to a
		common-vector-space linear reflection atlas.
	\end{theorem}
	
	\begin{proof}
		Suppose, to the contrary, that an atlas coincidence isomorphism sends
		$\mathfrak A$ to a common-vector-space linear reflection atlas on a
		vector space $V$. By Proposition~\ref{prop:atlas-theta-functoriality},
		the mixed invariant is unchanged under this isomorphism.
		
		In the local linear atlas, the two selected chart branches are
		induced by fixed pseudoreflections $s_i\in G_{\sigma(\alpha)}$
		and $s_j\in G_{\sigma(\beta)}$. These are fixed linear
		transformations of $V$, independent of the point of $V$. Their image
		lines and conormal lines are therefore constant subspaces of $V$ and
		$V^*$. The normalized cross-pairings, and hence the mixed invariant,
		are constant on the intersection of their fixed hyperplanes with the
		common chart domain
		$V_{\sigma(\alpha)}\cap V_{\sigma(\beta)}$. It follows
		that the original invariant must be constant on $S$, contradicting
		the hypothesis.
	\end{proof}
	
	\begin{corollary}[The Variable Atlas Is Not Simultaneously Linearizable]
		\label{cor:variable-atlas-not-linearizable}
		The transport atlas of Example~\ref{ex:transport-atlas-z2} is not
		simultaneously linearizable to a common-vector-space linear reflection
		atlas.
	\end{corollary}
	
	\begin{proof}
		The complex intersection locus in that example is
		$\Sigma=\{x=y=0\}\cong\mathbb A^1_{\mathbb C}$, and its mixed
		invariant is $\Theta_{12}(z)=z^2$. This is a nonconstant holomorphic
		function on the connected locus $\Sigma$. The Cartan splittings hold
		because both selected branches are involutions. Theorem~\ref{thm:linearization-obstruction} therefore applies.
	\end{proof}
	
	\begin{remark}[Individual and Simultaneous Linearization]
		The conclusion concerns simultaneous linearization, not the
		linearizability of either chart by itself. In Example
		\ref{ex:transport-atlas-z2}, the first chart map factors as
		$Q_1=\pi_1\circ\phi_1$, where
		$\phi_1(x,y,z)=(x,y-zx,z)$ and
		$\pi_1(u,v,w)=(u^2,v,w)$. The second factors as
		$Q_2=\pi_2\circ\phi_2$, where
		$\phi_2(x,y,z)=(x-zy,y,z)$ and
		$\pi_2(u,v,w)=(u,v^2,w)$. Thus each chart is individually
		coincidence-isomorphic to a degree-two reflection quotient. The
		nonconstant invariant $z^2$ proves that no single source isomorphism
		can linearize both charts simultaneously to quotient maps on one
		common vector space.
		
		This is precisely the relative geometric information retained by the
		atlas coincidence isomorphism. If two branches in an atlas do happen
		to preserve one common finite morphism, the local finite-symmetry
		results of the earlier sections continue to apply to that pair.
	\end{remark}
	
	Two supplementary directions arising from this framework are recorded in
	the appendices. Appendix~\ref{sec:Geometric-Rigidity-and-Degenerations}
	collects a projective-rigidity statement under the stated global
	regularity assumptions, together with results on tangencies and
	degenerations of Cartan-regular transport walls. Appendix~\ref{sec:Connection-Induced-Cartan-Extensions}
	gives a connection-induced two-point extension after a flat holomorphic
	connection and wall-point maps have been chosen. Further development of these directions,
	including global compatibility, dependence on auxiliary choices, and
	higher-rank extensions, is left for future work.

	\section{Conclusion}
\label{sec:Conclusion}

In this work, we develop a local coincidence-root framework for finite morphisms of smooth varieties. A selected regular branch of a non-diagonal component of $\mathcal C_Q=X\times_YX$ determines a local transport symmetry and, under the stated smooth-wall and rank-one hypotheses, associated root and conormal data. The construction is local and branch-indexed. Its symmetry operators arise directly from the geometry of the coincidence correspondence, with no pre-existing linear group action assumed.

The rank-two reduction is the main structural result. It gives explicit trace and determinant formulas, with the characteristic polynomial of the reduced operator determined by $(\zeta_i,\zeta_j,\Theta_{ij})$. In the real involutive case, the resulting spectrum is elliptic or hyperbolic, or lies on a repeated-root boundary. On that boundary, the scalar and non-scalar Jordan cases must be distinguished; in the two-dimensional projective setting, only the non-scalar case is parabolic. Finite order excludes that Jordan case. The normalized invariant also provides the edge datum for the associated transport Cartan and Dynkin structures.

For a single finite morphism, every holomorphic self-germ fixing a point and preserving the morphism is automatically finite order; the resulting finite local symmetry group is simultaneously analytically linearizable at that point in every dimension. In particular, averaging a Hermitian form gives $0\leq\Theta_{ij}<1$ for every transverse rank-one pair in this common group, and distinct intrinsic root-wall germs cannot have genuine tangency. These are local analytic statements, not global algebraic linearization results. The separate finite-monodromy rigidity theorem restricts global intrinsic transport invariants. Weak transport atlases allow mixed invariants to vary when meeting walls are represented by different finite morphisms in different charts, but these atlases are deliberately unglued and carry no overlap or descent data. The obstruction in Section~\ref{sec:Beyond-Reflection-Groups} concerns simultaneous linearization to the specific common-vector-space reflection model defined there. The connection-induced construction in Appendix~\ref{sec:Connection-Induced-Cartan-Extensions} is a separate, choice-dependent extension, not an intrinsic invariant of the underlying finite morphism.

Finite Weyl-group and complex-reflection-group quotients provide the classical cases: their coincidence correspondences recover the corresponding reflection geometry and acting groups. Beyond these quotient settings, the results give local transport data under explicit hypotheses, not a general global reflection group, Weyl groupoid, or classical root system. The framework is therefore best viewed as a correspondence-based local construction alongside existing theories of reflection groups, Weyl groupoids, hyperplane arrangements, lattice root systems, and monodromy of finite covers.

The appendices record two supplementary directions: projective-rigidity and tangency consequences of the normalized invariant (Appendix~\ref{sec:Geometric-Rigidity-and-Degenerations}), and connection-induced two-point Cartan fields (Appendix~\ref{sec:Connection-Induced-Cartan-Extensions}). These results are not intended as a comprehensive theory of either extension. Further work is needed on higher-rank transport geometry, global groupoid and descent structures, singular coincidence components, and moduli and deformation questions.

\section*{Declaration on the Use of Generative AI}
Generative AI tools were used to assist with language editing and improving the clarity and presentation of the manuscript, as well as with the development of Lean code supporting the formalization of selected mathematical results. The author carefully reviewed all AI-assisted material and takes full responsibility for the accuracy and integrity of the manuscript and its contents.

\appendix

	\section{Geometric Rigidity and Degenerations}
	\label{sec:Geometric-Rigidity-and-Degenerations}
	
	The normalized transport invariant $\Theta_{ij}=A_{ij}A_{ji}$
	contains information about the relative position of two selected
	transport branches. This appendix records two consequences of that
	invariant. First, projectivity forces constancy when the invariant
	extends as a global regular function on a connected projective locus.
	Second, tangency of two Cartan-regular transport walls forces the
	associated Cartan determinant to vanish. The latter statement is a
	Cartan-boundary result; it does not, without an additional Jordan
	condition, assert parabolicity in the rank-two spectral classification.
	
	\subsection{Projective Rigidity of Transport Invariants}
	\label{subsec:Projective-Rigidity-of-Transport-Invariants}
	
	The variable atlas in Section~\ref{sec:transport-atlases-and-Local-Finite-Covers}
	produces a nonconstant invariant on an affine intersection locus.
	Projectivity imposes a restriction on such variation, but the precise
	hypothesis is important: the scalar must extend to a global regular
	function on the entire connected projective locus. Pointwise-defined
	invariants, rational invariants with poles, and functions defined only
	on a proper open subset are not covered by the conclusion below.
	
	\begin{theorem}[Projective Rigidity]
		\label{thm:projective-rigidity}
		Let $k$ be an algebraically closed field and let $S$ be a connected,
		reduced, projective $k$-variety contained in a common smooth source
		variety. Suppose two selected transport-branch families are
		Cartan-regular along $S$. Assume that their conormal lines and
		transverse image lines form algebraic line subbundles
		$\mathcal L_i,\mathcal L_j\subseteq T_X^*|_S$ and
		$\mathcal C_i,\mathcal C_j\subseteq T_X|_S$, and that the pairings
		$\mathcal L_i\otimes\mathcal C_i\to\mathcal O_S$ and
		$\mathcal L_j\otimes\mathcal C_j\to\mathcal O_S$ are nowhere
		vanishing. Then the local normalized Cartan data define a global
		regular function $\Theta_{ij}\in H^0(S,\mathcal O_S)$, and this
		function is constant on $S$.
	\end{theorem}
	
	\begin{proof}
		Choose a Zariski open cover of $S$ on which each of the four line
		bundles has a nowhere-vanishing local generator. On such an open set,
		choose local generators $\alpha_i,\alpha_j$ of the conormal lines
		and local generators $e_i,e_j$ of the transverse image lines. The
		nowhere-vanishing self-pairings allow us to define the normalized
		vectors $\nu_i=e_i/\alpha_i(e_i)$ and
		$\nu_j=e_j/\alpha_j(e_j)$. The coefficients
		$A_{ij}=\alpha_j(\nu_i)$ and $A_{ji}=\alpha_i(\nu_j)$ are regular
		on this open set, so their product is regular there.
		
		On overlaps, changing a conormal generator by a nowhere-vanishing
		regular factor changes the corresponding normalized vector by the
		inverse factor. By the Transformation Law
		\Cref{prop:Transformation-Law}, the product $A_{ij}A_{ji}$ is
		unchanged. Hence these local regular functions agree on overlaps and
		glue to a global regular function $\Theta_{ij}$ on $S$.
		
		A connected, reduced, projective variety over an algebraically closed
		field has only constant global regular functions. Thus
		$\Theta_{ij}$ is constant on $S$.
	\end{proof}
	
	\begin{corollary}[Constancy on Projective Cartan-Regular Loci]
		\label{cor:projective-transport-loci}
		Under the hypotheses of \Cref{thm:projective-rigidity}, no connected
		projective Cartan-regular locus can support a nonconstant regular
		mixed transport invariant.
	\end{corollary}
	
	\begin{proof}
		This is the conclusion of \Cref{thm:projective-rigidity}.
	\end{proof}
	
	\begin{remark}[Scope of Projective Rigidity]
		The conclusion is about global regular functions, not arbitrary
		rational or meromorphic functions. A nonconstant rational function on
		a projective variety may have poles, and a regular function on a
		proper open subset need not extend to the projective completion. For
		example, the variable invariant $\Theta_{12}(z)=z^2$ on the affine
		locus $\Sigma\cong\mathbb A^1$ in Example~\ref{ex:transport-atlas-z2}
		is regular on $\mathbb A^1$ but has a pole at the point at infinity
		when viewed as a rational function on $\mathbb P^1$. Thus the example
		is fully compatible with projective rigidity: it does not define a
		global regular function on that projective completion.
		
		Accordingly, projective rigidity does not imply that every transport
		atlas on a projective ambient variety has constant invariants. It
		implies constancy only on a connected projective locus on which the
		selected branch data are algebraic and the mixed invariant extends
		regularly. Variation may occur on affine open subsets of a projective
		variety, or for rational invariants with poles. In the complex analytic
		setting, the analogous statement holds whenever the invariant is a
		global holomorphic function on a connected reduced compact complex space (in particular, on a connected compact complex manifold).
	\end{remark}
	
	\subsection{Tangency and Source-Wall Degenerations}
	\label{subsec:Transversality-Breakdown}
	
	For transverse walls, the common tangent space has codimension two,
	and the rank-two reduction of Section~\ref{subsec:Rank-Two-Reduction-Theory}
	applies. At a tangency point, the two tangent hyperplanes coincide;
	there is then no two-dimensional quotient by their common tangent
	space. Nevertheless, when the Cartan splittings hold, the mixed
	invariant remains defined and tangency forces the precise value
	$\Theta_{ij}=1$. This is the singular boundary of the Cartan pairing,
	not automatically a parabolic case of the transverse rank-two
	classification.
	
	\begin{definition}[Cartan-Regular Tangency Point]
		\label{def:Cartan-Regular-Tangency-Point}
		Let $H_i,H_j$ be smooth transport walls through a point $x$ for
		selected local branches. We call $x$ a \emph{Cartan-regular tangency
			point} if it is a regular transport root point for both branches,
		$T_xH_i=T_xH_j$, and both splitting conditions
		$T_xX=\ker(I-R_i(x))\oplus\operatorname{Im}(I-R_i(x))$ and
		$T_xX=\ker(I-R_j(x))\oplus\operatorname{Im}(I-R_j(x))$ hold.
	\end{definition}
	
	\begin{theorem}[Tangency Forces the Singular Cartan Boundary]
		\label{thm:tangency-cartan-boundary}
		Let $H_i,H_j$ be smooth transport walls through $x$, with selected
		local branches. Suppose $x$ is a Cartan-regular tangency point in the
		sense of Definition~\ref{def:Cartan-Regular-Tangency-Point}. Then
		$\Theta_{ij}(x)=1$. Consequently, for the rank-two Cartan matrix
		$\mathcal C_{ij}=\begin{pmatrix}2&2A_{ij}\\2A_{ji}&2\end{pmatrix}$,
		we have $\det(\mathcal C_{ij}(x))=0$.
	\end{theorem}
	
	\begin{proof}
		Since the walls are tangent at $x$, their tangent hyperplanes agree:
		$T_xH_i=T_xH_j$. Taking annihilators gives
		$N_x^*H_i=N_x^*H_j$. Choose nonzero conormal generators so that
		$\alpha_i=\lambda\alpha_j$ for some $\lambda\in k^\times$.
		The Cartan splittings ensure that the normalized vectors
		$\nu_i,\nu_j$ are defined and satisfy
		$\alpha_i(\nu_i)=\alpha_j(\nu_j)=1$. Hence
		$A_{ij}=\alpha_j(\nu_i)=\lambda^{-1}$ and
		$A_{ji}=\alpha_i(\nu_j)=\lambda$. Therefore
		$\Theta_{ij}=A_{ij}A_{ji}=1$.
		
		By the plus-sign convention fixed in
		Definition~\ref{def:Rank--Two-Transport-Cartan-Matrix}, the
		off-diagonal entries are $+2A_{ij}$ and $+2A_{ji}$; this convention
		does not assert that their numerical values are positive. Its determinant is
		$4(1-A_{ij}A_{ji})=4(1-\Theta_{ij})$, which vanishes at $x$.
	\end{proof}
	
	\begin{corollary}[Cartan-Regular Tangency Locus]
		\label{cor:Cartan-Regular-Tangency-Locus}
		Let $\Sigma_{ij}^{\mathrm{tan,CR}}$ be the locus of Cartan-regular
		tangency points for the selected pair of branches. Then
		$\Sigma_{ij}^{\mathrm{tan,CR}}\subseteq
		\{x:\Theta_{ij}(x)=1\}$.
	\end{corollary}
	
	\begin{proof}
		Apply \Cref{thm:tangency-cartan-boundary} at each point of
		$\Sigma_{ij}^{\mathrm{tan,CR}}$.
	\end{proof}
	
	\begin{proposition}[Unipotent Product for Tangent Involutions]
		\label{prop:tangent-involutions-unipotent}
		In the setting of \Cref{thm:tangency-cartan-boundary}, assume in
		addition that both selected transport operators are involutions. Choose
		a common conormal generator $\alpha$ and the corresponding normalized
		vectors $\nu_i,\nu_j$, so that
		$\alpha(\nu_i)=\alpha(\nu_j)=1$. Put
		$\Pi_i=\nu_i\otimes\alpha$ and
		$\Pi_j=\nu_j\otimes\alpha$. Then
		$R_iR_j=I+2(\Pi_i-\Pi_j)$ and
		$(\Pi_i-\Pi_j)^2=0$. In particular, $R_iR_j$ is unipotent, and it is
		nonidentity precisely when $R_i\ne R_j$.
	\end{proposition}
	
	\begin{proof}
		For an involution with rank-one image of $I-R_k$, the normalized
		rank-one projector satisfies $R_k=I-2\Pi_k$. Since
		$\alpha(\nu_i)=\alpha(\nu_j)=1$, we have
		$\Pi_i\Pi_j=\Pi_i$ and $\Pi_j\Pi_i=\Pi_j$. Therefore
		$R_iR_j=(I-2\Pi_i)(I-2\Pi_j)
		=I+2(\Pi_i-\Pi_j)$. Moreover,
		$(\Pi_i-\Pi_j)^2=\Pi_i-\Pi_i-\Pi_j+\Pi_j=0$.
		Thus $(R_iR_j-I)^2=0$. The product is the identity exactly when
		$\Pi_i=\Pi_j$, equivalently when $R_i=R_j$.
	\end{proof}

	\begin{example}[A Tangency Realized by Two Finite Quotient Maps]
		\label{ex:finite-morphism-tangency}
		This example realizes the tangency phenomenon using two genuine
		finite morphisms on the same source. On $X=\mathbb A^2_{\mathbb C}$,
		let
		\[
		Q_1(x,y)=(x,y^2),\qquad
		\tau_1(x,y)=(x,-y).
		\]
		The map $Q_1$ is finite of degree two, and $\tau_1$ is its
		nontrivial deck involution with fixed wall $H_1=\{y=0\}$.
		Fix $a\in\mathbb C^\times$ and consider the polynomial
		automorphism
		\[
		\phi(x,y)=(u,v)=
		(x+ay,\ y+(x+ay)^2),
		\qquad
		\phi^{-1}(u,v)=
		(u-a(v-u^2),\ v-u^2).
		\]
		Define $Q_2(x,y)=(u,v^2)$, where $(u,v)=\phi(x,y)$.
		Since $\phi$ is an automorphism and $(u,v)\mapsto(u,v^2)$
		is finite, $Q_2$ is finite of degree two. Its nontrivial deck
		involution is
		$\tau_2=\phi^{-1}\circ(u,v\mapsto u,-v)\circ\phi$,
		and its fixed wall is
		$H_2=\{y+(x+ay)^2=0\}$. The two smooth walls have the same
		tangent line $y=0$ at the origin.
		Their transport differentials at the origin are
		\[
		R_1=\begin{pmatrix}1&0\\0&-1\end{pmatrix},\qquad
		R_2=\begin{pmatrix}1&2a\\0&-1\end{pmatrix}.
		\]
		Using the common conormal $\alpha=dy$, take
		$\nu_1=(0,1)$ and $\nu_2=(-a,1)$, which span
		$\operatorname{Im}(I-R_1)$ and $\operatorname{Im}(I-R_2)$,
		respectively. Both satisfy $\alpha(\nu_k)=1$, so
		$A_{12}=A_{21}=1$ and
		$\Theta_{12}=1$. Direct multiplication gives
		\[
		R_1R_2=
		\begin{pmatrix}1&2a\\0&1\end{pmatrix}=I+N,
		\qquad N\ne0,\quad N^2=0.
		\]
		Thus the product on the full tangent space $T_0X$ is a nontrivial
		unipotent of infinite order. Since the walls are tangent, this example
		does not carry the canonical two-dimensional quotient from the
		transverse rank--two reduction theorem. The two involutions preserve
		different finite morphisms, so the example belongs to the weak-atlas
		setting and does not contradict the finite-order result for a common
		finite morphism.
	\end{example}
	
	\begin{remark}[Why Tangency Is Not Automatically Rank-Two Parabolicity]
		At a tangency point, $T_xH_i=T_xH_j$, so their common tangent space
		has dimension $\dim X-1$, not $\dim X-2$. The canonical two-dimensional
		quotient in the Rank-Two Reduction Theorem is therefore unavailable.
		The value $\Theta_{ij}=1$ is the repeated-root boundary value in the
		transverse involutive rank-two classification, but the tangency
		statement itself is not an application of that classification. When
		both operators are involutions, Proposition~\ref{prop:tangent-involutions-unipotent}
		gives the more precise conclusion: their product is unipotent, and it
		is nontrivial exactly when the two differentials are distinct. If the
		differentials agree, their product is the identity. The term
		\emph{parabolic} should be used only after the relevant projective
		Jordan action has been specified and shown to be nontrivial.
	\end{remark}
	
	\begin{remark}[Tangency as a Degeneracy Locus]
	\label{rem:tangency-degeneracy-locus}
	For a transverse Cartan-regular pair, substituting
	$\Theta_{ij}=1$ into the general characteristic polynomial gives
	\[
		 p_{ij}(t)=(t-1)(t-\zeta_i\zeta_j).
	\]
	Thus $\Theta_{ij}=1$ alone does not imply unipotence. If
	$\zeta_i\zeta_j\ne1$, the two eigenvalues are distinct. If
	$\zeta_i\zeta_j=1$, the repeated eigenvalue is $1$; transversality
	excludes the identity quotient, so the quotient product is necessarily
	a nontrivial unipotent. In particular, this conclusion holds in the
	involutive case $\zeta_i=\zeta_j=-1$ treated in the transverse
	rank--two classification.

	For two branches preserving one common finite morphism, however,
	Proposition~\ref{prop:strict-theta-common-morphism} gives the
	stronger restriction $0\leq\Theta_{ij}<1$ for every transverse
	rank-one pair, without any involutivity assumption. Hence
	$\Theta_{ij}=1$ is excluded altogether for transverse
	common-morphism pairs, including the non-involutive case in which
	the displayed polynomial would otherwise have distinct roots.
	Accordingly, the value $\Theta_{ij}=1$ for a transverse pair
	belongs to the mixed-chart setting, where the two branches need not
	preserve one common finite morphism.

	The tangency theorem separately shows that every Cartan-regular
	tangency point has $\Theta_{ij}=1$. At such a point, the walls share a common tangent hyperplane, so the transverse two-dimensional quotient is absent. For
	two distinct intrinsic root-wall germs preserving one common finite
	morphism, Corollary~\ref{cor:no-genuine-tangency-common-morphism}
	shows that genuine tangency cannot occur: equal tangent hyperplanes
	force the wall germs to coincide. Thus the tangency phenomenon is
	realized by Example~\ref{ex:finite-morphism-tangency}, which uses
	two different finite morphisms. When both operators are
	involutions, Proposition~\ref{prop:tangent-involutions-unipotent}
	shows that the full tangent-space product is the identity or a
	nontrivial unipotent, according as the two differentials agree or
	differ. Without involutivity, no such unipotence conclusion follows
	from $\Theta_{ij}=1$ alone. The converse is not asserted: the equality
	$\Theta_{ij}=1$ does not by itself imply that the source walls are
	tangent.
\end{remark}

	\begin{example}[The Value $\Theta_{12}=1$ Need Not Be Unipotent]
		\label{ex:theta-one-nonunipotent}
		This example shows why the involutivity hypothesis in the preceding
		unipotence statement is essential. On $X=\mathbb A^2_{\mathbb C}$,
		consider the finite degree-four morphisms
		\[
		Q_1(x,y)=(x^4,y-x),\qquad
		Q_2(x,y)=(x-y,y^4),
		\]
		with deck transformations
		\[
		\tau_1(x,y)=(ix,y+(i-1)x),\qquad
		\tau_2(x,y)=(x+(i-1)y,iy).
		\]
		Each map is finite as a polynomial automorphism followed by a
		fourth-power map, and direct substitution verifies
		$Q_k\circ\tau_k=Q_k$. Both transformations have order four.
		Their fixed walls are $H_1=\{x=0\}$ and $H_2=\{y=0\}$,
		which meet transversely at the origin. Their differentials are
		\[
		R_1=\begin{pmatrix}i&0\\i-1&1\end{pmatrix},
		\qquad
		R_2=\begin{pmatrix}1&i-1\\0&i\end{pmatrix}.
		\]
		With $\alpha_1=dx$, $\alpha_2=dy$ and
		$\nu_1=\nu_2=(1,1)$, both self-normalizations and both cross
		pairings equal one, so $\Theta_{12}=1$. Nevertheless,
		$R_1R_2$ has characteristic polynomial $t^2-1$ and satisfies
		$(R_1R_2)^2=I$. Its eigenvalues are $1$ and $-1$, so it is
		semisimple and not unipotent. Nevertheless, a different word in
		the same two generators is unipotent. Indeed,
		\[
		R_2^{-1}
		=\begin{pmatrix}1&-1-i\\0&-i\end{pmatrix},
		\qquad
		R_1R_2^{-1}
		=\begin{pmatrix}i&1-i\\i-1&2-i\end{pmatrix}
		=I+N,
		\]
		where
		\[
		N=(i-1)\begin{pmatrix}1&-1\\1&-1\end{pmatrix}
		\ne0,\qquad N^2=0.
		\]
		Thus $(R_1R_2^{-1})^m=I+mN\ne I$ for every positive
		integer $m$, and the group $\langle R_1,R_2\rangle$ is infinite,
		even though both generators and their product $R_1R_2$ have finite
		order. In particular, these two germs cannot preserve one common
		finite morphism: otherwise their differentials would belong to the
		finite local symmetry group of that morphism. The two involutivity
		assumptions therefore cannot be omitted from the $+1$ unipotence
		conclusion.
		\end{example}
	
	\begin{remark}[Branch-Locus Geometry]
		Tangency of two smooth source walls means that their union fails to
		have normal-crossing intersection at that point. This is a statement
		about the source wall arrangement. It does not, by itself, prove that
		a branch locus in a target variety is singular. Such a conclusion
		requires additional information about a common finite morphism and
		about the images and local scheme structure of the ramification
		loci. For branches belonging to different charts of a weak transport
		atlas, there may not even be a single target branch locus to compare.
		We therefore treat tangency as a source-side degeneracy unless those
		additional target-side hypotheses are supplied.
	\end{remark}
	
	\begin{remark}[Comparison with Reflection Arrangements]
		Distinct reflecting hyperplanes in a linear reflection arrangement
		are distinct linear hyperplanes, so they meet transversely wherever
		they intersect. They cannot exhibit the tangency considered here.
		The Cartan-regular tangency phenomenon is therefore absent for pairs
		of distinct linear reflecting hyperplanes. This comparison concerns
		pairs of distinct walls; a repeated label for the same hyperplane is
		not a pair of distinct reflecting hyperplanes.
	\end{remark}
	
	\section{Connection-Induced Cartan Extensions}
	\label{sec:Connection-Induced-Cartan-Extensions}
	
	The intrinsic transport Cartan coefficients developed in the main text
	are defined at a point where the selected transport walls meet. They
	compare root covectors and transverse vectors in one tangent--cotangent
	fibre. This pointwise construction requires no connection. In this
	appendix we describe a different, supplementary construction: a
	connection transports root data from possibly different points of the
	walls to a common ambient point, where they can be compared.
	
	The resulting coefficients depend on the chosen connection and, when
	one passes from two-point data to a field on an open set, on the chosen
	wall-point maps. They are therefore not intrinsic invariants of the
	finite morphism. The distinction is important: the construction below
	produces variable two-point invariants even for one finite morphism,
	but it does not contradict the finite-order restrictions on the
	intrinsic invariant formed from two branches at one common point.
	
	\subsection{Transport of Root Data}
	
	Let $X$ be a smooth complex variety, and let $U\subseteq X^{\mathrm{an}}$
	be a connected, simply connected analytic open set. Suppose that the
	holomorphic tangent bundle $T_X^{\mathrm{an}}|_U$ is equipped with a
	flat holomorphic linear connection $\nabla$. The connection induces a
	dual connection on $T_X^{*\,\mathrm{an}}|_U$. Write
	$P^\nabla_{q\to x}:T_qX\to T_xX$ for parallel transport of tangent
	vectors. The corresponding transport of covectors is
	$P^{\nabla,*}_{q\to x}:=(P^\nabla_{q\to x}{}^{-1})^*:
	T_q^*X\to T_x^*X$. With this convention, parallel transport preserves
	the natural pairing:
	$(P^{\nabla,*}_{q\to x}\alpha)(P^\nabla_{q\to x}v)=\alpha(v)$.
	Flatness and simple connectedness make these transports independent of
	the chosen path in $U$. On a non-simply-connected domain, flatness
	gives homotopy invariance but may still allow nontrivial holonomy;
	comparison then depends on the homotopy class of the path unless the
	holonomy is trivial (or a path class is specified).
	
	For each selected transport wall $H_i$ and each regular wall point
	$q_i\in H_i\cap U$, choose normalized root data
	$(\alpha_i(q_i),\nu_i(q_i))$ satisfying
	$\alpha_i(q_i)(\nu_i(q_i))=1$.
	
	\begin{definition}[Transported Root Data]
		\label{def:transported-root-data}
		The transported covector and vector based at $q_i$ are
		\[
		\widetilde\alpha_i(x;q_i)
		=P^{\nabla,*}_{q_i\to x}\alpha_i(q_i),\qquad
		\widetilde\nu_i(x;q_i)
		=P^\nabla_{q_i\to x}\nu_i(q_i),
		\]
		for $x\in U$. In particular,
		$\widetilde\alpha_i(x;q_i)(\widetilde\nu_i(x;q_i))=1$.
	\end{definition}
	
	Thus root data based at different wall points can be compared after
	being transported to the same tangent--cotangent fibre over $x$.
	
	\subsection{Two-Point Cartan Coefficients}
	
	Let $H_i,H_j$ be selected transport walls, and choose regular points
	$q_i\in H_i\cap U$ and $q_j\in H_j\cap U$ with normalized root data.
	
	\begin{definition}[Two-Point Cartan Coefficient]
		\label{def:two-point-cartan}
		The two-point Cartan coefficient is
		\[
		\widetilde A_{ij}(x;q_i,q_j)
		:=\widetilde\alpha_j(x;q_j)
		\bigl(\widetilde\nu_i(x;q_i)\bigr).
		\]
		For a finite collection of selected walls and base points, the matrix
		$\widetilde A(x;q_1,\ldots,q_r)
		=(\widetilde A_{ij}(x;q_i,q_j))$ is the two-point Cartan coefficient
		matrix.
	\end{definition}
	
	\begin{proposition}[Independence of the Comparison Point]
	\label{prop:two-point-comparison-point-independence}
	Let $\nabla$ be a flat holomorphic connection on the simply connected
	open set $U$ above, and fix wall points $q_i,q_j\in U$ with
	normalized root data. Then for every $x\in U$,
	\[
	\widetilde A_{ij}(x;q_i,q_j)
	=\alpha_j(q_j)\!\left(
	P^\nabla_{q_i\to q_j}\nu_i(q_i)
	\right).
	\]
	In particular, with $q_i,q_j$ fixed, the two-point Cartan coefficient
	is independent of the comparison point $x$.
\end{proposition}

\begin{proof}
	By the definition of dual parallel transport and preservation of the
	natural pairing,
	\[
	\begin{aligned}
	\widetilde A_{ij}(x;q_i,q_j)
	&=\bigl(P^{\nabla,*}_{q_j\to x}\alpha_j(q_j)\bigr)
	  \bigl(P^\nabla_{q_i\to x}\nu_i(q_i)\bigr)\\
	&=\alpha_j(q_j)\!\left(
	  P^\nabla_{x\to q_j}P^\nabla_{q_i\to x}\nu_i(q_i)
	  \right).
	\end{aligned}
	\]
	Since $\nabla$ is flat and $U$ is simply connected, parallel
	transport depends only on the endpoints. Composition along the path
	$q_i\to x\to q_j$ therefore equals $P^\nabla_{q_i\to q_j}$,
	which proves the formula.
\end{proof}

\begin{definition}[Two-Point Cartan Matrix]
		\label{def:two-point-cartan-matrix}
		The associated two-point transport Cartan matrix is defined by
		\[
		\widetilde{\mathcal C}_{ii}(x;q_i,q_i)=2,
		\qquad
		\widetilde{\mathcal C}_{ij}(x;q_i,q_j)
		=2\widetilde A_{ij}(x;q_i,q_j)\quad(i\ne j).
		\]
		Equivalently, its diagonal entries are $2$ and its off-diagonal entries
		are twice the corresponding two-point coefficients.
	\end{definition}
	
	The positive sign in this definition is the convention fixed in
	Section~\ref{sec:Rank-Two-Transport-Geometry}: the intrinsic
	rank-two matrix has off-diagonal entries $2A_{ij}$ and $2A_{ji}$.
	Using $-2\widetilde A_{ij}$ here would be inconsistent with that
	convention. The matrix is a transport Cartan matrix in the terminology
	of this paper; its entries need not satisfy the integrality or sign
	conditions of a classical generalized Cartan matrix.
	
	\subsection{Recovery of the Intrinsic Theory}
	
	The intrinsic coefficients are recovered when the two wall points and
	the evaluation point coincide.
	
	\begin{proposition}[Diagonal Compatibility]
		\label{prop:diagonal-compatibility}\label{prop:recovery}
		Let $p\in H_i\cap H_j\cap U$ be a point at which the intrinsic Cartan
		data for the selected branches are defined. Then
		$\widetilde A_{ij}(p;p,p)=A_{ij}(p)$ and
		$\widetilde A_{ji}(p;p,p)=A_{ji}(p)$. Consequently,
		$\widetilde\Theta_{ij}(p;p,p)=\Theta_{ij}(p)$.
	\end{proposition}
	
	\begin{proof}
		Parallel transport from $p$ to itself is the identity on both tangent
		and cotangent spaces. Thus
		$\widetilde\alpha_j(p;p)=\alpha_j(p)$ and
		$\widetilde\nu_i(p;p)=\nu_i(p)$, which gives
		$\widetilde A_{ij}(p;p,p)=\alpha_j(p)(\nu_i(p))=A_{ij}(p)$.
		Interchanging $i$ and $j$ gives the second equality. Multiplication
		then gives the assertion for $\Theta_{ij}$.
	\end{proof}
	
	\subsection{Fields from Wall-Point Maps}
	
	To obtain a matrix field on an open subset of $X$, one may choose the
	wall base points as functions of the ambient point. These maps are
	additional data; in general they are not sections of the inclusions
	$H_i\hookrightarrow X$.
	
	\begin{definition}[Wall-Point Maps and Their Cartan Fields]
		\label{def:section-induced-cartan}
		Let $V\subseteq U$ be open. For each selected wall $H_i$, choose a
		holomorphic map $\sigma_i:V\to H_i^{\mathrm{reg}}\cap U$, where
		$H_i^{\mathrm{reg}}$ denotes the locus of regular transport root
		points for the selected branch. Assume also that the normalized root
		data $(\alpha_i(q),\nu_i(q))$ have been chosen holomorphically along
		$\sigma_i(V)$, after restricting $V$ if necessary. Such a map assigns
		to each ambient point $x\in V$ a point $\sigma_i(x)$ on the wall.
		Define
		\[
		\widetilde A^{\,\sigma}_{ij}(x)
		:=\widetilde A_{ij}
		\bigl(x;\sigma_i(x),\sigma_j(x)\bigr),
		\qquad
		\widetilde\Theta^{\,\sigma}_{ij}(x)
		:=\widetilde A^{\,\sigma}_{ij}(x)
		\widetilde A^{\,\sigma}_{ji}(x).
		\]
		The matrix with diagonal entries $2$ and off-diagonal entries
		$2\widetilde A^{\,\sigma}_{ij}(x)$ is the wall-map-induced transport
		Cartan matrix field.
	\end{definition}
	
	\begin{remark}
		The two-point invariant is unchanged by reciprocal rescaling of the
		root data at their base points. Indeed, if
		$\nu_i(q_i)\mapsto\lambda_i\nu_i(q_i)$ and
		$\alpha_i(q_i)\mapsto\lambda_i^{-1}\alpha_i(q_i)$, then
		$\widetilde A_{ij}\mapsto\lambda_i\lambda_j^{-1}\widetilde A_{ij}$
		and $\widetilde A_{ji}\mapsto\lambda_j\lambda_i^{-1}\widetilde A_{ji}$.
		Their product $\widetilde\Theta_{ij}$ is therefore unchanged.
		However, changing the connection or the maps $\sigma_i$ can change
		the resulting field. These are connection- and selector-dependent
		objects, not intrinsic invariants of the coincidence correspondence.
	\end{remark}
	
	\subsection{Algebraic Connections}
	
	The construction can be associated with algebraic connection data,
	but one must distinguish the algebraic connection from its parallel
	transport.
	
	\begin{remark}[Source of Variation in the Cartan Field]
	For fixed wall points $q_i,q_j$, Proposition~\ref{prop:two-point-comparison-point-independence}
	shows that $\widetilde A_{ij}(x;q_i,q_j)$ is independent of the
	comparison point.
	Consequently, a nonconstant field obtained from
	$\widetilde A_{ij}^{\,\sigma}(x)$ varies through the selected
	wall-point maps $q_i=\sigma_i(x)$ and $q_j=\sigma_j(x)$ (and through
	any change of the connection), not through the choice of comparison
	point with the base points held fixed.
\end{remark}

\begin{remark}[Algebraic-Geometric Formulation]
		\label{rem:algebraic-connection}
		Let $X$ be a smooth complex algebraic variety equipped with an algebraic
		connection on $T_X$. If the connection is flat, its analytification is
		a holomorphic flat connection, and the preceding construction applies
		on simply connected analytic open subsets of $X^{\mathrm{an}}$.
		The resulting parallel transports and Cartan kernels are then
		holomorphic-analytic objects. An algebraic connection does not, by
		itself, guarantee that its horizontal extensions or parallel
		transports are algebraic morphisms or algebraic sections. Accordingly,
		this appendix gives an analytic construction from algebraic connection
		data; a wholly algebraic version requires additional hypotheses that
		ensure algebraicity of the relevant horizontal data.
	\end{remark}
	
	\subsection{A Family of Two-Point Cartan Kernels}
	\label{subsec:two-point-cartan-family}
	
	We now give an explicit finite morphism for which the two-point kernel
	and the wall-map-induced invariant can be computed. The variation
	comes from transporting root data based at different points of the
	walls; it is not variation of the intrinsic same-point invariant.
	
	Work over $\mathbb C$. Let $F,G\in\mathbb C[t]$ be nonconstant
	polynomials with $F(0)G(0)\ne0$, and define
	\[
	Q:\mathbb A^2\longrightarrow\mathbb A^2,\qquad
	Q(x,y)=\bigl(x^2F(y),\,y^2G(x)\bigr).
	\]
	This morphism is finite. To see this, suppose a sequence $(x_n,y_n)$
	leaves every compact subset of $\mathbb C^2$ while $Q(x_n,y_n)$ remains
	bounded. After passing to a subsequence, either $|x_n|\to\infty$ or
	$|y_n|\to\infty$. If $|x_n|\to\infty$, nonconstancy of $G$ and
	boundedness of $y_n^2G(x_n)$ imply $y_n\to0$. Since $F(0)\ne0$, it
	follows that $|x_n^2F(y_n)|\to\infty$, a contradiction. The case
	$|y_n|\to\infty$ is symmetric, using $G(0)\ne0$. Thus the
	analytification of $Q$ is proper. For a finite-type morphism of complex
	algebraic varieties, properness of the analytification is equivalent to
		algebraic properness. Since a morphism between affine varieties is
		affine, the standard proper-affine criterion implies that $Q$ is finite.
	It is also dominant: its Jacobian determinant is
		\[
		J_Q(x,y)=4xyF(y)G(x)-x^2y^2F'(y)G'(x),
		\]
	whose lowest-degree term at the origin is
	$4F(0)G(0)xy\ne0$. Thus the Jacobian is not identically zero, and
	the differential has full rank on a nonempty Zariski-open subset.
	The morphism is therefore dominant. Hence $Q$ is finite and dominant,
		as required in the main framework.
	Equip $X=\mathbb A^2$ with its standard flat affine connection.
	
	\subsubsection*{Coincidence Branches and Transport Walls}
	
	The coincidence correspondence is given by
	\[
	x^2F(y)=\widetilde x^{\,2}F(\widetilde y),
	\qquad
	y^2G(x)=\widetilde y^{\,2}G(\widetilde x).
	\]
	Consider a point $p_1=(0,y_0,0,y_0)$ on the diagonal, where
	$y_0\ne0$ and $F(y_0)\ne0$. Near this point, choose the holomorphic
	square root of $F(y)/F(\widetilde y)$ whose value at $p_1$ is $1$.
	The first coincidence equation then has the two local forms
	$\widetilde x=\pm x\bigl(F(y)/F(\widetilde y)\bigr)^{1/2}$.
	After substitution into the second equation, its derivative with
	respect to $\widetilde y$ at $p_1$ is $2y_0G(0)\ne0$. The implicit
	function theorem therefore gives two local graph branches. The plus
	branch is the diagonal branch; the minus branch is a non-diagonal
	transport branch, denoted locally by $D_1$. Its fixed-point wall is
	the germ of $H_1=\{x=0\}$.
	
	By the symmetric argument, at a point
	$p_2=(x_0,0,x_0,0)$ with $x_0\ne0$ and $G(x_0)\ne0$, there is a
	non-diagonal local transport branch $D_2$ with fixed-point wall the
	germ of $H_2=\{y=0\}$. Thus we work on the regular wall portions
	\[
	H_1^\circ=\{(0,y):y\ne0,\ F(y)\ne0\},\qquad
	H_2^\circ=\{(x,0):x\ne0,\ G(x)\ne0\}.
	\]
	The corresponding conormal covectors are $\alpha_1=dx$ and
	$\alpha_2=dy$. The nonzero conditions above ensure that the local graph
	branches are defined at the chosen wall points.
	
	\subsubsection*{Differentials Along the Walls}
	
	Let $q_1=(0,y_1)\in H_1^\circ$, and write
	$\tau_1(x,y)=(\widetilde x,\widetilde y)$ for the selected local
	transport map. Along $H_1$ it fixes every point. The chosen minus
	branch gives $\partial\widetilde x/\partial x=-1$ at $q_1$, and
	$\partial\widetilde x/\partial y=0$. Differentiating
	$y^2G(x)=\widetilde y^{\,2}G(\widetilde x)$ with respect to $x$ and
	then restricting to $H_1$ gives
	\[
	y_1^2G'(0)
	=2y_1G(0)\frac{\partial\widetilde y}{\partial x}
	-y_1^2G'(0),
	\]
	so $\partial\widetilde y/\partial x
	=y_1G'(0)/G(0)$. Since $\tau_1$ restricts to the identity on $H_1$,
	$\partial\widetilde y/\partial y=1$. Hence
	\[
	R_1(q_1)=d\tau_1|_{q_1}
	=\begin{pmatrix}
		-1&0\\[1mm]
		y_1G'(0)/G(0)&1
	\end{pmatrix}.
	\]
	
	Similarly, for $q_2=(x_2,0)\in H_2^\circ$, the selected local map
	$\tau_2(x,y)=(\widehat x,\widehat y)$ has differential
	\[
	R_2(q_2)=d\tau_2|_{q_2}
	=\begin{pmatrix}
		1&x_2F'(0)/F(0)\\[1mm]
		0&-1
	\end{pmatrix}.
	\]
	
	\subsubsection*{Normalized Transverse Vectors}
	
	The normalized $(-1)$-eigenvectors of these differentials are
	\[
	\nu_1(q_1)
	=\frac{\partial}{\partial x}
	-\frac12\frac{G'(0)}{G(0)}y_1\frac{\partial}{\partial y},
	\qquad
	\nu_2(q_2)
	=-\frac12\frac{F'(0)}{F(0)}x_2\frac{\partial}{\partial x}
	+\frac{\partial}{\partial y}.
	\]
	They satisfy $dx(\nu_1(q_1))=1$ and $dy(\nu_2(q_2))=1$, respectively.
	In particular, the transverse image lines complement the tangent lines
	to the corresponding walls at these points, as required for the
	normalization.
	
	\subsubsection*{The Two-Point Cartan Kernel}
	
	Let $p=(X,Y)\in\mathbb A^2$ be an ambient point. For the standard flat
	connection, the coordinate vector fields and covector fields are
	parallel. Thus the transported data are
	\[
	\begin{aligned}
		\widetilde\alpha_1(p;q_1)&=dx,&
		\widetilde\alpha_2(p;q_2)&=dy,\\
		\widetilde\nu_1(p;q_1)&=\frac{\partial}{\partial x}
		-\frac12\frac{G'(0)}{G(0)}y_1\frac{\partial}{\partial y},&
		\widetilde\nu_2(p;q_2)&=-\frac12\frac{F'(0)}{F(0)}x_2\frac{\partial}{\partial x}
		+\frac{\partial}{\partial y}.
	\end{aligned}
	\]
	Consequently,
	\[
	\widetilde A_{12}(p;q_1,q_2)
	=-\frac12\frac{G'(0)}{G(0)}y_1,
	\qquad
	\widetilde A_{21}(p;q_2,q_1)
	=-\frac12\frac{F'(0)}{F(0)}x_2,
	\]
	and
	\[
	\widetilde A(p;q_1,q_2)
	=\begin{pmatrix}
		1&-\dfrac12\dfrac{G'(0)}{G(0)}y_1\\[2mm]
		-\dfrac12\dfrac{F'(0)}{F(0)}x_2&1
	\end{pmatrix}.
	\]
	The associated two-point invariant is
	\[
	\widetilde\Theta_{12}(p;q_1,q_2)
	=\frac14\frac{F'(0)G'(0)}{F(0)G(0)}x_2y_1.
	\]
	It is independent of the ambient point $p$ for this flat connection,
	but depends on the two selected wall points. With the sign convention
	fixed above, the two-point transport Cartan matrix is
	\[
	\widetilde{\mathcal C}(p;q_1,q_2)
	=\begin{pmatrix}
		2&-\dfrac{G'(0)}{G(0)}y_1\\[2mm]
		-\dfrac{F'(0)}{F(0)}x_2&2
	\end{pmatrix}.
	\]
	
	\subsubsection*{A Wall-Map-Induced Cartan Field}
	
	On the open subset
	\[
	V=\{(X,Y)\in\mathbb A^2:
	XYF(Y)G(X)\ne0\},
	\]
	choose the wall-point maps
	$\sigma_1(X,Y)=(0,Y)$ and $\sigma_2(X,Y)=(X,0)$.
	These maps take values in $H_1^\circ$ and $H_2^\circ$, respectively.
	They are wall-point selectors, not sections of the inclusions of the
	walls into $X$. Substitution into the two-point kernel gives
	\[
	\widetilde A^{\,\sigma}(X,Y)
	=\begin{pmatrix}
		1&-\dfrac12\dfrac{G'(0)}{G(0)}Y\\[2mm]
		-\dfrac12\dfrac{F'(0)}{F(0)}X&1
	\end{pmatrix},
	\qquad
	\widetilde\Theta^{\,\sigma}_{12}(X,Y)
	=\frac14\frac{F'(0)G'(0)}{F(0)G(0)}XY.
	\]
	The corresponding transport Cartan matrix field is
	\[
	\widetilde{\mathcal C}^{\,\sigma}(X,Y)
	=\begin{pmatrix}
		2&-\dfrac{G'(0)}{G(0)}Y\\[2mm]
		-\dfrac{F'(0)}{F(0)}X&2
	\end{pmatrix}.
	\]
	If $F'(0)G'(0)\ne0$, the wall-map-induced invariant is nonconstant
	on $V$. The displayed polynomial expression extends to all of
	$\mathbb A^2$, but the geometric construction by these selected local
	branches is asserted only on $V$.
	
	For the special choice $F(t)=G(t)=1+t$, the domain is
	$V=\{(X,Y):XY(1+X)(1+Y)\ne0\}$, and
	\[
	\widetilde A^{\,\sigma}(X,Y)
	=\begin{pmatrix}1&-Y/2\\-X/2&1\end{pmatrix},
	\qquad
	\widetilde{\mathcal C}^{\,\sigma}(X,Y)
	=\begin{pmatrix}2&-Y\\-X&2\end{pmatrix},
	\qquad
	\widetilde\Theta^{\,\sigma}_{12}(X,Y)=\frac{XY}{4}.
	\]
	
	\begin{remark}[Scope of the Two-Point Construction]
		The nonconstant quantity $\widetilde\Theta^{\,\sigma}_{12}$ compares
		root data based at $q_1=\sigma_1(p)$ and $q_2=\sigma_2(p)$, which are
		generally different points. It is therefore a connection-dependent
		two-point invariant. It is not the intrinsic invariant $\Theta_{12}$
		formed from two transport branches at one common point, and no product
		of two local symmetry germs at a common point is being asserted here.
		The example consequently does not contradict the local finite-group
		or finite-order restrictions proved for the intrinsic same-point
		theory. It shows instead that a chosen connection and chosen
		wall-point maps can extend root data to a variable matrix field even
		when the underlying correspondence comes from a single finite
		morphism.
	\end{remark}

	\bibliographystyle{unsrt}

\begin{thebibliography}{10}
	
	\bibitem{Coxeter1934}
	H.~S.~M. Coxeter.
	\newblock Discrete groups generated by reflections.
	\newblock {\em Ann. of Math. (2)}, 35(3):588--621, 1934.
	
	\bibitem{MR376938}
	V.~S. Varadarajan.
	\newblock {\em Lie groups, {L}ie algebras, and their representations}.
	\newblock Prentice-Hall Series in Modern Analysis. Prentice-Hall, Inc.,
	Englewood Cliffs, NJ, 1974.
	
	\bibitem{MR323842}
	James~E. Humphreys.
	\newblock {\em Introduction to {L}ie algebras and representation theory},
	volume~9 of {\em Graduate Texts in Mathematics}.
	\newblock Springer-Verlag, New York-Berlin, 1972.
	
	\bibitem{ShephardTodd1954}
	Geoffrey~C Shephard and John~A Todd.
	\newblock Finite unitary reflection groups.
	\newblock {\em Canadian Journal of Mathematics}, 6:274--304, 1954.
	
	\bibitem{Chevalley1955}
	Claude Chevalley.
	\newblock Invariants of finite groups generated by reflections.
	\newblock {\em American Journal of Mathematics}, 77(4):778--782, 1955.
	
	\bibitem{humphreys1990reflection}
	James~E. Humphreys.
	\newblock {\em Reflection groups and {C}oxeter groups}, volume~29 of {\em
		Cambridge Studies in Advanced Mathematics}.
	\newblock Cambridge University Press, Cambridge, 1990.
	
	\bibitem{Heckenberger2006}
	Istv{\'a}n Heckenberger.
	\newblock The {W}eyl groupoid of a {N}ichols algebra of diagonal type.
	\newblock {\em Inventiones mathematicae}, 164(1):175--188, 2006.
	
	\bibitem{cuntz2012finite}
	Michael Cuntz and Istv{\'a}n Heckenberger.
	\newblock Finite {W}eyl groupoids of rank three.
	\newblock {\em Transactions of the American Mathematical Society},
	364(3):1369--1393, 2012.
	
	\bibitem{OrlikTerao1992}
	Peter Orlik and Hiroaki Terao.
	\newblock {\em Arrangements of hyperplanes}, volume 300 of {\em Grundlehren der
		mathematischen Wissenschaften [Fundamental Principles of Mathematical
		Sciences]}.
	\newblock Springer-Verlag, Berlin, 1992.
	
	\bibitem{Looijenga1981}
	Eduard Looijenga.
	\newblock Rational surfaces with an anti-canonical cycle.
	\newblock {\em Annals of Mathematics}, 114(2):267--322, 1981.
	
	\bibitem{mckay1980graphs}
	John McKay.
	\newblock Graphs, singularities, and finite groups.
	\newblock In {\em The {S}anta {C}ruz {C}onference on {F}inite {G}roups ({U}niv.
		{C}alifornia, {S}anta {C}ruz, {C}alif., 1979)}, volume~37 of {\em Proc.
		Sympos. Pure Math.}, pages 183--186. Amer. Math. Soc., Providence, RI, 1980.
	
	\bibitem{mckay1981cartan}
	John McKay.
	\newblock Cartan matrices, finite groups of quaternions, and {K}leinian
	singularities.
	\newblock {\em Proceedings of the American Mathematical Society},
	81(1):153--154, 1981.
	
	\bibitem{MR354651}
	{\em Rev\^{e}tements \'{e}tales et groupe fondamental}.
	\newblock Lecture Notes in Mathematics, Vol. 224. Springer-Verlag, Berlin-New
	York, 1971.
	\newblock S\'{e}minaire de G\'{e}om\'{e}trie Alg\'{e}brique du Bois Marie
	1960--1961 (SGA 1), Dirig\'{e} par Alexandre Grothendieck. Augment\'{e} de
	deux expos\'{e}s de M. Raynaud.
	
	\bibitem{Takamura2019Linearization}
	Shigeru Takamura.
	\newblock Linearization of quotient families.
	\newblock {\em J. Math. Sci. Univ. Tokyo}, 26(3):361--389, 2019.
	
\end{thebibliography}

	\end{document}